\documentclass[journal,twoside,web]{ieeecolor}
\usepackage{generic}
\usepackage{cite}
\usepackage{amsmath,amssymb,amsfonts}
\usepackage{algorithmic}
\usepackage{graphicx}
\usepackage{algorithm}
\usepackage{textcomp}
\usepackage{booktabs} % 三线表
\usepackage{tabularx}    % 用于自动调整列宽
\usepackage{multirow} % 表格多行合并
\usepackage{pifont}
\usepackage{caption}  % 自定义标题格式
\usepackage[labelfont=rm]{subcaption} % 替代 subfig，支持子图
\usepackage{float}
\usepackage[justification=centering]{caption} % 标题居中
\usepackage{soul} % 高亮
\usepackage{ulem} % 删除线
\usepackage{bm}
\usepackage{color}
\usepackage{tcolorbox}
\usepackage{stmaryrd} % 三条竖线的范数
\usepackage{array}
\newcolumntype{M}[1]{>{\centering\arraybackslash}m{#1}}
\usepackage{bbding,pifont} % 引入对勾和叉

\newtheorem{theorem}{Theorem}
\newtheorem{lemma}{Lemma}

\newtheorem{remark}{Remark}

\newtheorem{assumption}{Assumption}
\newtheorem{definition}{Definition}
\newtheorem{proposition}{Proposition}

\def\BibTeX{{\rm B\kern-.05em{\sc i\kern-.025em b}\kern-.08em
    T\kern-.1667em\lower.7ex\hbox{E}\kern-.125emX}}

\usepackage{hyperref}
\hypersetup{hidelinks=true}

\begin{document}
\title{Enhancing Accuracy in Differentially Private Distributed Optimization Through Sensitivity Reduction}
\author{Furan Xie, Bing Liu and Li Chai% <-this % stops a space
\thanks{This work was supported in part by the National Natural Science Foundation of China under Grant 62573332 and Grant U2441244, and in part by the Zhejiang Provincial Natural Science Foundation of China under Grant LZ24F030006. (\textit{Corresponding author: Bing Liu and Li Chai}).}
% <-this % stops a space
\thanks{Furan Xie and Bing Liu are with the Engineering Research Center of Metallurgical Automation and Measurement Technology, Wuhan University of Science and Technology, Wuhan 430081, China (e-mail: xiefuran328@qq.com; liubing17@wust.edu.cn).}% <-this % stops a space
\thanks{Li Chai is with the State Key Laboratory of Industrial Control Technology, College of Control Science and Engineering, Zhejiang University, Hangzhou 310027, China (e-mail: chaili@zju.edu.cn).}}
\maketitle

\begin{abstract}
In this paper, we investigate the problem of differentially private distributed optimization. Recognizing that lower sensitivity leads to higher accuracy, we analyze the key factors influencing the sensitivity of differentially private distributed algorithms. Building on these insights, we propose a novel differentially private distributed algorithm for undirected graphs that enhances optimization accuracy by reducing sensitivity. To ensure practical applicability, we derive an explicit closed-form expression for the noise parameter as a function of the privacy budget. Moreover, we rigorously prove that the proposed algorithm can achieve arbitrarily rigorous $\epsilon$-differential privacy, establish its convergence in the mean square sense, and provide an upper bound on its optimization accuracy. Finally, extensive comparisons with various privacy-preserving methods validate the effectiveness of our algorithm.
\end{abstract}
%In addition, we present a general privacy measure to quantify the privacy levels of various types of privacy-preserving algorithms. Finally, numerical simulations validate the effectiveness of the proposed algorithm and privacy measure.
%By leveraging the inverse relationship between the algorithm's sensitivity and accuracy, we enhance the algorithm's accuracy by reducing its sensitivity. To this end, we analyze the key factors influencing sensitivity and use them as guiding principles for algorithm design.
%We establish the algorithm's convergence under mild assumptions and give an upper bound on its optimization accuracy.
%For practical purposes, we derive a closed-form expression for the noise parameter with respect to the privacy budget. 
%Furthermore, we show that the proposed algorithm can achieve arbitrarily rigorous $\epsilon$-differential privacy under appropriate parameters, and its mean-square convergence under mild assumptions. 
%By injecting random noise into the transmitted messages, the proposed algorithm can achieve arbitrarily rigorous $\epsilon$-differential privacy. We also derive a closed-form expression for the noise parameter with respect to the privacy budget.
\section{Introduction}
\IEEEPARstart{W}{ith} the widespread application of distributed optimization in various fields such as machine learning, power systems, and robotic networks, the growing sophistication of cyberattacks has brought security and privacy concerns to the forefront of research \cite{wang2024privacy}. Compared to traditional centralized approaches, distributed optimization methods offer better privacy protection by avoiding the direct exchange of raw data, which may contain sensitive information. Nevertheless, studies have shown that the exchange of states and (estimated) gradients between agents in distributed optimization can still pose significant privacy risks. For instance, in distributed machine learning, attackers can accurately recover an agent's raw data through its gradient \cite{zhu2019deep}. In the distributed rendezvous problem, an agent's state itself is its sensitive information \cite{mo2016privacy}. Therefore, how to design distributed optimization methods that can protect the sensitive information of agents has become an urgent challenge.
%Compared to traditional centralized optimization approaches, distributed optimization methods are typically considered privacy-preserving, as they do not require the transmission of raw data of agents, which may contain sensitive information.
%still poses potential risks of privacy leakage.
%as they do not transmit the raw data of agents. The raw data often contains sensitive information about the agents.
%along with the advancement of attack techniques, security and privacy protection in distributed optimization have attracted significant attention
%exploit the robustness of distributed optimization dynamics
%The advantage of this type of method is to maintain the convergence rate of the original algorithm while achieving exact convergence. 

In general, privacy-preserving distributed optimization algorithms can be categorized into three types: correlated randomness-based methods, encryption-based methods, and differential privacy-based methods. 
%Correlated randomness-based methods achieve privacy protection by introducing temporally or spatially correlated randomness into optimization parameters \cite{cheng2024privacy,gade2018private,huan2023dynamics,li2020privacypreserving}. For example, \cite{huan2023dynamics} preserves the privacy of agents' gradients by adding randomness to the stepsize and coupling weights during the initial iterations. By leveraging the robustness of dynamics, these methods can attain exact convergence. However, they necessitate an additional topological assumption to ensure privacy protection and cannot defend against external eavesdroppers. 

Correlated randomness-based methods achieve privacy protection by introducing temporally or spatially correlated randomness into the optimization process. \cite{gade2018private} and \cite{Han2022PrivacyPreserving} protect agents' private information by injecting structured random noise adhering to a zero-sum structure into the interactive messages between agents. \cite{Zhang2025Privacy} further strengthens privacy by simultaneously introducing functional and state perturbations that also satisfy a zero-sum constraint. \cite{li2020privacypreserving} safeguards agents' privacy by injecting random noise into the dual variables directed towards a non-convergent subspace. \cite{huan2023dynamics} protects the privacy of agents' gradients by adding randomness to the stepsizes and coupling weights during the initial iterations. \cite{cheng2024privacy} and \cite{Lu2025Sdppaad}, on the other hand, utilize state decomposition techniques to split the local state into public and private substates, transmitting only the substate unrelated to private information during communication, thereby achieving privacy preservation. Leveraging the robustness of the system dynamics, these methods can attain exact convergence. However, they require an additional topological assumption to ensure privacy protection and cannot defend against external eavesdroppers.

In contrast, encryption-based methods can effectively defend against external eavesdroppers by encrypting the transmitted information. \cite{zhang2019admm} and \cite{zhang2019enabling} propose two privacy-preserving distributed optimization algorithms by integrating the Alternating Direction Method of Multipliers (ADMM) and Distributed Gradient Descent (DGD) methods with Homomorphic Encryption (HE). Considering the significant computational and communication overhead associated with HE, several efficiency optimization techniques can be adopted to mitigate these burdens. \cite{Mia2025QuanCrypt} combines quantization techniques with HE, significantly reducing the size of data to be encrypted and transmitted by converting high-precision gradients into low-bit integers. \cite{han2024Adaptive} integrates packing (batching) techniques with HE, reducing communication frequency by aggregating multiple data elements into a single plaintext or ciphertext before encryption and transmission. Moreover, replacing HE with lightweight cryptographic primitives, such as the Advanced Encryption Standard (AES) in \cite{Liu2024CryptographyBasedPM} and functional encryption in \cite{Yu2025Lightweight}, can reduce the overall computational and communication burden. Despite these advances, compared to other types of privacy-preserving methods, the encryption and decryption operations in encryption-based approaches still introduce additional computational or communication overhead.

Due to the wide applicability, simplicity, and strong theoretical privacy guarantees, Differential Privacy (DP) has emerged as the de facto standard for privacy protection in recent years \cite{han2017Differentially}. Several studies have integrated DP into distributed optimization by injecting random noise into the information transmitted between agents. \cite{huang2015differentially} proposes a differentially private algorithm based on DGD, under the assumption that the gradients of all individual cost functions are bounded. However, since DGD relies on inexact gradients, its convergence rate remains sublinear, even without the addition of DP noise. In contrast, the Gradient Tracking (GT) algorithms can achieve linear convergence with a constant stepsize \cite{Qu2018Harnessing}. To this end, \cite{ding2022differentially} presents a differentially private GT algorithm that achieves linear convergence over undirected graphs. This result is later extended to time-varying directed graphs by \cite{Yang2025Differentially}. Both works, however, rely on the assumption that different objective functions in two adjacent problems share the same Lipschitz constant and convexity parameter. On the other hand, \cite{huang2015differentially}, \cite{ding2022differentially}, and \cite{Yang2025Differentially} all reveal the inherent trade-off between privacy and accuracy in differentially private algorithms.

%To overcome this challenge, \cite{wang2024tailoring} proposes a distributed algorithm based on gradient tracking that ensures both DP and optimality by incorporating decaying weakening factors.

To optimize the privacy-accuracy trade-off in differentially private algorithms, various techniques have been proposed. \cite{wang2024tailoring} presents a GT-based distributed algorithm that ensures both DP and optimality by incorporating decaying weakening factors into weight matrices. However, to achieve rigorous DP, this method requires the assumption that adjacent objective functions have consistent gradients within the neighborhood of the optimal solution. In addition, in the presence of persistent DP noise, GT techniques suffer from noise accumulation, which negatively impacts the algorithm's accuracy. To mitigate this issue, \cite{yu2023gradienttracking} proposes a differentially private distributed algorithm based on Robust Gradient Tracking (RGT) for undirected graphs. \cite{huang2024differential} later extends this approach to directed graphs. \cite{Wang2023GradientTrackingBased} proposes a distributed algorithm that integrates RGT with decaying weakening factors, ensuring almost sure convergence in the presence of information-sharing noise. As an enhancement, \cite{Zhao2025VarianceReduced} further incorporates a variance reduction technique into this framework, which reduces the estimation error variance and relaxes the assumptions on network connectivity. However, whether these two methods can provide rigorous DP guarantees remains unclear.

Apart from the aforementioned research focused on suppressing or reducing the impact of noise, another line of work aims to optimize the noise itself, with the core idea being to dynamically adjust the scale and distribution of noise based on the characteristics of the model or data. The adaptive noise mechanism is a typical representative of this direction. For instance, in the federated learning setting, \cite{Fu2022AdapDPFL} proposes a method combining adaptive gradient clipping and adaptive noise scaling, which better aligns the noise with the training phase and client heterogeneity. Furthermore, \cite{Xue2024DifferentiallyPrivate} introduces a component-wise adaptive noise mechanism that takes into account the heterogeneity in parameter magnitudes. By estimating the sensitivity of each parameter component and adding Gaussian noise scaled accordingly, this approach prevents excessive distortion to parameters with smaller magnitudes. On the other hand, from the perspective of system and communication architecture, introducing the shuffle model to achieve ``privacy amplification'' has emerged as another important approach to enhance privacy. This model incorporates an anonymization step: users first perturb their data locally, after which a trusted shuffler randomly permutes all user messages before sending them to the analyzer. This process severs the association between outputs and their senders, thereby providing stronger central DP guarantees for aggregated results compared to pure local perturbation \cite{girgis2021shuffled, liu2021flame, Liu2023Echo}. However, how to apply adaptive noise mechanisms and privacy amplification strategies to general differentially private decentralized optimization remains an open problem requiring in-depth research.

%\cite{yu2023gradienttracking} and \cite{huang2024differential} also reveal the inherent trade-off between privacy and accuracy in differentially private algorithms. Furthermore, \cite{huang2024differential} relaxes the assumption that the gradients of all individual cost functions are bounded, instead assuming that the distance between the gradients of different objective functions in adjacent problems is bounded.

% However, its applicability is restricted to undirected graphs due to the use of doubly stochastic weight matrices. To address this limitation, \cite{huang2024differential} employs a row and a column stochastic weight matrices to develop a robust gradient tracking-based differentially private algorithm that extends applicability to directed graphs. 
%\cite{pu2020robust} proposes a robust gradient tracking algorithm. A differentially private algorithm based on robust gradient tracking is proposed in \cite{yu2023gradienttracking}
%posing challenges for real-world deployment

Despite significant advancements in differentially private distributed optimization, several critical challenges remain. Firstly, existing differentially private algorithms lack clear design principles to systematically improve accuracy, as well as theoretical frameworks to validate such improvements. As a result, efforts to enhance accuracy often rely on heuristic tuning rather than guided optimization, and it remains challenging to rigorously discern whether a modification truly improves performance or merely shifts the trade-off elsewhere. Secondly, while relationships between privacy and noise have been studied, there remains a lack of an explicit, closed-form expression that can directly translate a desired privacy budget into actionable noise parameters prior to execution. This makes it difficult to configure algorithms predictably and limits their practical deployment. Thirdly, the use of diminishing stepsizes introduces a fundamental analytical hurdle: the resulting state transition matrix of the algorithm no longer has a spectral radius strictly less than 1. This, combined with persistent noise injection, invalidates standard convergence analysis techniques, making it difficult to establish rigorous performance guarantees. To address these challenges, this paper proposes corresponding solutions, which, along with the contributions of this work, are summarized as follows:
%Thirdly, the combined effect of persistent noise injection and diminishing stepsizes renders many existing convergence analysis techniques inapplicable, hindering the theoretical understanding and performance characterization of the algorithm.
%Thirdly, existing convergence analysis techniques become inapplicable in the presence of noise injection and diminishing stepsize. 
%Thirdly, although the privacy budget provides a means to quantify the privacy levels of differentially private algorithms, there is no unified metric to compare different types of privacy-preserving algorithms, making objective evaluation challenging.}
%existing state-of-the-art differentially private algorithms do not establish a direct analytical relationship .. .
%Firstly, although existing state-of-the-art differentially private algorithms have made progress in improving accuracy, they lack both theoretical validation of the effectiveness of these improvements and design principles for enhancing the algorithm's accuracy.
%general privacy measure to evaluate the privacy levels of various types of privacy-preserving algorithms. This absence hinders objective comparisons among different privacy-preserving methods.
%Based on the above review of existing differentially private distributed algorithms, we identify several limitations in this field.

1) \textit{A Systematic, Sensitivity-Reduced Design for Accuracy Enhancement.} To overcome the lack of design principles, we first reveal that algorithm sensitivity is the key determinant of accuracy under a fixed privacy budget. This novel perspective enables the systematic derivation of design principles for sensitivity reduction (Proposition \ref{prop:sensitivity}). Implementing these principles, we propose an implicit gradient tracking-based differentially private distributed algorithm for undirected graphs that inherently achieves lower sensitivity. Both theoretical analysis and experiments confirm that the proposed algorithm achieves higher accuracy under the same privacy budget compared to existing methods in \cite{ding2022differentially, Yang2025Differentially, yu2023gradienttracking, wang2024tailoring, huang2024differential}.

2) \textit{An Explicit, Direct Noise Calibration Rule and Halved Overhead.} To address the impracticality of implicit noise-privacy relationships, we derive an explicit, directly usable closed-form expression for the noise parameter as a function of the privacy budget. This expression serves as a direct design tool, enabling the exact computation of noise parameters from any given privacy budget before execution. Additionally, the proposed algorithm halves both the number of communication variables and noise injections compared to existing methods in \cite{ding2022differentially, Yang2025Differentially, yu2023gradienttracking, wang2024tailoring, huang2024differential}, significantly reducing communication and computational overhead.

3) \textit{Theoretical Analysis of $\epsilon$-DP and Convergence.} We prove that the proposed algorithm with appropriate parameters can achieve rigorous $\epsilon$-DP under a weak gradient dissimilarity condition (Theorem \ref{lemma_sensitivity}). Furthermore, to address the analytical difficulty arising from diminishing stepsizes and persistent noise, specifically, the resulting non-contractive dynamics, we develop an analysis framework to establish mean-square convergence and derive an upper bound on the optimization accuracy (Theorem \ref{theorem:convergence}). This bound explicitly quantifies the privacy-accuracy trade-off. Based on the theoretical results, we provide two practical parameter tuning strategies.

%Under a weak assumption on the gradients of different objective functions in adjacent problems, we prove that the proposed algorithm with appropriate parameters can achieve rigorous $\epsilon$-DP (Theorem \ref{lemma_sensitivity}). Additionally, under the assumption that the objective function is Lipschitz smooth and strongly convex, we establish the convergence of the proposed algorithm with diminishing stepsize and noise parameter, as well as an upper bound for its optimization accuracy (Theorem \ref{theorem:convergence}). The theoretical results reveal the trade-off between the algorithm's privacy and accuracy. Based on these results, we further provide two tuning strategies for the algorithm's parameters.

%\textcolor{red}{4) \textit{Definition of a general privacy measure:} Given that existing privacy measures cannot evaluate the privacy levels across different types of privacy-preserving algorithms, we propose a general privacy measure based on mutual information from information theory, termed Maximization Normalized Mutual Information (M-NMI), to quantify the privacy levels of various privacy-preserving distributed algorithms. Moreover, we demonstrate how to compute the M-NMI under both honest-but-curious adversaries and external eavesdroppers. Through numerical simulations, we validate the effectiveness and applicability of the proposed privacy measure.}

The remainder of this paper is organized as follows. Section \ref{sec:problem} presents the problem formulation and provides relevant definitions of DP. Section \ref{sec: algorithm} proposes a differentially private distributed algorithm with lower sensitivity. Sections \ref{section_DP} and \ref{section:convergence} analyze the DP guarantees and convergence of the proposed algorithm, respectively. Section \ref{section_simulation} validates the effectiveness of the proposed algorithm through numerical simulations. Finally, Section \ref{section_conclusion} concludes the paper.
%\textcolor{red}{Section \ref{section_measure} presents a general privacy measure.}
%Contributions:
%\begin{enumerate}
%  \item \textit{$\epsilon$-differentilly privacy}
%  \item \textit{Analytical solution}
%  \item \textit{Lower sensitivity}
%  \item \textit{Analytical expression for the variance of injected noise:}
%  \item \textit{Enhancement of accuracy:} The proposed differentially private algorithm has lower sensitivity.
%  \item \textit{Reduction of computational and communication overheads:} In the proposed differentially private algorithm, each agent needs to share only one optimization variable and accordingly, only one noise vector needs to be injected.
%\end{enumerate}

\textit{Notation:} Table~\ref{tab:notation} lists the key notations used in this paper; others are defined in their respective sections.

\begin{table}[!ht]
\centering
\small % 使用小号字体
\caption{Summary of Key Notations}
\label{tab:notation}
\begin{tabularx}{\linewidth}{lX} % 使用tabularx环境，第二列自动调整宽度
\toprule
\multicolumn{1}{l}{\textbf{Symbol}} & \multicolumn{1}{l}{\textbf{Meaning}} \\
\midrule
\multicolumn{2}{l}{\textbf{A. Key Problem and Algorithmic Parameters}} \\
\midrule
$n$ & Number of agents \\
$p$ & Dimension of the decision variable \\
%$\mathcal{G} = (\mathcal{N}, \mathcal{E})$ & Communication graph (node set $\mathcal{N}$ and edge set $\mathcal{E}$) \\
%$W$ & Symmetric, doubly-stochastic weight matrix \\
$f_i$, $f$ & Local and global cost functions \\
$\bm{x}_i$, $\mathbf{x}$ & Local variable of agent $i$ and stack matrix  of all $\bm{x}_i$ \\
$\bm{x}^*$ & Optimal solution to the global problem \\
%$\nabla F(\mathbf{x})$ & Stack matrix of all local gradients $[\nabla f_1(\bm{x}_1), \cdots, \nabla f_n(\bm{x}_n)]^T$ \\
$L$, $\mu$ & Lipschitz smoothness and strong convexity constants \\
%$q_1$ & Stepsize decay parameter ($q_1 \in (0,1)$) \\
$\bm{y}_i$ & Auxiliary variable for agent $i$ (not communicated) \\
$\bm{z}_i$ & Noisy state of agent $i$ (shared with neighbors) \\
$\bm{\xi}_i$ & Laplace noise added by agent $i$ \\
$\alpha_k$ & Stepsize at iteration $k$ \\
$\beta$ & Gain parameter for the auxiliary variable \\
$\nu_k$ & Parameter of the Laplace noise at iteration $k$ \\
$\epsilon$ & Privacy budget \\
$\delta$ & Bound on gradient differences for adjacent problems \\
$\Delta(k)$ & Sensitivity of the algorithm at iteration $k$ \\
\midrule
\multicolumn{2}{l}{\textbf{B. General Mathematical Symbols}} \\
\midrule
$\|\bm{v}\|_a$ & Standard $L^a$-norm of vector $\bm{v}$ \\
$\|\bm{v}\|$ & $L^2$-norm of vector $\bm{v}$ \\
$\|M\|_{a}$ & Standard induced $L^a$-norm of matrix $M$ \\
$\|M\|$ & Frobenius norm of matrix $M$ \\
%$\|\boldsymbol{v}\|_a$, $\|M\|_a$ & Standard $L^a$-norm of vector $\boldsymbol{v}$ and induced norm of matrix $M$ \\
%$\|\boldsymbol{v}\|$, $\|M\|$ & $L^2$-norm of vector $\boldsymbol{v}$, Frobenius norm of matrix $M$ \\
$\mathbf{0}$, $\mathbf{1}$ & All-zeros and all-ones vectors \\
$I$ & Identity matrix \\
$\rho(M)$ & Spectral radius of square matrix $M$ \\
%$\prod_{k=K_1}^{K_2} M_k$ & Product of a sequence of matrices \\
$\mathbb{P}[\cdot]$, $\mathbb{E}[\cdot]$ & Probability and expectation operators \\
$Lap(c)$ & Laplace distribution with scale parameter $c$ \\
\bottomrule
\end{tabularx}
\color{black}
\end{table}

\section{Problem Formulation}\label{sec:problem}
Consider a multi-agent system consisting of $n$ agents. Each agent $i$ possesses an individual cost function $f_{i}: \mathbb{R}^{p} \to \mathbb{R}$, which is only known by agent $i$ itself. All agents perform local computation and communication over an undirected network to collectively minimize:
\begin{equation}\label{problem}
	f(\bm{x}) \triangleq \frac{1}{n}\sum_{i=1}^{n} f_{i}(\bm{x}),
\end{equation}
where $\bm{x}\in\mathbb{R}^{p}$ is the global decision variable and $f: \mathbb{R}^{p} \to \mathbb{R}$ is the average of all cost functions. This paper aims to protect the privacy of each agent's local cost function $f_i$, which constitutes sensitive information. We consider an adversary with arbitrary computational power that can observe all shared information among all agents. This adversary may be an external eavesdropper or any coalition of internal honest-but-curious agents, with knowledge of the network topology, the algorithm's parameters and rules, and the cost functions of all non-target agents. The privacy objective is to prevent this adversary from inferring the private cost function $f_{i}$ of any individual agent. Through DP, we ensure the adversary cannot reliably determine the identity or exact form of any agent's cost function from the algorithm's outputs. Let $\bm{x}_{i}$ be a local copy of the decision variable for agent $i$. Denote the gradient of $f_{i}$ with respect to $\bm{x}_{i}$ as $\nabla f_{i}(\bm{x}_{i})$. For convenience, we further define $\mathbf{x} \triangleq \left[ \bm{x}_{1},\cdots,\bm{x}_{n} \right]^{T} \in \mathbb{R}^{n\times p}$ and $\nabla F(\mathbf{x}) \triangleq \left[ \nabla f_{1}(\bm{x}_{1}),\cdots,\nabla f_{n}(\bm{x}_{n}) \right]^{T} \in \mathbb{R}^{n\times p}$.

We model an undirected communication network using the graph $\mathcal{G} \triangleq (\mathcal{N}, \mathcal{E})$, where $\mathcal{N}\triangleq \{1,2,\cdots,n\}$ represents the set of agents and $\mathcal{E}\subseteq\mathcal{N}\times\mathcal{N}$ denotes the set of edges. $(i,j)\in\mathcal{E}$ if and only if agents $i$ and $j$ can send information to each other. The set of agents that can communicate with agent $i$ is defined as $\mathcal{N}_{i} \triangleq \left\{j\in\mathcal{N}: (i,j)\in\mathcal{E}\right\}$. Moreover, we employ a nonnegative and symmetric matrix $W \in \mathbb{R}^{n\times n}$ to represent the inter-agent weights. For any $i,j\in\mathcal{N}$, $W_{ij}>0$ if $i=j$ or $(i,j)\in\mathcal{E}$. For all other $(i,j)$, $W_{ij}=0$. We define $\mathcal{G}_{W}\triangleq(\mathcal{N}, \mathcal{E}_{W})$ as the graph induced by the weight matrix $W$, where $(i,j)\in\mathcal{E}_{W}$ if and only if $i\not=j$ and $W_{ij}>0$.

Throughout the paper, we make the following assumptions:
\begin{assumption}\label{assump:smooth_convex}
For each $i\in\mathcal{N}$, the individual cost function $f_{i}$ is $L$-smooth and $\mu$-strongly convex, i.e., for any $\bm{x}, \bm{x}' \in \mathbb{R}^{p}$,
\begin{gather*}
\| \nabla f_{i}(\bm{x}) - \nabla f_{i}(\bm{x}') \| \le L \| \bm{x} - \bm{x}' \|,\\
f_{i}(\bm{x}') \ge f_{i}(\bm{x}) + \nabla f_{i}(\bm{x})^{T}(\bm{x}' - \bm{x}) + \frac{\mu}{2}\| \bm{x}' - \bm{x} \|^{2}.
\end{gather*}
\end{assumption}

\begin{assumption}\label{assump:W}
The nonnegative weight matrix $W$ is symmetric and doubly-stochastic, satisfying $W = W^{T}$, $\bm{1}^{T}W = \bm{1}^{T}$, and $W\bm{1} = \bm{1}$. Its induced graph $\mathcal{G}_{W}$ is connected, that is, there exists a path between any two agents.
\end{assumption}

\begin{remark}
(On the network topology) Assumption \ref{assump:W} requires the communication graph to be undirected and connected, and the weight matrix to be doubly-stochastic. While this is more restrictive than the directed-graph settings considered in some existing privacy-preserving algorithms (e.g., \cite{wang2024tailoring, huang2024differential, Yang2025Differentially}), it is adopted here to enable a clear investigation of the fundamental challenges in differentially private distributed optimization. Extending the design methodology and theoretical results of this work to directed graphs, which typically rely only on row- or column-stochastic weights, is an important and nontrivial direction for future research.
\end{remark}

Next, we introduce some notations. We describe the distributed optimization problem $\mathcal{P}$ in (\ref{problem}) using the tuple $(\mathcal{X}, \mathcal{F}, f, \mathcal{G}_{W})$, where $\mathcal{X} = \mathbb{R}^{p}$ is the domain of optimization. $\mathcal{F} \subseteq \{ \mathcal{X} \to \mathbb{R} \}$ is a set of real-valued cost functions. $f(\bm{x}) \triangleq \frac{1}{n}\sum_{i=1}^{n} f_{i}(\bm{x})$ with $f_{i} \in \mathcal{F}$ for each $i \in \mathcal{N}$. $\mathcal{G}_{W}$ is a communication graph induced by the weight matrix $W$. We represent a distributed optimization algorithm for solving the problem $\mathcal{P}$ by a mapping $Alg\left( \mathcal{P}, \mathbf{x}(0) \right): \{ \mathcal{P}, \mathbf{x}(0) \} \to \mathcal{Z}$, where $\mathbf{x}(0) \in \mathcal{X}^{n}$ is the initial state of all agents. $\mathcal{Z} = \{ \mathbf{z}(0), \mathbf{z}(1), \cdots \} \in \mathbb{O}$ is the sequence of shared information of all agents, also known as the observation sequence, and $\mathbb{O}$ denotes the set of all possible observation sequences. Note that since each $\mathbf{z}(k)$ takes values in the continuous space $\mathcal{X}^{n}$, the observation set $\mathbb{O}$ is continuous. Furthermore, under a distributed optimization algorithm $Alg$, for $\mathcal{P}$, $\mathbf{x}(0)$, and $\mathcal{Z}$, we denote the algorithm's internal states as $Sta(\mathcal{P}, \mathbf{x}(0), \mathcal{Z})$, which encompass all data and variables that are stored locally by each agent and not directly shared with other agents.

With the above notations, we are ready to present the definitions of \textit{adjacency} for two distributed optimization problems, and of \textit{$\epsilon$-DP} and \textit{sensitivity} for a distributed algorithm.
\begin{definition}\label{definition_adjacency}
(Adjacency \cite{huang2024differential}) Two distributed optimization problems $\mathcal{P}$ and $\mathcal{P}'$ are adjacent, if the following hold:
\begin{enumerate}
  \item $\mathcal{X} = \mathcal{X}'$, $\mathcal{F} = \mathcal{F}'$, and $\mathcal{G}_{W} = \mathcal{G}_{W}'$.
  \item There exists an $i_{0} \in \mathcal{N}$ such that $f_{i_{0}} \not= f_{i_{0}}'$ and for all $j \not= i_{0}$, $f_{j} = f_{j}'$.
  \item The distance between the gradients of $f_{i_{0}}$ and $f_{i_{0}}'$ is bounded by $\delta$ on $\mathcal{X}$, i.e., for any $\bm{x} \in \mathcal{X}$, $\| \nabla f_{i_{0}}(\bm{x}) - \nabla f_{i_{0}}'(\bm{x}) \|_{1} \le \delta$.
\end{enumerate}
\end{definition}

From Definition \ref{definition_adjacency}, two distributed optimization problems are adjacent if only one agent modifies its cost function while all other parameters remain unchanged. Furthermore, we require that the distance between the gradients of two different cost functions is bounded. To the best of our knowledge, this condition is the most relaxed in the existing literature, and a detailed discussion is provided in the following remark.
\begin{remark}\label{remark_cost_function}
Regarding the condition on different cost functions in adjacent problems, \cite{huang2015differentially} and \cite{yu2023gradienttracking} require that the gradients of all cost functions be bounded, i.e., for any $\bm{x}\in\mathcal{X}$ and $f_{i}\in\mathcal{F}$, $\|\nabla f_{i}(\bm{x})\|\le \delta_{1}$ with $\delta_{1}>0$. In contrast, our requirement in Definition \ref{definition_adjacency} is more relaxed, allowing for a broader class of cost functions. For example, the functions $f_{i_{0}}(\bm{x}) = \frac{1}{2}\bm{x}^{T}\bm{x}$ and $f_{i_{0}}'(\bm{x}) = \frac{1}{2}\bm{x}^{T}\bm{x} + \frac{\delta}{p}\bm{x}$ satisfy our requirement, but their gradients are unbounded on $\mathbb{R}^{p}$. Furthermore, \cite{ding2022differentially} and \cite{Yang2025Differentially} assume that $f_{i_{0}}$ and $f_{i_{0}}'$ share the same Lipschitz and convexity constants. \cite{wang2024tailoring} assumes that $f_{i_{0}}$ and $f_{i_{0}}'$ have identical gradients in the neighborhood of the solution of $\mathcal{P}$. These assumptions impose significantly stricter constraints compared to ours.
\end{remark}

\begin{definition}\label{definition_DP}
($\epsilon$-DP) Given an $\epsilon > 0$, a distributed optimization algorithm $Alg$ is $\epsilon$-differentially private if for any two adjacent distributed optimization problems $\mathcal{P}$ and $\mathcal{P}'$, any initial state $\mathbf{x}(0) \in \mathcal{X}^{n}$, and any set of the observation sequences $\mathcal{O} \subseteq \mathbb{O}$,
\begin{equation}
\mathbb{P}\left[ Alg\left( \mathcal{P}, \mathbf{x}(0) \right) \in \mathcal{O} \right] \le e^{\epsilon}\mathbb{P}\left[ Alg\left( \mathcal{P}', \mathbf{x}(0) \right) \in \mathcal{O} \right],
\end{equation}
where the probability $\mathbb{P}$ is taken over the randomness of the algorithm.
\end{definition}

By Definition \ref{definition_DP}, a distributed algorithm with $\epsilon$-DP ensures that the distribution of any agent's shared information remains nearly indistinguishable between adjacent problems. In other words, an adversary who has obtained the observation sequence cannot identify the individual cost function of any one agent with significant probability. In addition, a smaller $\epsilon$ indicates a higher level of privacy.
%is less distinguishable

\begin{definition}\label{def:sensitivity}
(Sensitivity) At each iteration $k$, for any two adjacent distributed optimization problems $\mathcal{P}$ and $\mathcal{P}'$ and any initial state $\mathbf{x}(0) \in \mathcal{X}^{n}$, the sensitivity of the algorithm is
\begin{equation}
\Delta(k) \triangleq \sup_{\mathcal{Z} \in \mathbb{O}} \sup_{\substack{\mathbf{x}(k) \in Sta(\mathcal{P}, \mathbf{x}(0), \mathcal{Z}) \\ \mathbf{x}'(k) \in Sta(\mathcal{P}', \mathbf{x}(0), \mathcal{Z})}} \| \mathbf{x}(k) - \mathbf{x}'(k) \|_{1}.
\end{equation}
%where the used norm is $L^{1}$-norm.
\end{definition}

The sensitivity $\Delta(k)$ characterizes the maximum deviation in the algorithm's states at iteration $k$ under adjacent problems. More importantly, the algorithm's sensitivity directly influences the noise variance required to achieve DP. According to Lemma 2 in \cite{huang2015differentially} (or see Lemma \ref{lemma_dp} in Section \ref{section_DP}), the sensitivity of the algorithm is positively correlated with the noise variance under a fixed privacy budget, i.e., lower sensitivity necessitates a smaller noise variance. Conversely, the algorithm's optimization accuracy is inversely proportional to the noise variance, i.e., the smaller the noise variance, the higher the algorithm's accuracy \cite{huang2024differential}. In summary, reducing the sensitivity of the algorithm can enhance its accuracy. This insight provides a clear guideline for designing differentially private algorithms with enhanced accuracy: minimizing the algorithm's sensitivity as much as possible.

\section{Algorithm Design}\label{sec: algorithm}
In this section, we propose a differentially private distributed algorithm with lower sensitivity. Additionally, we conduct a comparative analysis of our algorithm against related methods.

\subsection{Algorithm Design}
As discussed in the previous section, the key to designing a differentially private algorithm with enhanced accuracy lies in reducing the algorithm's sensitivity. Thus, identifying the main factors that affect the algorithm's sensitivity becomes crucial. Generally, distributed optimization algorithms consist of two main steps: a \textit{consensus} step and an \textit{optimization} step. In the consensus step, each agent combines its state with those of neighboring agents to minimize the consensus error related to other agents' states. In the optimization step, each agent updates its state along the gradient direction to reduce the optimization error associated with Problem (\ref{problem}). To achieve DP, each agent needs to inject random noise into its shared state, which we refer to as the noisy state. However, when designing a differentially private distributed algorithm, two issues arise:
	\begin{enumerate}
    \item In the consensus step, should agents use their true states or noisy states? 
    \item In the optimization step, should agents compute gradients at true states or noisy states? 
	\end{enumerate}
We argue that both the design choices mentioned above and the amount of shared information influence the algorithm's sensitivity. The key findings are summarized as follows:
%	\begin{proposition}\label{prop:sensitivity}
%	\textcolor{red}{
%For a differentially private distributed algorithm, the following claims hold:
%	\begin{enumerate}
%    \item[(i)] If agents use their noisy states in the consensus step, the upper bound of the algorithm's sensitivity is lower.
%    \item[(ii)] If agents compute gradients based on noisy states, the upper bound of the algorithm's sensitivity is lower.
%    \item[(iii)] If the shared information between agents is less, the upper bound of the algorithm's sensitivity is lower.
%	\end{enumerate}}
%	\end{proposition}
%	\begin{proposition}\label{prop:sensitivity}
%	Consider two differentially private distributed algorithms (\textit{Alg 1} and \textit{Alg 2}). Let $\Delta_{1}(k)$ and $\Delta_{2}(k)$ be the sensitivities of the two algorithms, respectively. Under Definition 1, the following claims hold:
%	\begin{enumerate}
%	\item[(i)] \textit{Alg 1:} agents use their noisy states in the consensus step. \textit{Alg 2:} agents use their true states in the consensus step. Then, $\Delta_{1}(k) \le \Delta_{2}(k)$ for all $k\ge1$.
%	\item[(ii)] \textit{Alg 1:} agents compute gradients based on noisy states. \textit{Alg 2:} agents compute gradients based on true states. Then, $\Delta_{1}(k) \le \Delta_{2}(k)$ for all $k\ge1$.
%	\item[(iii)] \textit{Alg 1:} agents exchange one variable $\bm{x}$. \textit{Alg 2:} agents exchange two variables $\bm{x}$ and $\bm{y}$. Then, $\Delta_{1}(k) \le \Delta_{2}(k)$ for all $k\ge1$.
%	\end{enumerate}
%	\end{proposition}
\begin{proposition}\label{prop:sensitivity}
Consider two differentially private distributed algorithms, denoted as \textit{Alg 1} and \textit{Alg 2}, with sensitivities $\Delta_{1}(k)$ and $\Delta_{2}(k)$, respectively. Under the conditions given in Definition \ref{definition_adjacency}, the following claims hold:
\begin{enumerate}
\item[(i)] Consensus step design: If \textit{Alg 1} has each agent use its own noisy state in the consensus step, while \textit{Alg 2} has each agent use its own true state, then $\Delta_{1}(k) \le \Delta_{2}(k)$ for all $k\ge1$.
\item[(ii)] Gradient computation design: If \textit{Alg 1} computes gradients based on noisy states while \textit{Alg 2} computes gradients based on true states, then $\Delta_{1}(k) \le \Delta_{2}(k)$ for all $k\ge1$.
\item[(iii)] Number of exchanged variables: If \textit{Alg 1} exchanges only variable $\bm{x}$ while \textit{Alg 2} exchanges two variables $\bm{x}$ and $\bm{y}$, then $\Delta_{1}(k) \le \Delta_{2}(k)$ for all $k\ge1$.
\end{enumerate}
\end{proposition}
	
	%with sensitivity $\Delta_{1}(k)$ and $\Delta_{2}(k)$, respectively. 
	\begin{proof}
The proof is provided in Appendix \ref{proof_proposition}.
	\end{proof}

Proposition \ref{prop:sensitivity} reveals that in differentially private distributed optimization, more extensive use of noisy information (i.e., using noisy states in both consensus and gradient computation) together with reducing the number of variables exchanged between agents can systematically reduce the algorithm's sensitivity. The intuition is twofold. Firstly, using noisy information acts like a ``blurring'' process on the exact information, which mitigates the influence of an agent's private data on its outputs, thereby reducing sensitivity. Secondly, reducing the shared information means an agent exposes less data externally, which naturally limits the amount of private information that could be disclosed and thus lowers sensitivity. Crucially, a lower sensitivity means that less noise is required to achieve the same privacy protection level. This directly translates into higher accuracy for the algorithm under an identical privacy budget.

\begin{algorithm}[!t]
\caption{Differentially Private Distributed Algorithm with Lower Sensitivity}\label{alg:alg_private}
\begin{algorithmic}[1]
\REQUIRE
Initialization $\bm{x}_{i}(0)\in\mathbb{R}^{p}$ and $\bm{y}_{i}(0)=\bm{0}$, the weight matrix $W$, the stepsize sequence $\{\alpha_{k}\}_{k=1}^{\infty}$, $\beta > 0$.

\ENSURE
\STATE\label{step_1}
Each agent $i$ injects noise $\bm{\xi}_{i}(k)$ into $\bm{x}_{i}(k-1)$.
\STATE\label{step_2}
Each agent $i$ sends $\bm{z}_{i}(k) = \bm{x}_{i}(k-1) + \bm{\xi}_{i}(k)$ to $j\in\mathcal{N}_{i}$ and receives $\bm{z}_{j}(k)$ from $j\in\mathcal{N}_{i}$.
\STATE\label{step_3}
Each agent $i$ executes the following updates:
	\begin{subequations}\label{algorithm}
	\begin{align}
	&\bar{\bm{z}}_{i}(k) = \sum_{j = 1}^{n} W_{ij}\bm{z}_{j}(k), \label{algorithm_1}\\
	&\bm{y}_{i}(k) = \bm{y}_{i}(k-1) + \beta\left(\bm{z}_{i}(k) - \bar{\bm{z}}_{i}(k)\right), \label{algorithm_2}\\
	&\bm{x}_{i}(k) = \bar{\bm{z}}_{i}(k) - \alpha_{k}\left( \bm{y}_{i}(k) + \nabla f_{i}\left(\bm{z}_{i}(k)\right) \right). \label{algorithm_3}
	\end{align}
	\end{subequations}
\end{algorithmic}
\end{algorithm}

Building on the design principles in Proposition \ref{prop:sensitivity}, we propose a differentially private distributed algorithm with lower sensitivity, as detailed in Algorithm \ref{alg:alg_private}. In this algorithm, $\bm{y}_{i}$ serves as an auxiliary variable that compensates for the error between the local gradient and the average gradient, thereby accelerating convergence. Its principle will be discussed in the next subsection. To achieve $\epsilon$-DP, Algorithm \ref{alg:alg_private} injects additive noise $\bm{\xi}_{i}$ into the shared information of agent $i$ at every iteration $k$. Notably, in Algorithm \ref{alg:alg_private}, each agent $i$ utilizes its noisy state $\bm{z}_{i}$ in the consensus step (\ref{algorithm_1}), computes the gradient $\nabla f_{i}\left(\cdot\right)$ at its noisy state $\bm{z}_{i}$ in the optimization step (\ref{algorithm_3}), and shares only the variable $\bm{z}_{i}$ to its neighboring agents, without disclosing the auxiliary variable $\bm{y}_{i}$. These designs adhere to Proposition \ref{prop:sensitivity} and ensure that Algorithm \ref{alg:alg_private} achieves lower sensitivity.
%, as demonstrated in Table \ref{tab: sensitivity} of Section \ref{section_DP}.

Let $\bm{\xi}(k) \triangleq [\bm{\xi}_{1}(k), \bm{\xi}_{2}(k), \cdots, \bm{\xi}_{n}(k)]^{T} \in \mathbb{R}^{n\times p}$. The injected random noise satisfies the following assumption:
	\begin{assumption}\label{assumption_noise}
	 The random noise $\bm{\xi}(k)$ is independent across dimensions $p$, agents $i$, and iterations $k$. Moreover, it follows a zero-mean Laplace distribution with parameter $\nu_{k}$, i.e., $\bm{\xi}(k) \sim Lap(\nu_{k})$.
	 %$[\bm{\xi}_{i}(k)]_{j} \sim Lap(\nu_{k}), \forall j\in\{1,2,\cdots,p\}$
	\end{assumption}
	
For the stepsize $\alpha_{k}$ and the noise parameter $\nu_{k}$, we choose them as exponentially decaying sequences:
	\begin{equation}\label{eq_alpha_nu}
	\alpha_{k} = \gamma q_{1}^{k-1},\quad \nu_{k} = \frac{\gamma\delta q_{2}}{\epsilon(q_{2} - q_{1})}q_{2}^{k-1},
	\end{equation}
where $\gamma$ is the initial stepsize parameter with $\gamma\beta\le1$, $q_{1}\in(0,1)$ is the stepsize decay parameter, $\delta$ is the bound on the gradient distance in Definition \ref{definition_adjacency}, $q_{2}\in (q_{1},1)$ is the noise decay parameter, and $\epsilon > 0$ is the privacy budget.

\begin{table}[!t]
  \centering
  \caption{Comparison of Algorithm \ref{alg:alg_private} with Existing Algorithms.} 
   \begin{tabular}{cccc}%{M{1.5cm}M{3cm}M{3cm}M{1.5cm}}
    \toprule
    \multirow{2}{*}{Algorithm} & Explicit, practical & Number of & Communica- \\
    & expression for $\nu_{k}(\epsilon)$ & added noise & tion \\
    \midrule
    \cite{ding2022differentially} & \ding{55} & $2p$ & $2p$ \\
    \cite{yu2023gradienttracking} & \ding{55} & $2p$ & $2p$ \\ 
    \cite{wang2024tailoring} & \ding{55} & $2p$ & $2p$ \\
    \cite{huang2024differential} & \ding{55} & $2p$ & $2p$ \\
    \cite{Yang2025Differentially} &  \ding{55} & $2p$ & $2p$ \\
    Our method & \Checkmark & $p$ & $p$ \\
    \bottomrule
  \end{tabular}
  \label{table_algorithm_communication}
\end{table}

Algorithm \ref{alg:alg_private} involves four tunable parameters $\gamma$, $\beta$, $q_{1}$, and $q_{2}$. For these parameters, we provide two tuning strategies in Section \ref{section:convergence}. Using parameters (\ref{eq_alpha_nu}), Algorithm \ref{alg:alg_private} achieves  $\epsilon$-DP, convergence, and a guaranteed level of accuracy. These properties will be analyzed in detail in Sections \ref{section_DP} and \ref{section:convergence}. Moreover, we summarize the advantages of Algorithm \ref{alg:alg_private} over existing state-of-the-art algorithms \cite{ding2022differentially, Yang2025Differentially, yu2023gradienttracking, wang2024tailoring, huang2024differential} in Table \ref{table_algorithm_communication}, with detailed discussions provided in the following remarks.

\begin{remark}
For Algorithm \ref{alg:alg_private}, given any privacy budget $\epsilon > 0$, the noise parameter $\nu_{k}$ can be directly and explicitly computed offline using the closed-form expression in Eq. (\ref{eq_alpha_nu}). This provides a straightforward design rule, contrasting with the implicit or complex analytical relations in \cite{ding2022differentially, Yang2025Differentially, yu2023gradienttracking, wang2024tailoring, huang2024differential}, which primarily serve for posterior analysis or as verification constraints rather than direct synthesis tools.
\end{remark}

\begin{remark}
In the proposed algorithm, only variable $\bm{x}$ is shared among agents, while the auxiliary variable $\bm{y}$ is kept local. Compared to existing GT-based algorithms \cite{ding2022differentially, yu2023gradienttracking, wang2024tailoring, huang2024differential, Yang2025Differentially}, which require exchanging two variables, the proposed algorithm reduces both the number of communication variables and the amount of injected noise by half, significantly lowering communication and computational overhead. Note that the DPOP method \cite{huang2015differentially} also exchanges only one variable. However, DPOP is based on DGD and relies on a decaying stepsize for convergence, which leads to a sublinear convergence rate even without privacy noise. In contrast, our algorithm employs an implicit gradient-tracking mechanism (see the next subsection for a detailed comparison among DGD, GT, and our scheme), which enables more effective gradient alignment. Under the same privacy budget, our algorithm achieves higher optimization accuracy (see Fig. \ref{fig_residual_dpop} in Section \ref{section_simulation}), demonstrating a superior privacy-accuracy trade-off inherited from the GT framework.
\end{remark}
%Moreover, in Algorithm \ref{alg:alg_private}, each agent $i$ only needs to add noise to $\bm{x}_{i}$ and exchange $\bm{z}_{i}$ with neighboring agents at every iteration $k$. The number of injected noise and communication variables in Algorithm \ref{alg:alg_private} is half that of the existing state-of-the-art algorithms, as shown in Table \ref{table_algorithm_communication}. Reducing the number of injected noise can decrease the computational cost of the algorithm, especially for medium- to high-dimensional problems. Additionally, reducing the number of communication variables can significantly decrease the communication cost of the algorithm.

\subsection{Related Methods}\label{sebsec_related_methods}
To get a better insight into the update rule of our algorithm, we compare it with the DGD and GT methods. Notably, \cite{huang2015differentially} and \cite{wang2024tailoring} develop differentially private distributed algorithms based on the DGD method. The update rule for DGD is
	\begin{equation}\label{eq:DGD}
	\mathbf{x}(k+1) = W\mathbf{x}(k) - \alpha \nabla F(\mathbf{x}(k)).
	\end{equation}
The DGD method with a constant stepsize cannot converge to the optimal solution. To illustrate it, assume that $\mathbf{x}(k)$ converges to a consistent limit $\mathbf{x}(\infty)$. Since $W\bm{1}=\bm{1}$, it follows that $\mathbf{x}(\infty)=W\mathbf{x}(\infty)$. Taking the limit as $k \to \infty$ on both sides of (\ref{eq:DGD}), we obtain $\nabla F(\mathbf{x}(\infty)) = \bm{0}$. This means that all local cost functions $f_{i}$ attain their minimum at the same $\bm{x}_{i}(\infty)$, which is generally impossible. Thus, the DGD method's stepsize must be decaying to ensure exact convergence. However, the decaying stepsize results in slow convergence both in theory and practice \cite{Jakovetic2014fast}.

The GT method overcomes the above limitation by introducing an auxiliary variable $\mathbf{y}$, which enables exact convergence with a constant stepsize. Note that the differentially private distributed algorithms in \cite{ding2022differentially, yu2023gradienttracking, wang2024tailoring, huang2024differential, Yang2025Differentially} are based on the GT method or its variants. The update rule for GT is given by:
	\begin{subequations}
	\begin{align}
	\mathbf{x}(k+1) &= W\mathbf{x}(k) - \alpha\mathbf{y}(k), \label{eq:GT_x} \\
	\mathbf{y}(k+1) &= W\mathbf{y}(k) + \nabla F(\mathbf{x}(k+1)) - \nabla F(\mathbf{x}(k)), \label{eq:GT_y}
	\end{align}
	\end{subequations}
where $\mathbf{y}(0) = \nabla F(\mathbf{x}(0))$. To demonstrate how GT achieves exact convergence, we first substitute (\ref{eq:GT_x}) into (\ref{eq:GT_y}) to derive the formula of $\mathbf{x}(k+2) - \mathbf{x}(k+1)$. Next, combining the first update $\mathbf{x}(1) = W\mathbf{x}(0) - \alpha\nabla F( \mathbf{x}(0))$ with $\mathbf{x}(2) - \mathbf{x}(1), \cdots, \mathbf{x}(k+1) - \mathbf{x}(k)$ yields:
	\begin{equation}\label{eq:GT_var}
	\begin{aligned}
	\mathbf{x}(k+1) =&\ W\mathbf{x}(k) - \alpha \nabla F(\mathbf{x}(k)) + \sum_{l=1}^{k}\left(W-I\right)\mathbf{x}(l) \\
	& + \sum_{l=0}^{k-1}(W-W^{2})\mathbf{x}(l).
	\end{aligned}
	\end{equation}
In comparison with the DGD method, the GT method introduces two additional terms. Taking the limit as $k \to \infty$ on both sides of (\ref{eq:GT_var}), we obtain $\alpha\nabla F(\mathbf{x}(\infty)) = \sum_{l=1}^{\infty}\left(W-I\right)\mathbf{x}(l) + \sum_{l=0}^{\infty}(W- W^{2})\mathbf{x}(l)$. Since $\bm{1}^{T}(W-I) = 0$ and $\bm{1}^{T}(W-W^{2}) = 0$, we obtain $\bm{1}^{T}\nabla F(\mathbf{x}(\infty)) = 0$, which corresponds to the objective of finding a point that minimizes the average of $\sum_{i=1}^{n} f_{i}$.

For Algorithm \ref{alg:alg_private}, to better illustrate its fundamental mechanism, we omit the noise term and replace the decaying stepsize $\alpha_{k}$ with a fixed stepsize $\alpha$, yielding the following update rule:
	\begin{subequations}
	\begin{align}
	\mathbf{y}(k+1) &= \mathbf{y}(k) + \beta(I-W)\mathbf{x}(k), \label{alg1:no_noise_y} \\
	\mathbf{x}(k+1) &= W\mathbf{x}(k) - \alpha\left(\mathbf{y}(k+1)+\nabla F(\mathbf{x}(k))\right), \label{alg1:no_noise_x}
	\end{align}
	\end{subequations}
where $\mathbf{y}(0) = \bm{0}$. From (\ref{alg1:no_noise_y}), we can deduce that $\mathbf{y}(k+1) = \sum_{l=0}^{k}\beta(I-W)\mathbf{x}(l)$. Substituting this into (\ref{alg1:no_noise_x}) gives:
	\begin{equation}\label{eq:alg1_var}
	\mathbf{x}(k+1) = W\mathbf{x}(k) - \alpha \nabla F(\mathbf{x}(k)) + \alpha\beta\sum_{l=0}^{k}\left(W-I\right)\mathbf{x}(l).
	\end{equation}
In comparison with the GT method (\ref{eq:GT_var}), when $\alpha\beta=1$, our algorithm introduces the term $(W-I)\mathbf{x}(0)$ but eliminates the term $\sum_{l=0}^{k-1}(W-W^{2})\mathbf{x}(l)$. Taking the limit as $k \to \infty$ on both sides of (\ref{eq:alg1_var}), we obtain $\nabla F(\mathbf{x}(\infty)) = \beta\sum_{l=0}^{\infty}(W-I)\mathbf{x}(l)$. Since $\mathbf{1}^{T}(W-I) = 
0$, it follows that $\mathbf{1}^{T}\nabla F(\mathbf{x}(\infty)) = 0$. Roughly speaking, our algorithm, similar to the GT method, can achieve exact convergence with a constant stepsize. More importantly, in our algorithm, neighboring agents only share the variable $\bm{x}$, whereas in the GT method, both $\bm{x}$ and the auxiliary variable $\bm{y}$ need to be exchanged. Thus, our algorithm reduces the communication burden by half compared to the GT method. On the other hand, in the GT method, the auxiliary variable $\mathbf{y}$, satisfying $\frac{1}{n}\mathbf{1}^{T}\mathbf{y}(k) = \frac{1}{n}\mathbf{1}^{T}\nabla F(\mathbf{x}(k))$ for any $ k \ge 0 $, is used to track the average gradient. In our algorithm, the auxiliary variable $\mathbf{y}$ satisfies $\frac{1}{n}\mathbf{1}^{T}(\mathbf{y}(k)+\nabla F(\mathbf{x}(k))) = \frac{1}{n}\mathbf{1}^{T}\nabla F(\mathbf{x}(k))$ for any $k\ge0$, which can be interpreted as tracking the difference between the average gradient and the local gradient.
%However, our algorithm is more streamlined. 

\section{Differential Privacy}\label{section_DP}
\begin{table}[!t]
  \centering
  \caption{Comparison of Sensitivity of Algorithm \ref{alg:alg_private} with Existing Algorithms.}
  \begin{tabular}{c>{\raggedright\arraybackslash}p{6.5cm}}% 第一列居中，第二列左对齐且自动换行
    \toprule
    \textbf{Algorithm} & \textbf{Sensitivity} \\
    \midrule
    \cite{ding2022differentially} & Not lower than $\bm{(2c_{1} + 2)(\delta + c_{2})}$ with $c_{1}, c_{2} > 0$. \\ 
    \addlinespace[0.5ex] \cline{2-2} \addlinespace[0.5ex]
    \cite{yu2023gradienttracking} & Not lower than $\bm{(2c_{3} + 1)(\delta + c_{4})\alpha_{1}(k)}$ with $c_{3}, c_{4} > 0$ and $\sum_{k=1}^{\infty}\alpha_{1}(k) < \infty$. \\ 
    \addlinespace[0.5ex] \cline{2-2} \addlinespace[0.5ex]
    \cite{wang2024tailoring} & Not lower than $\bm{(2 - \alpha_{2}(k))(\delta + c_{5})(\alpha_3(k) + 1)}$ with $\sum_{k=1}^{\infty} \alpha_{2}(k) = \infty$, $\lim_{k\to\infty} \alpha_{2}(k) = 0$, $\sum_{k=1}^{\infty} \alpha_{3}(k) = \infty$, $\lim_{k\to\infty} \alpha_{3}(k) = 0$, and $c_{5} > 0$. \\ 
    \addlinespace[0.5ex] \cline{2-2} \addlinespace[0.5ex]
    \cite{huang2024differential} & Not lower than $\bm{3(\delta + c_{6})\alpha_{4}(k)}$ with $c_{6} > 0$ and $\sum_{k=1}^{\infty}\alpha_{4}(k) < \infty$. \\ 
    \addlinespace[0.5ex] \cline{2-2} \addlinespace[0.5ex]
    \cite{Yang2025Differentially} & Not lower than $\bm{(2c_{7} + 2)(\delta + c_{8})}$ with $c_{7}, c_{8} > 0$. \\ 
    \addlinespace[0.5ex] \cline{2-2} \addlinespace[0.5ex]
    Our method & No more than $\bm{\delta\alpha_{k}}$ with $ \sum_{k=1}^{\infty} \alpha_{k} < \infty$. \\
    \bottomrule
  \end{tabular}
  \label{tab: sensitivity}
\end{table}

In this section, we show that Algorithm \ref{alg:alg_private} with the parameter settings (\ref{eq_alpha_nu}) guarantees $\epsilon$-DP. We then validate the advantages of Algorithm \ref{alg:alg_private} in terms of accuracy.

%Recall from Section \ref{sec:problem} that we introduced an inverse observation mapping $Sta(\mathcal{P}, \mathbf{x}(0), \mathcal{Z})$, which maps a distributed optimization problem $\mathcal{P}$, an initial state $\mathbf{x}(0)$, and an observation sequence $\mathcal{Z}$ to the set of the algorithm's internal states, i.e., $\{ \mathbf{x}(0), \mathbf{y}(0), \mathbf{x}(1), \mathbf{y}(1), \cdots \}$. Using this mapping, we subsequently define the sensitivity of Algorithm \ref{alg:alg_private} that characterizes the maximum change in the states $\mathbf{x}(k)$ for two adjacent problems.

The following lemma provides a sufficient condition for Algorithm \ref{alg:alg_private} to guarantee $\epsilon$-DP.

\begin{lemma}\label{lemma_dp}
(\cite[Lemma 2]{huang2015differentially}) If each agent $i$ adds a noise vector $\bm{\xi}_{i}(k)$ under Assumption \ref{assumption_noise} such that $\sum_{k=1}^{\infty} \frac{\Delta(k)}{\nu_{k}} \le \epsilon$, then Algorithm \ref{alg:alg_private} is $\epsilon$-differentially private.
\end{lemma}

According to Lemma \ref{lemma_dp}, ensuring that the algorithm achieves $\epsilon$-DP requires satisfying two conditions: (i) the algorithm's sensitivity $\Delta(k)$ is bounded, and (ii) the variance $\nu_{k}$ of the injected noise is sufficiently large such that $\sum_{k=1}^{\infty} \frac{\Delta(k)}{\nu_{k}} \le \epsilon$. Furthermore, it is established that the variance of the required noise is proportional to the sensitivity. This implies that the sensitivity bound directly determines the lower limit on the scale of noise that must be added to achieve $\epsilon$-DP. In the following theorem, we first give an upper bound on the sensitivity of Algorithm \ref{alg:alg_private}. Then,  based on Lemma \ref{lemma_dp}, we prove that Algorithm \ref{alg:alg_private}, with the parameter settings given in (\ref{eq_alpha_nu}), is $\epsilon$-differentially private.

\begin{theorem}\label{lemma_sensitivity}
Under Definitions \ref{definition_adjacency} and \ref{def:sensitivity}, the sensitivity of Algorithm \ref{alg:alg_private} satisfies $\Delta(k) \le \delta \alpha_{k}$. Given any $\epsilon > 0$, if the random noise $\bm{\xi}(k) \sim Lap(\nu_{k})$ satisfies Assumption \ref{assumption_noise}, and the stepsize $\alpha_{k}$ and the noise parameter $\nu_{k}$ are chosen according to (\ref{eq_alpha_nu}), then Algorithm \ref{alg:alg_private} guarantees $\epsilon$-DP.
\end{theorem}
%Algorithm \ref{alg:alg_private}, with the parameter settings in (\ref{eq_alpha_nu}), guarantees $\epsilon$-DP for any choice of $\gamma > 0$, $q_{1} \in (0, 1)$, and $q_{2} \in (q_{1}, 1)$.
\begin{proof}
With the same initial states $\mathbf{x}(0)$, $\mathbf{y}(0)$ and the same observation sequence $\mathcal{Z}$, for any two adjacent problems $\mathcal{P}$ and $\mathcal{P}'$, we have $\bm{z}_{i}(k) = \bm{z}_{i}'(k)$ for all $i\in\mathcal{N}$ and all $k\ge1$. Based on (\ref{algorithm_1}) and (\ref{algorithm_2}), it follows that $\bm{y}_{i}(k) = \bm{y}_{i}'(k)$ for all $i \in \mathcal{N}$ and all $k \ge 1$. From Definition \ref{definition_adjacency}, we know that $f_{j} \not= f_{j}'$ if and only if $j=i_{0}$. In light of (\ref{algorithm_3}), it follows that $\bm{x}_{j}(k) = \bm{x}_{j}'(k)$ for all $j\in\mathcal{N}$, $j\not=i_{0}$ and all $k \ge 1$. Therefore, the sensitivity of Algorithm \ref{alg:alg_private} at iteration $k$ is bounded by
\begin{equation}
\begin{aligned}
\| \mathbf{x}(k) - \mathbf{x}'(k) \|_{1} &= \| \bm{x}_{i_{0}}(k) - \bm{x}_{i_{0}}'(k) \|_{1} \\
&= \alpha_{k} \| \nabla f_{i_{0}}(\bm{z}_{i_{0}}(k)) - \nabla f_{i_{0}}'(\bm{z}_{i_{0}}'(k)) \|_{1} \\
&= \alpha_{k} \| \nabla f_{i_{0}}(\bm{z}_{i_{0}}(k)) - \nabla f_{i_{0}}'(\bm{z}_{i_{0}}(k)) \|_{1} \\
&\le \delta \alpha_{k}.
\end{aligned}
\end{equation}

It follows from Lemma \ref{lemma_dp} that Algorithm \ref{alg:alg_private} is $\epsilon$-differentially private if $\sum_{k=1}^{\infty} \frac{\delta \alpha_{k}}{\nu_{k}} \le \epsilon$. When $\alpha_{k} = \gamma q_{1}^{k-1}$ and $\nu_{k} = \frac{\gamma \delta q_{2}}{\epsilon(q_{2} - q_{1})}q_{2}^{k-1}$, and given that $\gamma > 0$, $q_{1} \in (0, 1)$, $q_{2} \in (q_{1}, 1)$, we have
\begin{equation}
\sum_{k=1}^{\infty} \frac{\delta \alpha_{k}}{\nu_{k}} = \frac{\epsilon (q_{2} - q_{1})}{q_{2}}\sum_{k=1}^{\infty}\left( \frac{q_{1}}{q_{2}} \right)^{k-1} = \epsilon.
\end{equation}
Therefore, Algorithm \ref{alg:alg_private} with the parameter settings given in (\ref{eq_alpha_nu}) guarantees $\epsilon$-DP.
\end{proof}

%\begin{remark}
%In Table \ref{tab: sensitivity}, we compare the sensitivity of differentially private distributed algorithms under the same definitions and assumptions outlined in this paper. As shown, our algorithm demonstrates lower sensitivity compared to algorithms in \cite{ding2022differentially, yu2023gradienttracking, wang2024tailoring, huang2024differential, Yang2025Differentially}. This advantage stems from its adherence to the design principles outlined in Proposition \ref{prop:sensitivity}, whereas other algorithms did not consider the sensitivity. According to Lemma \ref{lemma_dp}, for a fixed privacy budget, the algorithm's sensitivity is positively correlated with the noise variance. That is, lower sensitivity requires smaller noise variance. This leads to possible improvement of the optimization accuracy. Therefore, under the same privacy budget, our algorithm can achieve higher optimization accuracy than existing state-of-the-art algorithms.
%\end{remark}
\begin{remark}
In Table \ref{tab: sensitivity}, we compare algorithm sensitivities under the consistent framework. As Lemma \ref{lemma_dp} establishes, for a fixed privacy budget, sensitivity directly dictates the required noise variance: lower sensitivity enables smaller variance, which in turn improves optimization accuracy. As shown, our algorithm achieves the lowest sensitivity among the considered methods \cite{ding2022differentially, yu2023gradienttracking, wang2024tailoring, huang2024differential, Yang2025Differentially}. This stems from its adherence to the design principles outlined in Proposition \ref{prop:sensitivity}, whereas other algorithms did not consider the sensitivity. Therefore, under the same privacy budget, our algorithm guarantees higher accuracy.
\end{remark}

%Furthermore, smaller injected noise leads to less degradation in the optimization accuracy of the algorithm.
%This is due to our algorithm adhering
%According to Lemma \ref{lemma_dp}, by adding appropriate noise, the algorithm can achieve $\epsilon$-DP. The amount of noise depends on the algorithm's sensitivity. 

\begin{remark}
In Theorem \ref{lemma_sensitivity}, the DP guarantee is stated for the infinite observation sequence $\mathcal{Z}$, but it directly applies to any adversary observing only a finite number of iterations $K$. This is because the $\epsilon$-DP guarantee on the entire sequence implies that the distributions of any finite prefix of the sequence are also $\epsilon$-indistinguishable. Therefore, our privacy guarantee is robust for practical (finite-observation) threat models.
\end{remark}

\section{Convergence}\label{section:convergence}
%In this section, we investigate the convergence properties of Algorithm \ref{alg:alg_private}. We first define $\bar{\bm{x}}(k) := \frac{1}{n}\mathbf{1}^{T}\mathbf{x}(k)$. In Subsection \ref{subsection:basic_lemma}, we provide some fundamental relationships and lemmas. Then, in Subsection \ref{subsection:main_result}, we construct a linear system of inequalities about $\mathbb{E}\left[\left\|\bar{\bm{x}}(k+1)-\bm{x}^*\right\|^{2}\right]$, $\mathbb{E}\left[\left\|\mathbf{x}(k+1)-\mathbf{1}\bar{\bm{x}}(k+1)\right\|^{2}\right]$, $\mathbb{E}\left[\left\|\mathbf{x}(k+1)-\mathbf{x}(k)\right\|^{2}\right]$ and their previous values. Subsequently, we prove the convergence and give a bound on the convergence accuracy. The theoretical results reveal a trade-off between privacy and accuracy.
In this section, we first present some fundamental lemmas, followed by the convergence proof of Algorithm \ref{alg:alg_private} and analysis of the optimization accuracy. We show clearly the trade-off between privacy and accuracy and provide guidance for tuning the algorithm's parameters.

\subsection{Preliminary Analysis and Fundamental Lemmas}\label{subsection:basic_lemma}
We rewrite Algorithm \ref{alg:alg_private} in an augmented form as follows:
\begin{subequations}
\begin{align}
\mathbf{y}(k) &= \mathbf{y}(k-1) + \beta\left( I - W\right)\mathbf{z}(k), \label{algorithm_compact_y}\\
\mathbf{x}(k) &= W\mathbf{z}(k) - \alpha_{k}\left( \mathbf{y}(k) + \nabla F(\mathbf{z}(k)) \right), \label{algorithm_compact_x}
\end{align}
\end{subequations}
where $\mathbf{y}(0) = \mathbf{0}$ and $\mathbf{z}(k) = \mathbf{x}(k - 1) + \bm{\xi}(k)$. Given that $W$ is doubly-stochastic, it follows from (\ref{algorithm_compact_y}) that $\frac{1}{n}\mathbf{1}^{T}\mathbf{y}(k) = 0$ for all $k\ge1$. Define $\bar{\bm{x}}(k) \triangleq \frac{1}{n}\mathbf{1}^{T}\mathbf{x}(k)$. From (\ref{algorithm_compact_x}), we obtain
\begin{equation}\label{equation_bar_x}
\begin{aligned}
\bar{\bm{x}}(k) &= \bar{\bm{x}}(k-1) - \frac{\alpha_{k}}{n}\mathbf{1}^{T} \nabla F(\mathbf{z}(k)) + \frac{1}{n}\mathbf{1}^{T}\bm{\xi}(k).
\end{aligned}
\end{equation}
According to Assumption \ref{assump:smooth_convex}, Problem (\ref{problem}) has a unique optimal solution, denoted as $\bm{x}^{*}$. We further have
\begin{equation}\label{eq:x_x_star}
\begin{aligned}
&\bar{\bm{x}}(k+1) - \bm{x}^{*} = \bar{\bm{x}}(k) - \frac{\alpha_{k+1}}{n}\mathbf{1}^{T}\nabla F(\mathbf{1}\bar{\bm{x}}(k)) - \bm{x}^{*} \\
&\ - \frac{\alpha_{k+1}}{n}\mathbf{1}^{T}\left( \nabla F(\mathbf{z}(k+1)) - \nabla F(\mathbf{1}\bar{\bm{x}}(k)) \right) + \frac{1}{n}\mathbf{1}^{T}\bm{\xi}(k+1).
\end{aligned}
\end{equation}

Using (\ref{algorithm_compact_y}) to represent $\mathbf{y}(k+1)$ in terms of $\mathbf{y}(k)$ and $\mathbf{z}(k+1)$, and next eliminating $\mathbf{y}(k)$ via (\ref{algorithm_compact_x}), we get
\begin{equation}\label{eq:y_k+1}
\begin{aligned}
\mathbf{y}&(k+1) =\ \frac{1}{\alpha_{k}}W\mathbf{x}(k-1) + \Big((\beta-\frac{1}{\alpha_{k}})I - \beta W \Big)\mathbf{x}(k) \\
&- \nabla F(\mathbf{z}(k)) + \beta(I-W)\bm{\xi}(k+1) + \frac{1}{\alpha_{k}}W\bm{\xi}(k).
\end{aligned}
\end{equation}
Then, it follows from (\ref{algorithm_compact_x}) and (\ref{equation_bar_x}) that
%Based on (\ref{algorithm_compact_x}) and (\ref{equation_bar_x}), we have
\begin{equation}\label{eq:x_x_bar}
\begin{aligned}
&\quad \mathbf{x}(k+1) - \mathbf{1}\bar{\bm{x}}(k+1) \\
&= W\mathbf{x}(k) - \alpha_{k+1}\mathbf{y}(k+1) - \alpha_{k+1}\nabla F(\mathbf{z}(k+1)) - \mathbf{1}\bar{\bm{x}}(k) \\
&\quad + \frac{\alpha_{k+1}}{n}\mathbf{1}\mathbf{1}^{T} \nabla F(\mathbf{z}(k+1)) + W\bm{\xi}(k+1) - \frac{1}{n}\mathbf{1}\mathbf{1}^{T}\bm{\xi}(k+1) \\
&= \tilde{W}(k)\mathbf{x}(k) - \mathbf{1}\bar{\bm{x}}(k) + \frac{\alpha_{k+1}}{\alpha_{k}}W(\mathbf{x}(k) - \mathbf{x}(k-1))\\
&\quad - \alpha_{k+1}(\nabla F(\mathbf{z}(k+1)) - \nabla F(\mathbf{z}(k))) \\
&\quad + \frac{\alpha_{k+1}}{n}\mathbf{1}\mathbf{1}^{T}\nabla F(\mathbf{z}(k+1))  - \frac{\alpha_{k+1}}{\alpha_{k}}W\bm{\xi}(k) \\
&\quad + \Big((1+\alpha_{k+1}\beta)W-\alpha_{k+1}\beta I-\frac{1}{n}\mathbf{1}\mathbf{1}^{T}\Big)\bm{\xi}(k+1),
\end{aligned}
\end{equation}
where $\tilde{W}(k) \triangleq (\frac{\alpha_{k+1}}{\alpha_{k}} - \alpha_{k+1}\beta)I + (1 - \frac{\alpha_{k+1}}{\alpha_{k}} + \alpha_{k+1}\beta)W$ and the second equality follows from (\ref{eq:y_k+1}).

\begin{remark}\label{remark:W_tilde}
Given that $\alpha_{k} = \gamma q_{1}^{k-1}$ with $\gamma > 0$, $q_{1} \in (0, 1)$, and $\gamma\beta \le 1$, we have $\frac{\alpha_{k+1}}{\alpha_{k}} - \alpha_{k+1}\beta = q_{1} - \gamma\beta q_{1}^{k} \in [0,1)$ for all $k\ge 1$. Since $W$ is doubly-stochastic, we conclude that $\tilde{W}(k)$ is also a doubly-stochastic matrix for all $k \ge 1$. Let $\sigma_{k} \triangleq \| \tilde{W}(k) - \frac{1}{n}\mathbf{1}\mathbf{1}^{T} \|$. Since the graph induced by $\tilde{W}(k)$ is connected, we have $\sigma_{k} \in (0,1)$ for all $k \ge 1$.
\end{remark}

Using (\ref{algorithm_compact_x}) and (\ref{eq:y_k+1}) again, we have
\begin{equation}\label{eq:x_x_k}
\begin{aligned}
&\quad \mathbf{x}(k+1) - \mathbf{x}(k) \\
&= W\mathbf{z}(k+1) - \alpha_{k+1}\mathbf{y}(k+1) - \alpha_{k+1}\nabla F(\mathbf{z}(k+1)) - \mathbf{x}(k) \\
&= \frac{\alpha_{k+1}}{\alpha_{k}}W(\mathbf{x}(k) - \mathbf{x}(k-1)) \\
&\quad + \Big(1 - \frac{\alpha_{k+1}}{\alpha_{k}} + \alpha_{k+1}\beta\Big)(W-I)\mathbf{x}(k) \\
%\left( (\frac{\alpha_{k+1}}{\alpha_{k}} - \alpha_{k+1}\beta)I + (1 - \frac{\alpha_{k+1}}{\alpha_{k}} + \alpha_{k+1}\beta)W  - I\right)\mathbf{x}(k) \\
&\quad - \alpha_{k+1}(\nabla F(\mathbf{z}(k+1)) - \nabla F(\mathbf{z}(k)))\\
&\quad + \left((1+\alpha_{k+1}\beta)W - \alpha_{k+1}\beta I\right)\bm{\xi}(k+1) - \frac{\alpha_{k+1}}{\alpha_{k}}W\bm{\xi}(k).
\end{aligned}
\end{equation}

Let $\mathcal{M}_{k}$ be the $\sigma$-algebra generated by $\{\bm{\xi}(1), \cdots, \bm{\xi}(k-1)\}$. Define $\mathbb{E}\left[\cdot|\mathcal{M}_{k}\right]$ as the conditional expectation given $\mathcal{M}_{k}$. In the following, we present some auxiliary lemmas that will be used in the analysis.

\begin{lemma}\label{lemma_smooth_convex}
(\cite[Lemma 2]{pu2020push}) Under Assumption \ref{assump:smooth_convex}, the following inequalities hold:
\begin{gather*}
\|\mathbf{1}^{T}\nabla F(\mathbf{x}(k)) - \mathbf{1}^{T}\nabla F(\mathbf{1}\bar{\bm{x}}(k)) \| \le \sqrt{n}L\| \mathbf{x}(k) - \mathbf{1}\bar{\bm{x}}(k) \|,\\
\| \mathbf{1}^{T}\nabla F(\mathbf{1}\bar{\bm{x}}(k)) \| \le nL\| \bar{\bm{x}}(k) - \bm{x}^{*} \|.
\end{gather*}
In addition, when $\alpha' \le \frac{2}{\mu + L}$, it follows that
\begin{gather*}
\|\bar{\bm{x}}(k) - \frac{\alpha'}{n}\mathbf{1}^{T}\nabla F(\mathbf{1}\bar{\bm{x}}(k)) - \bm{x}^{*}\| \le (1 - \alpha'\mu)\|\bar{\bm{x}}(k) - \bm{x}^{*}\|.
\end{gather*}
\end{lemma}

%$\textit{(a)}\ \mathbb{E}\left[\langle\mathbf{z}(k),\bm{\xi}(k)\rangle|\mathcal{M}_{k}\right] \le 2np\nu_{k}^{2}$. $\textit{(b)}\ \mathbb{E}\left[\left\langle\nabla F(\mathbf{z}(k)),\bm{\xi}(k)\right\rangle|\mathcal{M}_{k}\right] \le 2npL\nu_{k}^{2}$.
\begin{lemma}\label{lemma:xi}
The following inequalities hold for all $k\ge1$: 
\begin{enumerate}
\item[\textit{(a)}] $\mathbb{E}\left[\langle\mathbf{z}(k),\bm{\xi}(k)\rangle|\mathcal{M}_{k}\right] \le 2np\nu_{k}^{2}$.
\item[\textit{(b)}] $\mathbb{E}\left[\left\langle\nabla F(\mathbf{z}(k)),\bm{\xi}(k)\right\rangle|\mathcal{M}_{k}\right] \le 2npL\nu_{k}^{2}$.
\item[\textit{(c)}] $\mathbb{E}\left[\left\langle\nabla F(\mathbf{z}(k+1)),\bm{\xi}(k)\right\rangle|\mathcal{M}_{k}\right] \le npL(2+np+2L\alpha_{k})\nu_{k}^{2}$.
\item[\textit{(d)}] $\mathbb{E}\left[\langle\mathbf{x}(k),\bm{\xi}(k)\rangle|\mathcal{M}_{k}\right] \le 2np\left(3+L\alpha_{k}\right)\nu_{k}^{2}$.
\item[\textit{(e)}] $\mathbb{E}\left[\langle\mathbf{1}\bar{\bm{x}}(k),\bm{\xi}(k)\rangle|\mathcal{M}_{k}\right] \le 2np(1+L\alpha_{k})\nu_{k}^{2}$.
\item[\textit{(f)}] $\mathbb{E}\left[ \| \mathbf{z}(k+1) - \mathbf{1}\bar{\bm{x}}(k) \|^{2} | \mathcal{M}_{k} \right] \le 2np\nu_{k+1}^{2} + \mathbb{E}\left[ \|\mathbf{x}(k) - \mathbf{1}\bar{\bm{x}}(k) \|^{2} | \mathcal{M}_{k} \right]$.
\item[\textit{(g)}] $\mathbb{E}\left[\|\mathbf{z}(k+1)-\mathbf{z}(k)\|^{2}|\mathcal{M}_{k}\right] \le 2np\nu_{k+1}^{2} + 2np(7+2L\alpha_{k})\nu_{k}^{2} + \mathbb{E}\left[\|\mathbf{x}(k)-\mathbf{x}(k-1)\|^{2}|\mathcal{M}_{k}\right]$.
\end{enumerate}
\end{lemma}
\begin{proof}
The proof can be found in Appendix \ref{appe:proof_lemma_xi}.
\end{proof}

\begin{lemma}\label{lemma_rho}
Let $M\in\mathbb{R}^{n\times n}$ be a nonnegative, irreducible matrix with $n = 2$ or $n=3$, and let $M_{ii}<\lambda^{*}$ for some $\lambda^{*}>0$. Then, $\rho(M)<\lambda^{*}$ if and only if $\det(\lambda^{*}I - M) > 0$.
\end{lemma}
\begin{proof}
The proof for the case $n=3$ can be found in Lemma 5 in \cite{pu2021distributed}. The proof for the case $n=2$ is similar to that of $n=3$.
\end{proof}

\subsection{Main Results}\label{subsection:main_result}
Before presenting the main results, we first state some useful equations and inequalities. From the parameter settings in (\ref{eq_alpha_nu}), we obtain $\frac{\alpha_{k+1}}{\alpha_{k}}=q_{1}<1$ and $\nu_{k+1}\le \nu_{k}$. According to Assumption \ref{assump:W} and Remark \ref{remark:W_tilde}, we know that both $W$ and $\tilde{W}(k)$ are symmetric doubly-stochastic matrices. This implies that the eigenvalues of $W$ and $\tilde{W}(k)$ lie within $[-1,1]$. Since the induced graphs are connected, both $W$ and $\tilde{W}(k)$ have a unique largest eigenvalue equal to 1, with eigenvector $\bm{1}$. Thus, we have $\|W\|=1$, $\|\tilde{W}(k)\|=1$, and $\|\tilde{W}(k) - I\|\le2$. Since $0<\alpha_{k+1}\beta=\gamma\beta q_{1}^{k}<1$, we have $\|(1+\alpha_{k+1}\beta)W - \alpha_{k+1}\beta I\|\le3$ and $\|(1+\alpha_{k+1}\beta) W - \alpha_{k+1}\beta I - \frac{1}{n}\mathbf{1}\mathbf{1}^{T}\|\le3$. In addition, $\|I-\frac{1}{n}\mathbf{1}\mathbf{1}^{T}\| = 1$.

Next, we present three inequalities in Lemmas \ref{lemma:first_inequality}, \ref{lemma:second_inequality}, and \ref{lemma:third_inequality}, which will assist in deriving our main results. The proofs of these lemmas are provided in Appendix \ref{proof_lemmas}.

\begin{lemma}\label{lemma:first_inequality}
When $\alpha_{k+1}\le\frac{2}{\mu+L}$ for all $k\ge0$, the following inequality holds for every $k\ge0$:
	\begin{equation}
	\begin{aligned}
&\quad \mathbb{E}\left[ \|\bar{\bm{x}}(k+1) - \bm{x}^{*}\|^{2} \right] \\
&\le (1-\alpha_{k+1}\mu) \mathbb{E}\left[ \|\bar{\bm{x}}(k) - \bm{x}^{*}\|^{2} \right] \\
&\quad + \frac{L^{2}\alpha_{k+1}}{n\mu} \mathbb{E}\left[  \|\mathbf{x}(k) - \mathbf{1}\bar{\bm{x}}(k)\|^{2} \right] \\
&\quad + 2p(1+2L\alpha_{k+1}+\frac{L^{2}\alpha_{k+1}}{\mu})\nu_{k+1}^{2}.
	\end{aligned}
	\end{equation}
\end{lemma}

\begin{lemma}\label{lemma:second_inequality}
The following inequality holds for all $k\ge1$:
	\begin{equation}
	\begin{aligned}
&\quad\mathbb{E}\left[\|\mathbf{x}(k+1) - \mathbf{1}\bar{\bm{x}}(k+1)\|^{2}\right] \\
&\le \left(\frac{2+\sigma^{2}}{3}+\frac{4(2+\sigma^{2})L^{2}\alpha_{k+1}^{2}}{1-\sigma^{2}}\right)\mathbb{E}\left[\|\mathbf{x}(k) - \mathbf{1}\bar{\bm{x}}(k)\|^{2}\right] \\
&\quad + \frac{4n(2+\sigma^{2})L^{2}\alpha_{k+1}^{2}}{1-\sigma^{2}}\mathbb{E}\left[\left\|\bar{\bm{x}}(k)-\bm{x}^{*}\right\|^{2}\right] \\
&\quad + \frac{(2+\sigma^{2})(q_{1}^{2}+2L^{2}\alpha_{k+1}^{2})}{1-\sigma^{2}}\mathbb{E}\left[\|\mathbf{x}(k)-\mathbf{x}(k-1)\|^{2} \right] \\
&\quad + 2np\left(9+6L\alpha_{k+1}+\frac{18L^{2}\alpha_{k+1}^{2}}{1-\sigma^{2}}\right)\nu_{k+1}^{2} \\
&\quad + 2np\Big(15+6L\alpha_{k}+4L\alpha_{k+1}+npL\alpha_{k+1}+2L^{2}\alpha_{k}\alpha_{k+1}\\
&\qquad +\frac{6L^{2}(7+2L\alpha_{k})\alpha_{k+1}^{2}}{1-\sigma^{2}}\Big)\nu_{k}^{2},
%12np\left(4+3L\alpha_{k}+\frac{2(5+L\alpha_{k})L^{2}\alpha_{k+1}^{2}}{1-\sigma^{2}}\right)\nu_{k}^{2}
	\end{aligned}
	\end{equation}
where $\sigma \triangleq \max_{k\ge1}\{\|\tilde{W}(k) - \frac{1}{n}\mathbf{1}\mathbf{1}^{T}\|\} < 1$.
\end{lemma}

\begin{lemma}\label{lemma:third_inequality}
The following inequality holds for all $k\ge1$:
	\begin{equation}
	\begin{aligned}
&\quad \mathbb{E}\left[\|\mathbf{x}(k+1) - \mathbf{x}(k)\|^{2}\right] \\
&\le \Big( \frac{3+\sigma^{2}}{4} + \frac{3(3+\sigma^{2})L^{2}\alpha_{k+1}^{2}}{1-\sigma^{2}} \Big) \mathbb{E}\left[\|\mathbf{x}(k) - \mathbf{x}(k-1)\|^{2}\right] \\
&\quad + \frac{4n(3+\sigma^{2})L^{2}\alpha_{k+1}^{2}}{1-\sigma^{2}} \mathbb{E}\left[\|\bar{\bm{x}}(k) - \bm{x}^{*}\|^{2}\right] \\
&\quad + \frac{(3+\sigma^{2})(\|W-I\|^{2} + 4L^{2}\alpha_{k+1}^{2})}{1-\sigma^{2}} \mathbb{E}\left[\|\mathbf{x}(k) - \mathbf{1}\bar{\bm{x}}(k)\|^{2}\right] \\
&\quad + 2np\Big(9+6L\alpha_{k+1}+\frac{28L^{2}\alpha_{k+1}^{2}}{1-\sigma^{2}}\Big)\nu_{k+1}^{2} \\
&\quad + 4np\Big(19+6L\alpha_{k}+5L\alpha_{k+1}+npL\alpha_{k+1}+2L^{2}\alpha_{k}\alpha_{k+1}\\
&\qquad +\frac{6L^{2}(7+2L\alpha_{k})\alpha_{k+1}^{2}}{1-\sigma^{2}}\Big)\nu_{k}^{2}.
%&\quad + 2np\left(47 + 32L\alpha_{k} + \frac{8(14+3L\alpha_{k})L^{2}\alpha_{k}^{2}}{1-\sigma^{2}}\right)\nu_{k}^{2}.
	\end{aligned}
	\end{equation}
where $\sigma$ is defined as in Lemma \ref{lemma:second_inequality}.
\end{lemma}

Let $s_{1}(k) \triangleq \mathbb{E}\left[\|\bar{\bm{x}}(k)-\bm{x}^*\|^{2}\right]$, $s_{2}(k) \triangleq \mathbb{E}\left[\|\mathbf{x}(k)-\mathbf{1}\bar{\bm{x}}(k)\|^{2}\right]$, and $s_{3}(k) \triangleq \mathbb{E}\left[\|\mathbf{x}(k)-\mathbf{x}(k-1)\|^{2}\right]$. In light of Lemmas \ref{lemma:first_inequality}, \ref{lemma:second_inequality}, and \ref{lemma:third_inequality}, we have
\begin{equation}\label{equation_linear_inequality}
\bm{s}(k+1) \le A_{k}\bm{s}(k) + \bm{b}_{k},
\end{equation}
where $\bm{s}(k)$, $A_{k}$, and $\bm{b}_{k}$ are defined in (\ref{eq:Ak}) and (\ref{eq:bk}).

\begin{figure*}[!t]
\centering
\begin{equation}\label{eq:Ak}
\bm{s}(k) \triangleq
\left[ \begin{matrix}
s_{1}(k) \\
s_{2}(k) \\
s_{3}(k)
\end{matrix} \right],\ 
A_{k} \triangleq 
\left[ \begin{matrix}
1-\alpha_{k+1}\mu & \frac{L^{2}\alpha_{k+1}}{n\mu} & 0\\
\frac{4n(2+\sigma^{2})L^{2}\alpha_{k+1}^{2}}{1-\sigma^{2}} & \frac{2+\sigma^{2}}{3}+\frac{4(2+\sigma^{2})L^{2}\alpha_{k+1}^{2}}{1-\sigma^{2}} & \frac{(2+\sigma^{2})(q_{1}^{2}+2L^{2}\alpha_{k+1}^{2})}{1-\sigma^{2}} \\
\frac{4n(3+\sigma^{2})L^{2}\alpha_{k+1}^{2}}{1-\sigma^{2}} & \frac{(3+\sigma^{2})(\|W-I\|^{2} + 4L^{2}\alpha_{k+1}^{2})}{1-\sigma^{2}} & \frac{3+\sigma^{2}}{4} + \frac{3(3+\sigma^{2})L^{2}\alpha_{k+1}^{2}}{1-\sigma^{2}}
\end{matrix} \right],
\end{equation}
\begin{equation}\label{eq:bk}
\bm{b}_{k} \triangleq
\left[ \begin{matrix}
2p(1+2L\alpha_{k+1}+\frac{L^{2}\alpha_{k+1}}{\mu})\nu_{k+1}^{2}\\
2np\big(9+6L\alpha_{k+1}+\frac{18L^{2}\alpha_{k+1}^{2}}{1-\sigma^{2}}\big)\nu_{k+1}^{2} + 2np\big(15+6L\alpha_{k}+4L\alpha_{k+1}+npL\alpha_{k+1}+2L^{2}\alpha_{k}\alpha_{k+1}+\frac{6L^{2}(7+2L\alpha_{k})\alpha_{k+1}^{2}}{1-\sigma^{2}}\big)\nu_{k}^{2} \\
2np\big(9+6L\alpha_{k+1}+\frac{28L^{2}\alpha_{k+1}^{2}}{1-\sigma^{2}}\big)\nu_{k+1}^{2} + 4np\big(19+6L\alpha_{k}+5L\alpha_{k+1}+npL\alpha_{k+1}+2L^{2}\alpha_{k}\alpha_{k+1}+\frac{6L^{2}(7+2L\alpha_{k})\alpha_{k+1}^{2}}{1-\sigma^{2}}\big)\nu_{k}^{2}
\end{matrix} \right].
\end{equation}
\end{figure*}

We now present the convergence result and optimization accuracy of our proposed algorithm in the following theorem.
\begin{theorem}\label{theorem:convergence}
Consider the distributed optimization problem (\ref{problem}) with Assumptions \ref{assump:smooth_convex}, \ref{assump:W}, \ref{assumption_noise}. If the stepsize decay parameter $q_{1}$ in (\ref{eq_alpha_nu}) satisfy $q_{1}^{2} \le \frac{(1-\sigma^{2})^{4}}{48(\theta+1)(2+\sigma^{2})(3+\sigma^{2})\|W-I\|^{2}}$, where $\sigma < 1$ is defined in Lemma \ref{lemma:second_inequality} and $\theta > 1$ is a constant, then the sequence $\{\mathbf{x}(k)\}$ generated by Algorithm \ref{alg:alg_private} converges, i.e., $\lim_{k\to\infty}\mathbb{E}\left[\left\|\mathbf{x}(k)-\mathbf{1}\bar{\bm{x}}(k)\right\|^{2}\right] = 0$. The optimization accuracy $d \triangleq \lim_{k\to\infty}\mathbb{E}\left[\left\|\bar{\bm{x}}(k)-\bm{x}^*\right\|^{2}\right]$ is bounded by
	\begin{equation}\label{eq:accuracy_bound}
	\begin{aligned}
	&e^{-\frac{\mu\gamma}{1-q_{1}}}c_{1} + \frac{\gamma c_{2}L^{2}}{n\mu(1-q_{1})} + \frac{2p\delta^{2}\gamma^{2}q_{2}^{2}}{\epsilon^{2}(q_{2}-q_{1})^{2}(1-q_{2}^{2})} \\
	&\quad + \Big(4pL + \frac{2pL^{2}}{\mu}\Big)\frac{\delta^{2}q_{2}^{2}\gamma^{3}}{\epsilon^{2}(q_{2}-q_{1})^{2}(1-q_{1}q_{2}^{2})},
	\end{aligned}
	\end{equation}
where $c_{1},c_{2}>0$ are constants.
\end{theorem}
%$e_{4}, e_{5}, e_{6}>0$ are defined in (\ref{eq_gamma_22}) and (\ref{eq_gamma_33}). 
\begin{proof}
We first prove that under appropriate stepsize parameters, $\rho(A_{k})\le1$ for any $k\ge1$. Then, we prove the boundedness of $\bm{s}(k)$ and finally obtain the desired results. Considering that $\alpha_{k}\to0$ as $k\to\infty$, we analyze the spectral radius of $A_{k}$ both for $k<\infty$ and $k\to\infty$, respectively.

For the case of $k<\infty$, we demonstrate that $\rho(A_{k})$ is strictly less than 1. By Lemma \ref{lemma_rho}, it suffices to ensure that $[A_{k}]_{11}, [A_{k}]_{22}, [A_{k}]_{33}<1$ and $\det(I-A_{k})>0$. If $\alpha_{k+1} \le \frac{2}{\mu + L}$, then $0 \le [A_{k}]_{11} = 1-\alpha_{k+1}\mu < 1$. To ensure $[A_{k}]_{22}<1$ and $[A_{k}]_{33}<1$, it is sufficient that $\alpha_{k+1}^{2} < \frac{(1-\sigma^{2})^{2}}{12(3+\sigma^{2})L^{2}}$. To satisfy the more stringent condition $\det(I-A_{k})>0$, a tighter bound on $\alpha_{k+1}$ is required. If $\alpha_{k+1}$ further satisfies
	\begin{equation}\label{eq_alpha_11}
	\begin{aligned}
	\alpha_{k+1}^{2} &\le \frac{(1-\sigma^{2})^{2}}{24(3+\sigma^{2})L^{2}},
	\end{aligned}
	\end{equation}
then $1-[A_{k}]_{22} \ge \frac{1-\sigma^{2}}{6}$ and $1-[A_{k}]_{33} \ge \frac{1-\sigma^{2}}{8}$. Suppose $\alpha_{k+1}$ also satisfies the following relations:
%	\begin{subequations}\label{seq_phi}
%	\begin{gather}
%	[A_{k}]_{12}[A_{k}]_{21} \le \frac{1}{\theta+1}\left(1-[A_{k}]_{11}\right)\left(1-[A_{k}]_{22}\right), \label{seq_phi_1221}\\
%	[A_{k}]_{23}[A_{k}]_{32} \le \frac{1}{\theta+1}\left(1-[A_{k}]_{22}\right)\left(1-[A_{k}]_{33}\right), \label{seq_phi_2332}\\
%	[A_{k}]_{12}[A_{k}]_{23}[A_{k}]_{31} \le \frac{1}{\theta+2}\Big(\left(1-[A_{k}]_{11}\right)\left(1-[A_{k}]_{22}\right)\left(1-[A_{k}]_{33}\right) - [A_{k}]_{12}[A_{k}]_{21}(1-[A_{k}]_{33}) - [A_{k}]_{23}[A_{k}]_{32}(1-[A_{k}]_{11})\Big), \label{seq_phi_122331}
%	\end{gather}
%	\end{subequations}
	\begin{subequations}\label{seq_phi}
	\begin{align}
	&\quad [A_{k}]_{12}[A_{k}]_{21} \le \frac{1}{\theta+1}\left(1-[A_{k}]_{11}\right)\left(1-[A_{k}]_{22}\right), \label{seq_phi_1221}\\
	&\quad [A_{k}]_{23}[A_{k}]_{32} \le \frac{1}{\theta+1}\left(1-[A_{k}]_{22}\right)\left(1-[A_{k}]_{33}\right), \label{seq_phi_2332}\\
	&\quad [A_{k}]_{12}[A_{k}]_{23}[A_{k}]_{31} \nonumber\\
	&\le \frac{1}{\theta+2}\big(\left(1-[A_{k}]_{11}\right)\left(1-[A_{k}]_{22}\right)\left(1-[A_{k}]_{33}\right) \label{seq_phi_122331}\\
	&\quad - [A_{k}]_{12}[A_{k}]_{21}(1-[A_{k}]_{33}) - [A_{k}]_{23}[A_{k}]_{32}(1-[A_{k}]_{11})\big), \nonumber
	\end{align}
	\end{subequations}
where $\theta > 1$. We then have
	\begin{equation}
	\begin{aligned}
	&\quad \det(I-A_{k}) \\
	&= (1-[A_{k}]_{11})(1-[A_{k}]_{22})(1-[A_{k}]_{33}) \\
	&\quad- [A_{k}]_{12}[A_{k}]_{21}(1-[A_{k}]_{33}) - [A_{k}]_{23}[A_{k}]_{32}(1-[A_{k}]_{11}) \\
	&\quad - [A_{k}]_{12}[A_{k}]_{23}[A_{k}]_{31} \\
	&\ge \frac{\theta+1}{\theta+2}\big((1-[A_{k}]_{11})(1-[A_{k}]_{22})(1-[A_{k}]_{33}) \\
	&\qquad - [A_{k}]_{12}[A_{k}]_{21}(1-[A_{k}]_{33}) - [A_{k}]_{23}[A_{k}]_{32}(1-[A_{k}]_{11})\big) \\
	&\ge \frac{\theta-1}{\theta+2}(1-[A_{k}]_{11})(1-[A_{k}]_{22})(1-[A_{k}]_{33}) \\
	&> 0.
	\end{aligned}
	\end{equation}
Therefore, we obtain $\rho(A_{k}) < 1$ for any $1\le k<\infty$.

For the case of $k \to \infty$,
	\begin{equation}\label{equation_A_infty}
A_{\infty} = 
\left[ \begin{matrix}
1 & 0 \\
0 & \tilde{A}
	\end{matrix} \right],
	\tilde{A} = 
\left[ \begin{matrix}
\frac{2+\sigma^{2}}{3} & \frac{(2+\sigma^{2})q_{1}^{2}}{1-\sigma^{2}} \\
\frac{(3+\sigma^{2})\|W-I\|^{2}}{1-\sigma^{2}} & \frac{3+\sigma^{2}}{4}
	\end{matrix} \right].
	\end{equation}
If $q_{1}$ satisfies 
	\begin{equation}\label{eq_q1_1}
	\begin{aligned}
	q_{1}^{2} < \frac{(1-\sigma^{2})^{4}}{12(2+\sigma^{2})(3+\sigma^{2})\|W-I\|^{2}},
	\end{aligned}
	\end{equation}
then $\det(I-\tilde{A}) = (1-\tilde{A}_{11})(1-\tilde{A}_{22}) - \tilde{A}_{12}\tilde{A}_{21} > 0$. In light of Lemma \ref{lemma_rho}, we have $\rho(\tilde{A})<1$.

We now show that $\bm{s}(k)$ is bounded for all $k\ge1$. By Lemma 5.6.10 in \cite{horn2012matrix}, for any $\varepsilon>0$, there exists an induced matrix norm $\|\cdot\|_{s}$ such that
	\begin{equation}\label{eq:induced_norm}
	\|A_{k}\|_{s}\le\left\{
\begin{aligned}
&\rho(A_{k}) + \varepsilon,\ \ k<\infty, \\
&\max\{1, \rho(\tilde{A})+\varepsilon\} ,\ \ k\to\infty.
\end{aligned}
\right.
	\end{equation}
%when $\varepsilon$ is small enough, we have $\|A_{k}\|_{s} \le 1$ for all $k\ge1$.
Since $\rho(A_{k}) < 1$ for $k<\infty$ and $\rho(\tilde{A}) < 1$, we can choose $\varepsilon$ small enough such that $\|A_{k}\|_{s} \le 1$ for all $k\ge1$. Then, by recursively invoking (\ref{equation_linear_inequality}), it follows that
	\begin{equation}\label{equation_s_bounded}
	\begin{aligned}
	\|\bm{s}(k+1)\|_{s} &\le \|A_{k}A_{k-1}\cdots A_{1}\bm{s}(1)\|_{s} + \sum_{i=1}^{k}\Big\|\Big(\prod_{j=i+1}^{k}A_{j}\Big)\bm{b}_{i}\Big\|_{s} \\
	&\le \|\bm{s}(1)\|_{s} + \sum_{i=1}^{\infty}\|\bm{b}_{i}\|_{s}.
	\end{aligned}
	\end{equation}
Since $\sum_{i=1}^{\infty}[\bm{b}_{i}]_{j} < \infty$ for $j = 1,2,3$ and by the equivalence of norms, we have $\sum_{i=1}^{\infty}\|\bm{b}_{i}\|_{s} <\infty$. Thus, $\bm{s}(k)$ is bounded for all $k\ge1$.  

%将它们代入到1中，我们有
Let $s_{1}(k)\le c_{1}$, $s_{2}(k) \le c_{2}$, and $s_{3}(k) \le c_{3}$, where $c_{1}, c_{2}, c_{3}>0$ are some constants. Substituting these bounds into (\ref{equation_linear_inequality}), we obtain
	\begin{subequations}
	\begin{align}
	s_{1}(k+1) &\le (1-\alpha_{k+1}\mu)s_{1}(k) + c_{2}[A_{k}]_{12} + [\bm{b}_{k}]_{1}, \label{ineq_s1}\\
	\tilde{\bm{s}}(k+1) &\le \tilde{A}\tilde{\bm{s}}(k) + \tilde{\bm{b}}_{k},
	\end{align}
	\end{subequations}
where $\tilde{\bm{s}}(k) \triangleq \left[s_{2}(k), s_{3}(k)\right]^{T}$, $\tilde{A}$ is defined as in (\ref{equation_A_infty}), and $[\tilde{\bm{b}}_{k}]_{1} \triangleq \frac{2(2+\sigma^{2})(2nc_{1}+2c_{2}+c_{3})L^{2}\alpha_{k+1}^{2}}{1-\sigma^{2}} + [\bm{b}_{k}]_{2}$, $[\tilde{\bm{b}}_{k}]_{2} \triangleq \frac{(3+\sigma^{2})(4nc_{1}+4c_{2}+3c_{3})L^{2}\alpha_{k+1}^{2}}{1-\sigma^{2}} + [\bm{b}_{k}]_{3}$.

%$\tilde{\bm{b}}_{k}=\left[\frac{2(2+\sigma^{2})(2ne_{1}+2e_{2}+e_{3})L^{2}\alpha_{k+1}^{2}}{1-\sigma^{2}} + [\bm{b}_{k}]_{2}, \frac{(3+\sigma^{2})(4ne_{1}+4e_{2}+3e_{3})L^{2}\alpha_{k+1}^{2}}{1-\sigma^{2}} + [\bm{b}_{k}]_{3}\right]^{T}$.

Since $\rho(\tilde{A})<1$, it follows that $\tilde{\bm{s}}(k)$ converges and satisfies $\tilde{\bm{s}}(\infty) \le (I-\tilde{A})^{-1}\tilde{\bm{b}}_{\infty}$. Given that $\tilde{\bm{b}}_{\infty} = 0$, it follows that $\lim_{k\to\infty}\mathbb{E}\left[\|\mathbf{x}(k)-\mathbf{1}\bar{\bm{x}}(k)\|^{2}\right] = 0$. For the optimization accuracy, by (\ref{ineq_s1}) and the parameter settings in (\ref{eq_alpha_nu}), we obtain
	\begin{equation}
	\begin{aligned}
	&\quad \lim_{k\to\infty}\mathbb{E}\left[\left\|\bar{\bm{x}}(k+1)-\bm{x}^*\right\|^{2}\right] \\
	&\le \prod_{i=0}^{\infty}(1-\alpha_{i+1}\mu)s_{1}(0) + \sum_{i=0}^{\infty}\prod_{j=i+1}^{\infty}(1-\alpha_{j+1}\mu)\Big(\frac{c_{2}L^{2}\alpha_{i+1}}{n\mu} \\
	&\qquad + 2p(1+2L\alpha_{i+1}+\frac{L^{2}\alpha_{i+1}}{\mu})\nu_{i+1}^{2}\Big) \\
	&\le e^{-\mu\sum_{i=1}^{\infty}\alpha_{i}}c_{1} + \frac{c_{2}L^{2}}{n\mu}\sum_{i=1}^{\infty}\alpha_{i} + 2p\sum_{i=1}^{\infty}\nu_{i}^{2} \\
	&\quad + \Big(4pL + \frac{2pL^{2}}{\mu}\Big)\sum_{i=1}^{\infty}\alpha_{i}\nu_{i}^{2} \\
	&\le e^{-\frac{\mu\gamma}{1-q_{1}}}c_{1} + \frac{\gamma c_{2}L^{2}}{n\mu(1-q_{1})} + \frac{2p\delta^{2}\gamma^{2}q_{2}^{2}}{\epsilon^{2}(q_{2}-q_{1})^{2}(1-q_{2}^{2})} \\
	&\quad + \Big(4pL + \frac{2pL^{2}}{\mu}\Big)\frac{\delta^{2}q_{2}^{2}\gamma^{3}}{\epsilon^{2}(q_{2}-q_{1})^{2}(1-q_{1}q_{2}^{2})},
	\end{aligned}
	\end{equation}
%其中在第二个不等式中我们使用了a。
where the second inequality holds due to $1-x \le e^{-x}$ for all $x\in\mathbb{R}$. 

It remains to show that (\ref{seq_phi}) holds.
%Based on (\ref{condition_initial_stepsize}), we have $\gamma^{2} \le \frac{\mu^{2}(1-\sigma^{2})^{2}}{24(1+\theta)(2+\sigma^{2})L^{4}q_{1}^{2}}$. Since $\alpha_{k} = \gamma q_{1}^{k-1}$, it follows that
%	\begin{equation}
%	\frac{4(2+\sigma^{2})L^{4}\alpha_{k+1}^{2}}{\mu(1-\sigma^{2})} \le \frac{\mu(1-\sigma^{2})}{6(\theta+1)},
%	\end{equation}
%which means that (\ref{seq_phi_1221}) holds. In addition, when $q_{1}^{2}\le\frac{(1-\sigma^{2})^{4}}{48(\theta+1)(2+\sigma^{2})(3+\sigma^{2})\|W-I\|^{2}}$ and $\gamma^{2} \le \frac{-e_{4} + \sqrt{e_{4}^{2} + 32e_{5}L^{4}}}{16q_{1}^{2}L^{4}}$, we obtain
%	\begin{equation}
%	\begin{aligned}
%	&\quad \frac{(2+\sigma^{2})(3+\sigma^{2})(q_{1}^{2}+2L^{2}\alpha_{k+1}^{2})(\|W-I\|^{2}+4L^{2}\alpha_{k+1}^{2})}{(1-\sigma^{2})^{2}} \\
%	&\le \frac{(1-\sigma^{2})^{2}}{48(\theta+1)},
%	\end{aligned}
%	\end{equation}
%which means that (\ref{seq_phi_2332}) holds. When $\gamma^{2} \le \frac{-1+\sqrt{1+e_{6}}}{4L^{2}}$, it follows that
%	\begin{equation}
%	\begin{aligned}
%	&\quad \frac{4L^{4}(2+\sigma^{2})(3+\sigma^{2})(q_{1}^{2}+2L^{2}\alpha_{k+1}^{2})\alpha_{k+1}^{2}}{\mu(1-\sigma^{2})^{2}} \\
%	&\le \frac{\mu(\theta-1)(1-\sigma^{2})^{2}}{48(\theta+1)(\theta+2)}.
%	\end{aligned}
%	\end{equation}
%This ensures (\ref{seq_phi_122331}) holds. Combining (\ref{eq_alpha_11}) and $\alpha_{k+1} \le \frac{2}{\mu + L}$, we complete the proof.
From (\ref{seq_phi_1221}), we require
	\begin{equation}\label{eq_alpha_22}
	\alpha_{k+1}^{2} \le \frac{\mu^{2}(1-\sigma^{2})^{2}}{24(\theta+1)(2+\sigma^{2})L^{4}}.
	\end{equation}
From (\ref{seq_phi_2332}), we need
	\begin{equation}
	\begin{aligned}
	&\quad \frac{(2+\sigma^{2})(3+\sigma^{2})(q_{1}^{2}+2L^{2}\alpha_{k+1}^{2})(\|W-I\|^{2}+4L^{2}\alpha_{k+1}^{2})}{(1-\sigma^{2})^{2}} \\
	&\le \frac{(1-\sigma^{2})^{2}}{48(\theta+1)},
	\end{aligned}
	\end{equation}
equivalently,
	\begin{align}
	&q_{1}^{2} \le \frac{(1-\sigma^{2})^{4}}{48(\theta+1)(2+\sigma^{2})(3+\sigma^{2})\|W-I\|^{2}}, \label{eq_q1_2}\\
	&\alpha_{k+1}^{2} \le \frac{-c_{4} + \sqrt{c_{4}^{2} + 32c_{5}L^{4}}}{16L^{4}}, \label{eq_alpha_33}
	\end{align}
where $c_{4} \triangleq 4L^{2}q_{1}^{2} + 2L^{2}\|W-I\|^{2}$ and $c_{5} \triangleq \frac{(1-\sigma^{2})^{4}}{48(\theta+1)(2+\sigma^{2})(3+\sigma^{2})} - q_{1}^{2}\|W-I\|^{2}$. From (\ref{seq_phi_122331}), we need 
%If $q_{1}^{2}\le\frac{(1-\sigma^{2})^{4}}{48(\theta+1)(2+\sigma^{2})(3+\sigma^{2})\|W-I\|^{2}}$, then the above inequality implies that
%	\begin{equation}\label{eq_gamma_22}
%	\gamma^{2} \le \frac{-e_{4} + \sqrt{e_{4}^{2} + 32e_{5}L^{4}}}{16q_{1}^{2}L^{4}},
%	\end{equation}
%where $e_{4} \triangleq 2L^{2}\|W-I\|^{2} + 4L^{2}q_{1}^{2}$ and $e_{5} \triangleq \frac{(1-\sigma^{2})^{4}}{48(\theta+1)(2+\sigma^{2})(3+\sigma^{2})} - q_{1}^{2}\|W-I\|^{2} \ge 0$.
	\begin{equation}
	\begin{aligned}
	&\quad \frac{4(2+\sigma^{2})(3+\sigma^{2})(q_{1}^{2}+2L^{2}\alpha_{k+1}^{2})L^{4}\alpha_{k+1}^{2}}{\mu(1-\sigma^{2})^{2}} \\
	&\le \frac{\mu(\theta-1)(1-\sigma^{2})^{2}}{48(\theta+1)(\theta+2)},
	\end{aligned}
	\end{equation}
which means that
	\begin{equation}\label{eq_alpha_44}
	\alpha_{k+1}^{2} \le \frac{(-1+\sqrt{1+c_{6}})q_{1}^{2}}{4L^{2}},
	\end{equation}
where $c_{6} \triangleq \frac{\mu^{2}(\theta-1)(1-\sigma^{2})^{4}}{24(\theta+1)(\theta+2)(2+\sigma^{2})(3+\sigma^{2})L^{2}q_{1}^{4}}$. In summary, the stepsize $\alpha_{k}$ must satisfy conditions (\ref{eq_alpha_11}), (\ref{eq_alpha_22}), (\ref{eq_alpha_33}), (\ref{eq_alpha_44}), and $\alpha_{k+1} \le \frac{2}{\mu + L}$. Since $\alpha_{k}$ decays exponentially, there always exists a $K$ such that $\alpha_{k}$ meets all the aforementioned conditions for all $k\ge K$. Let $k = k - K$. Considering (\ref{eq_q1_1}) and (\ref{eq_q1_2}), we complete the proof.
\end{proof}

In Theorem \ref{theorem:convergence}, we first establish the convergence of Algorithm \ref{alg:alg_private}, demonstrating that all agents eventually reach a consistent solution in the mean square sense. Subsequently, we derive a bound on the optimization accuracy of Algorithm \ref{alg:alg_private}. This accuracy bound depends on four parameters: $\gamma, q_{1}, q_{2}, \epsilon$. For fixed $\gamma, q_{1}, q_{2}$, the accuracy bound has the order of $\mathcal{O}(\frac{1}{\epsilon^{2}})$. It reveals the inherent trade-off between exact convergence and DP: the smaller the privacy budget $\epsilon$, the higher the privacy level, but the lower the accuracy.

\begin{remark}\label{remark: tuning}
The established bound (\ref{eq:accuracy_bound}) shows the trade-off between privacy and accuracy. Superior optimization accuracy can be achieved under a given privacy budget by solving the following optimization problem:
\begin{equation}\label{prob:accuracy}
\begin{aligned}
&\min_{\gamma, q_{1}, q_{2}} d(\gamma, q_{1}, q_{2}) \\
\text{subject to}\quad &\gamma>0, q_{1}\in(0,1), q_{2}\in(q_{1}, 1).
\end{aligned}
\end{equation}
Note that Problem (\ref{prob:accuracy}) is non-convex, which poses challenges for finding the global optimum. To address this issue, one possible approach is to utilize non-convex optimizers. Alternatively, a heuristic strategy can be employed to obtain a locally optimal solution: 1) Randomly initialize $\gamma$, $q_{1}$, and $q_{2}$ within their feasible domains. 2) Fix two parameters and optimize the remaining one to a local optimum. 3) Repeat Step 2) by iteratively optimizing different parameters several times.
\end{remark}

\section{Numerical Results}\label{section_simulation}
\begin{figure}[!t]
  \centerline{
  \includegraphics[width = \linewidth]{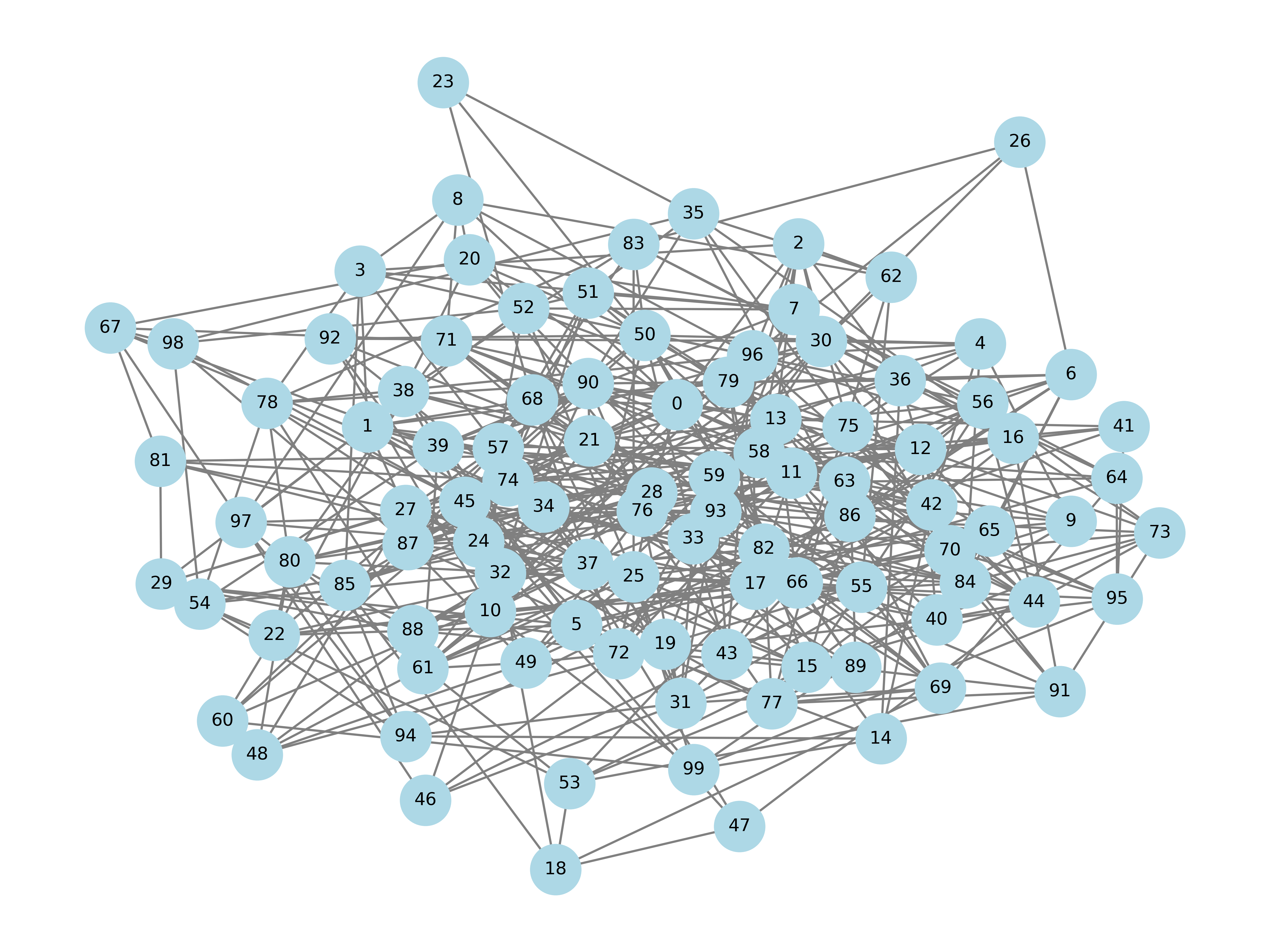}}
  \caption{The communication topology of 100 sensors.}
  \label{fig_graph}
\end{figure}

\begin{figure*}[!ht]
    \centering
    \begin{subfigure}{0.3\textwidth}% 第三张图片
        \includegraphics[width=\linewidth]{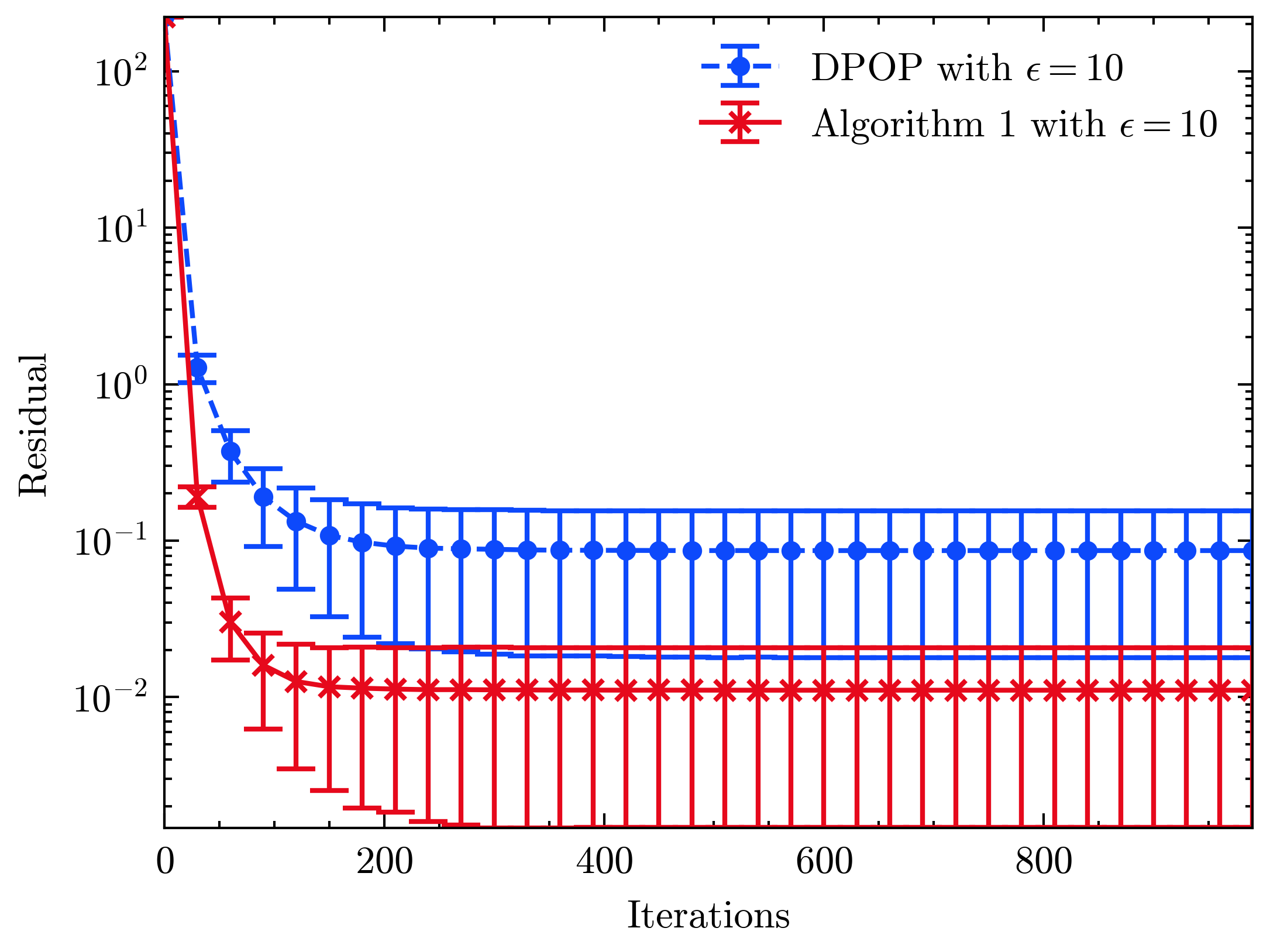}
%        \caption{第三张图片}
%        \label{fig:subfig3}
    \end{subfigure}
    \hfill
    \begin{subfigure}{0.3\textwidth}% 第二张图片
        \includegraphics[width=\linewidth]{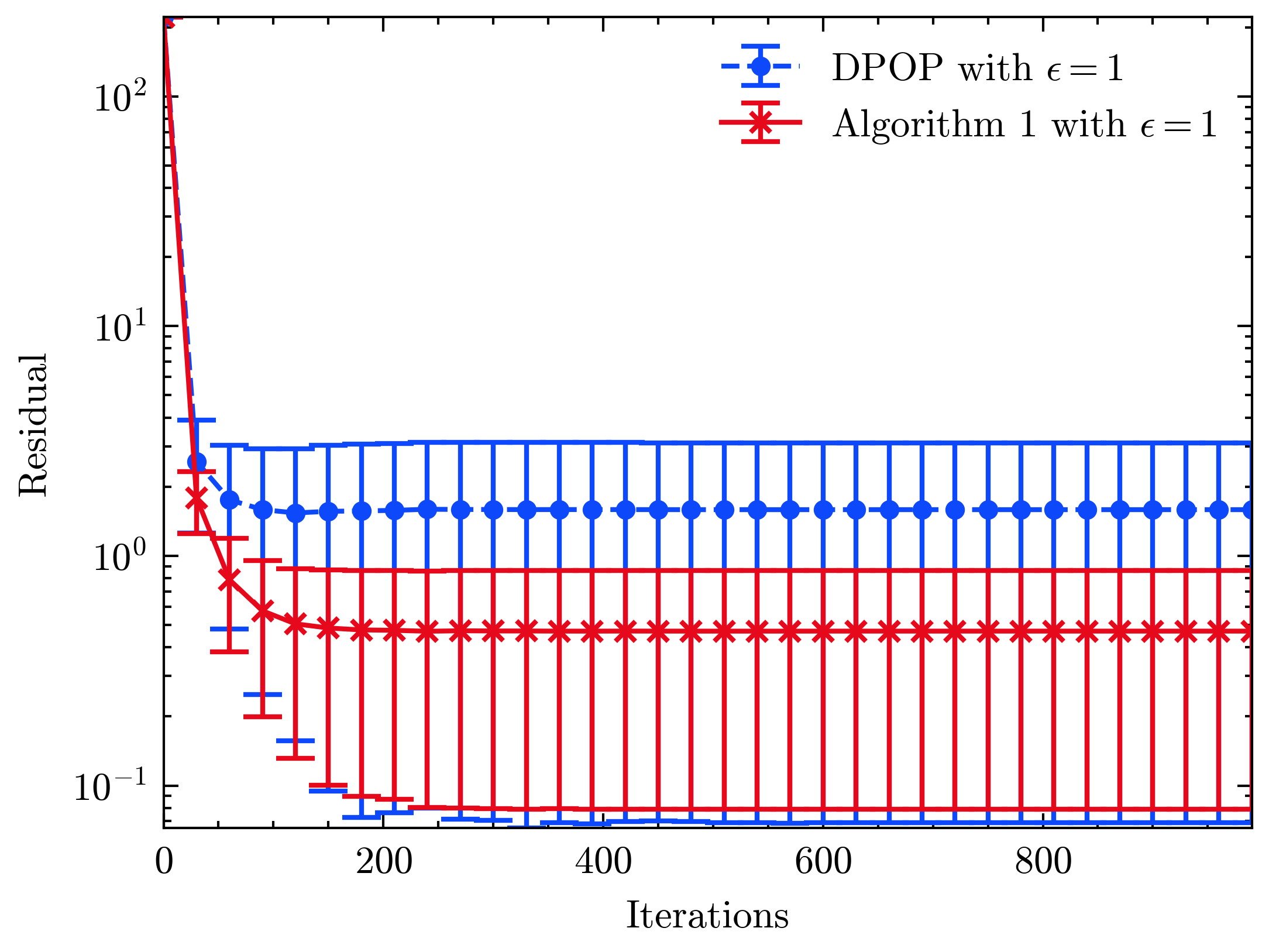}
%        \caption{第二张图片}
%        \label{fig:subfig2}
    \end{subfigure}
    \hfill
    \begin{subfigure}{0.3\textwidth}% 第一张图片
        \includegraphics[width=\linewidth]{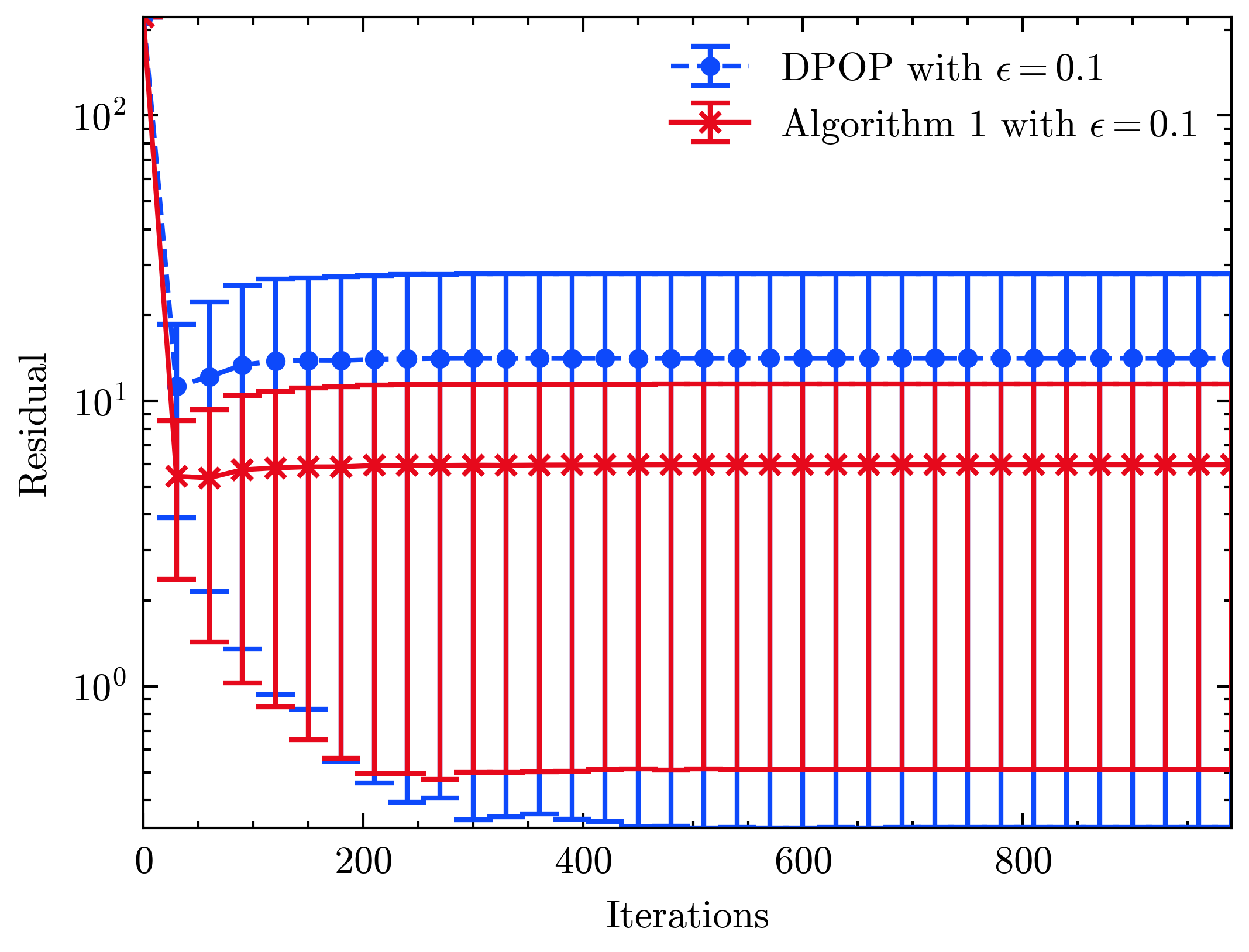}
%        \caption{第一张图片}
%        \label{fig:subfig1}
    \end{subfigure}
    \caption{Comparison of convergence between Algorithm \ref{alg:alg_private} and the DPOP method \cite{huang2015differentially} under the same privacy budget in distributed estimation problems.}
    \label{fig_residual_dpop}
\end{figure*}

\begin{figure*}[!t]
    \centering
    \begin{subfigure}{0.3\textwidth}% 第三张图片
        \includegraphics[width=\linewidth]{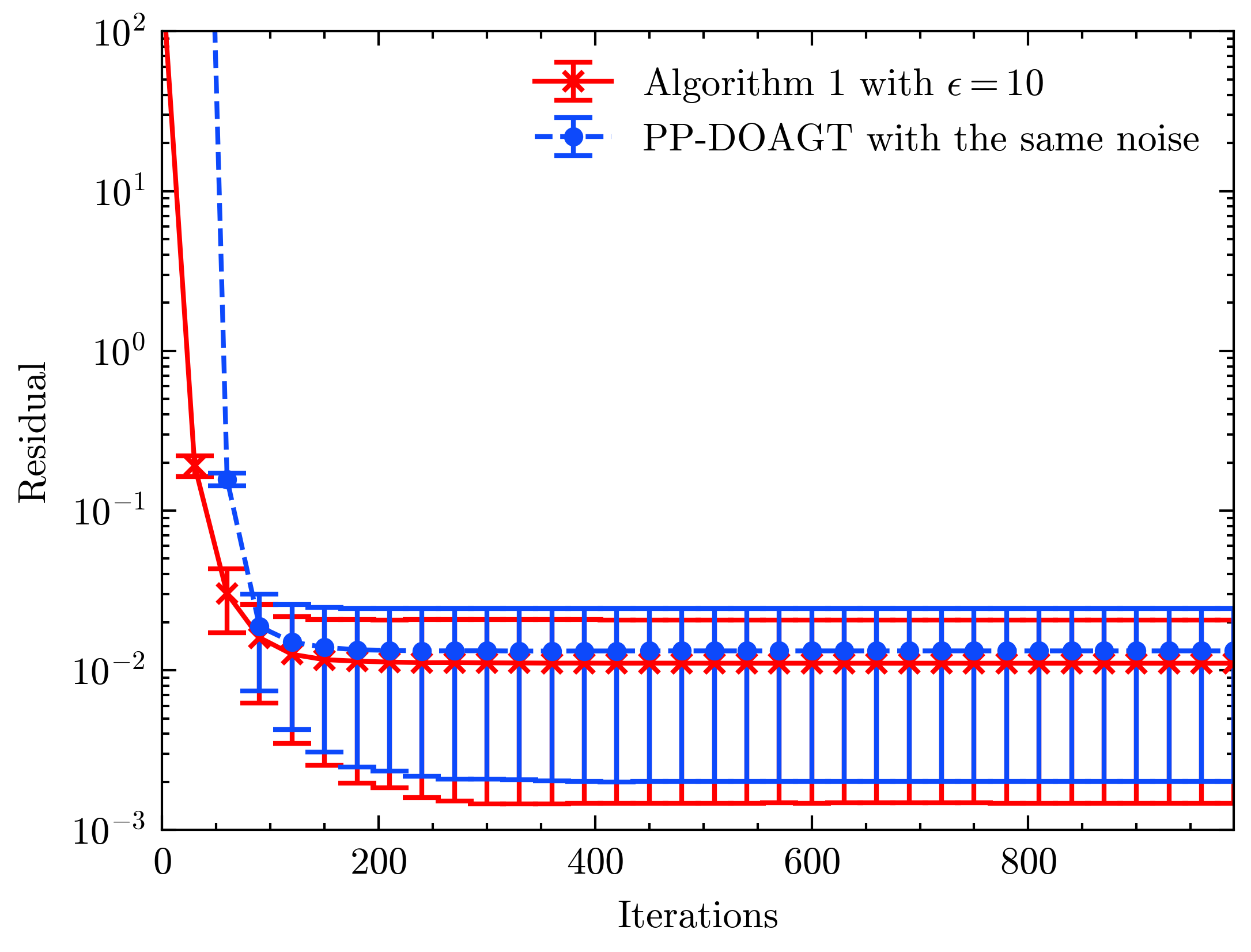}
%        \caption{第三张图片}
%        \label{fig:subfig3}
    \end{subfigure}
    \hfill
    \begin{subfigure}{0.3\textwidth}% 第二张图片
        \includegraphics[width=\linewidth]{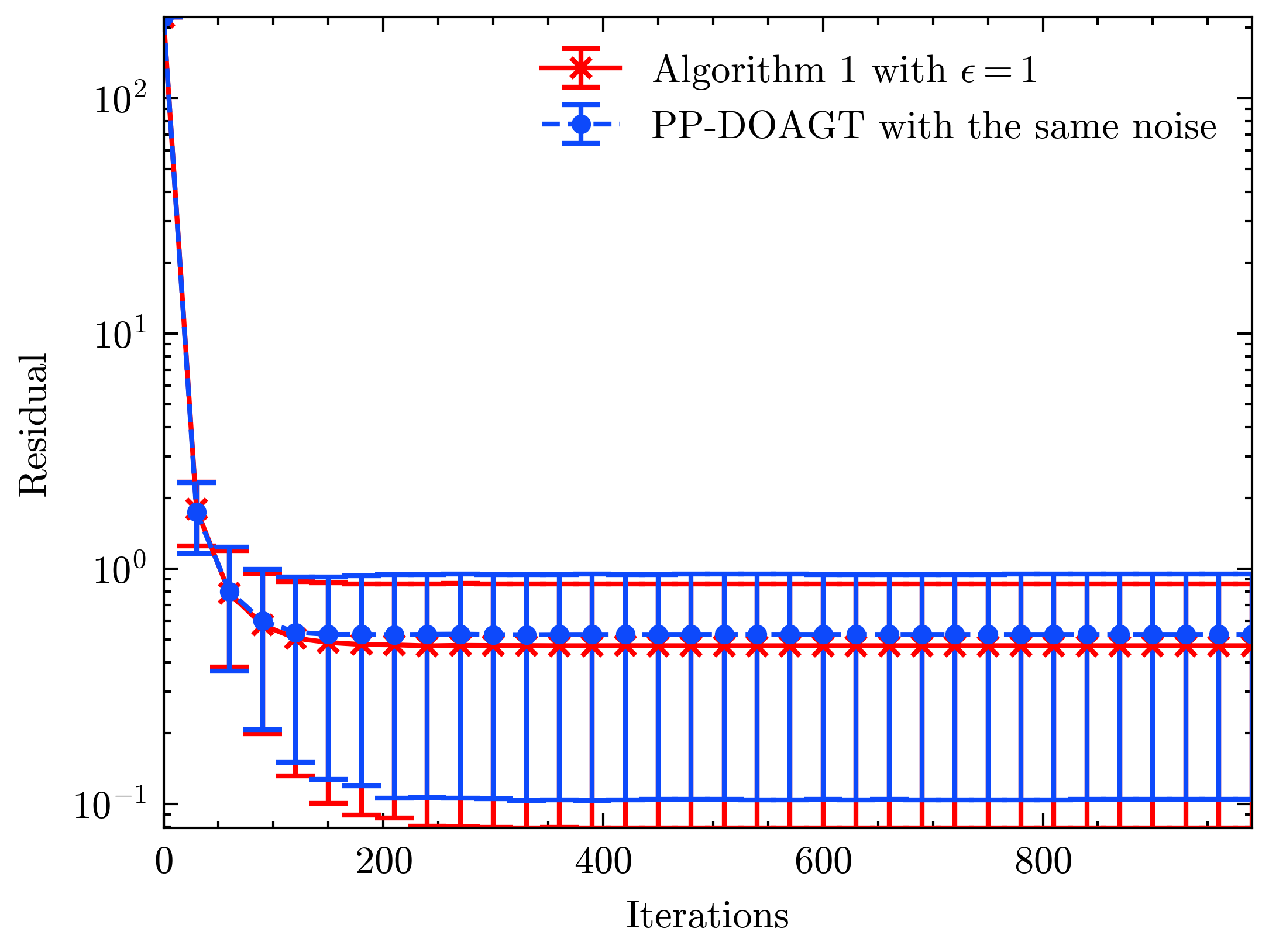}
%        \caption{第二张图片}
%        \label{fig:subfig2}
    \end{subfigure}
    \hfill
    \begin{subfigure}{0.3\textwidth}% 第一张图片
        \includegraphics[width=\linewidth]{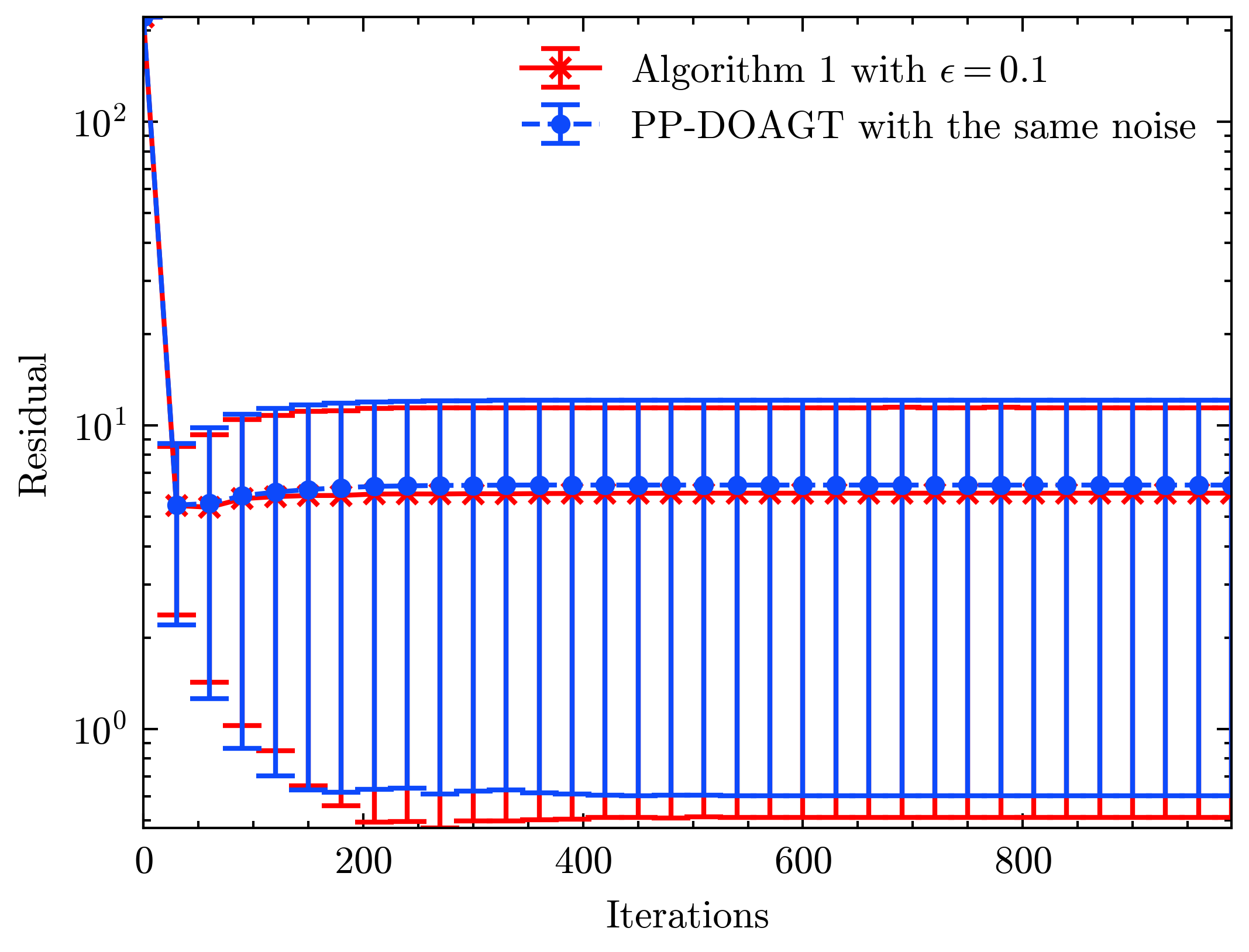}
%        \caption{第一张图片}
%        \label{fig:subfig1}
    \end{subfigure}
    \caption{Comparison of convergence between Algorithm \ref{alg:alg_private} and the PP-DOAGT method \cite{huang2024differential} under the same noise in distributed estimation problems.}
    \label{fig_residual_ppdoagt}
\end{figure*}

\begin{figure*}[!t]
    \centering
    \begin{subfigure}{0.3\textwidth}% 第三张图片
        \includegraphics[width=\linewidth]{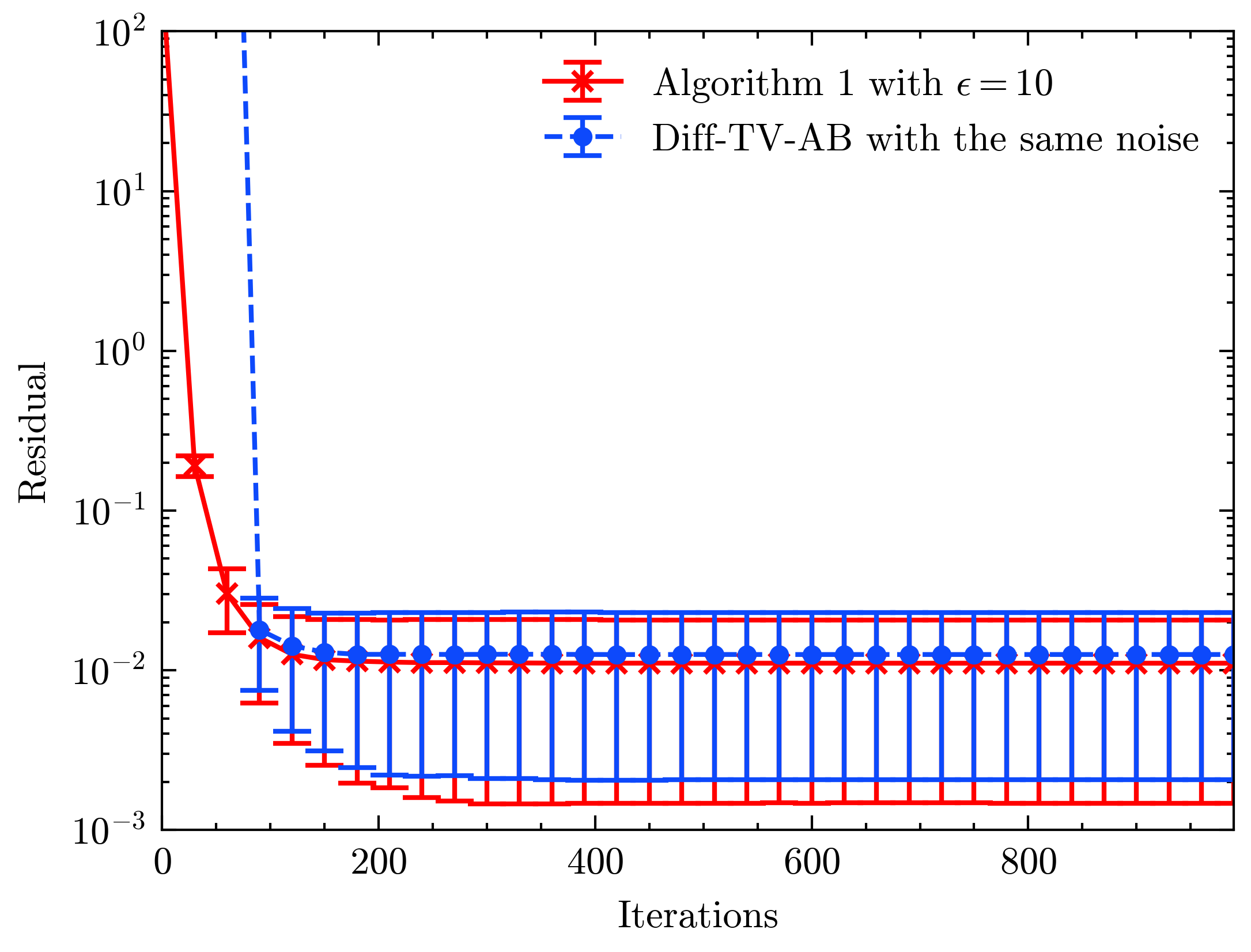}
    \end{subfigure}
    \hfill
    \begin{subfigure}{0.3\textwidth}% 第二张图片
        \includegraphics[width=\linewidth]{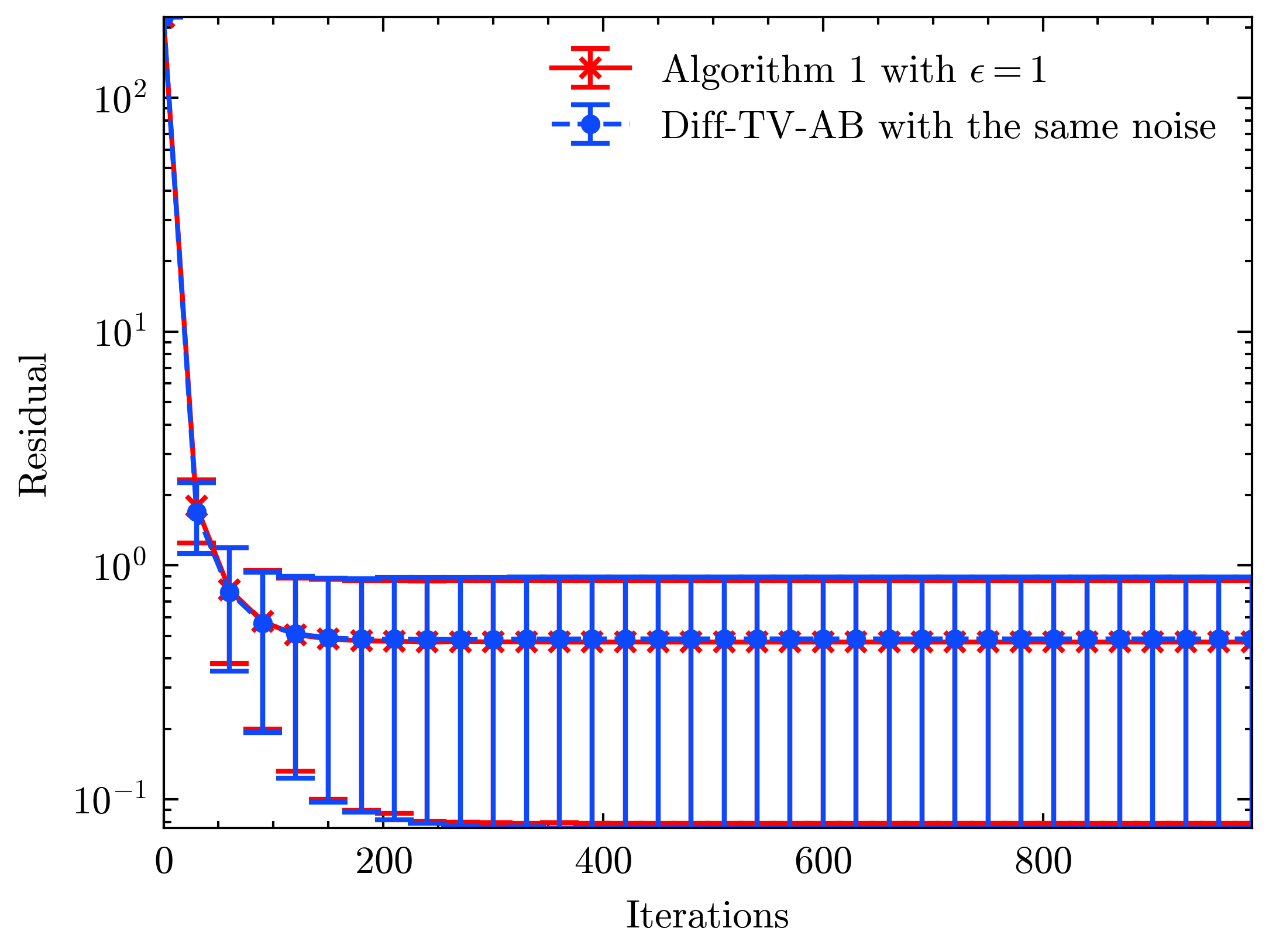}
    \end{subfigure}
    \hfill
    \begin{subfigure}{0.3\textwidth}% 第一张图片
        \includegraphics[width=\linewidth]{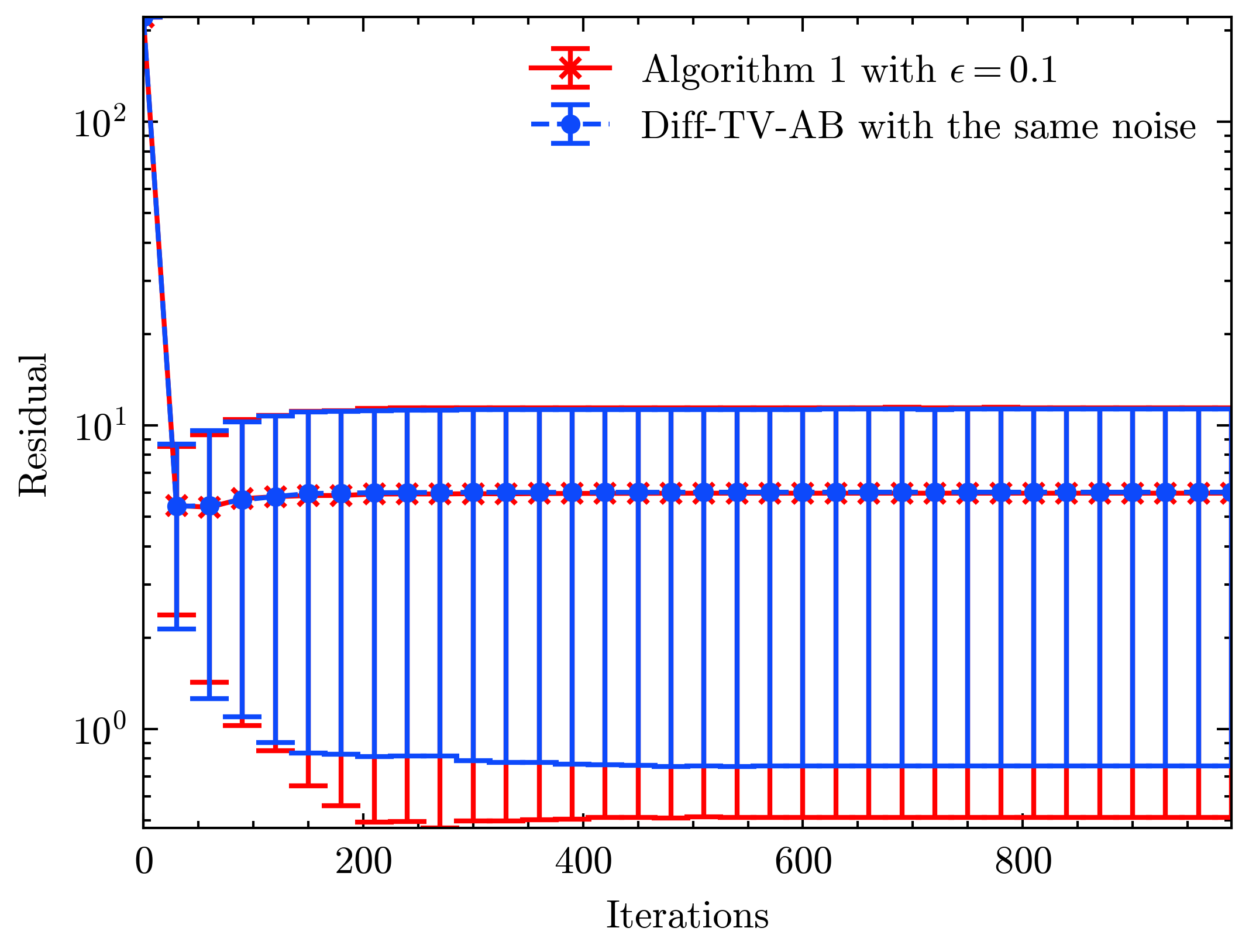}
    \end{subfigure}
    \caption{Comparison of convergence between Algorithm \ref{alg:alg_private} and the Diff-TV-AB method \cite{Yang2025Differentially} under the same noise in distributed estimation problems.}
    \label{fig_residual_difftvab}
\end{figure*}

\subsection{Comparison of Differentially Private Methods Based on Distributed Sensor Fusion}
In this subsection, we validate the effectiveness of Algorithm \ref{alg:alg_private} by comparing it with existing differentially private methods on the canonical distributed sensor fusion problem. In this problem, all $n$ sensors collaborate to estimate an unknown variable $\bm{x}$ by solving the following optimization problem:
\begin{equation}\label{problem_sensor}
\min_{\bm{x}\in\mathbb{R}^{p}} \sum_{i=1}^{n}\left( \| \bm{v}_i - M_i \bm{x} \|^2 + \omega_i \| \bm{x} \|^2 \right),
\end{equation}
where $\bm{v}_i \in \mathbb{R}^{m}$, $M_i \in \mathbb{R}^{m \times p}$, and $\omega_i \in \mathbb{R}$ represent the observation, measurement matrix, and regularization coefficient of sensor $i$, respectively. The parameter settings are consistent with those in \cite{Liu2024CryptographyBasedPM}. We consider a multi-agent system of 100 sensors, where the communication network is generated using an Erdős-Rényi (ER) random graph model with an edge creation probability of 0.1, as illustrated in Fig. \ref{fig_graph}. The dimensions in Problem (\ref{problem_sensor}) are set as $m = 3$ and $p  = 2$. 

For Algorithm \ref{alg:alg_private}, the initial states $\bm{x}_i(0)$ are randomly sampled from a standard normal distribution. The stepsize $\alpha_k$ and noise parameter $\nu_k$ are set according to (\ref{eq_alpha_nu}), and the relevant parameters $\gamma$, $\beta$, $q_2$, and $q_1$ are tuned following the heuristic strategy described in Remark \ref{remark: tuning}. We implement Algorithm \ref{alg:alg_private} under different privacy budgets ($\epsilon \in \{0.1, 1, 10\}$). For $\epsilon = 0.1/1/10$, the parameters are set as $\gamma = 0.001/0.001/0.002$, $\beta = 1000/1000/100$, $q_{2} = 0.99/0.99/0.99$, and $q_{1} = 0.92/0.97/0.97$, respectively. For each case, we calculate the mean and standard deviation of the residuals (denoted as ${\| \mathbf{x}(k) - \mathbf{x}^*\|^2}$) over 1000 trials, with the results shown as the red line in Fig. \ref{fig_residual_dpop}. It can be observed that our algorithm quickly approaches the optimal solution under all privacy budgets. This confirms that the proposed algorithm can achieve arbitrary $\epsilon$-DP. Moreover, higher privacy levels result in lower optimization accuracy, which aligns with the conclusion of Theorem \ref{theorem:convergence}.

For comparison, we implement the DPOP method from \cite{huang2015differentially} under the same privacy budgets. The corresponding results are shown by the blue lines in Fig. \ref{fig_residual_dpop}. It can be seen that our algorithm achieves higher accuracy than the DPOP method. This is because Algorithm \ref{alg:alg_private}'s sensitivity is more compact, and the introduction of auxiliary variable $\bm{y}_{i}$ enhances its convergence. Furthermore, we compare the proposed algorithm with the latest PP-DOAGT method \cite{huang2024differential} and the Diff-TV-AB method \cite{Yang2025Differentially}. Note that, unlike our algorithm and DPOP, both PP-DOAGT and Diff-TV-AB do not provide an explicit closed-form expression that directly maps a given privacy budget to the required noise parameters. Instead, they offer an implicit and complex relationship in which the privacy budget is expressed as a function of the noise and other internal parameters (see Eq. (17) or Eq. (22) in \cite{huang2024differential} and Theorem 3 in \cite{Yang2025Differentially}). Such a relationship cannot be directly used to derive deployable noise sequences under a prescribed privacy budget. Thus, to enable a meaningful comparison, we inject the same noise-variance sequence (i.e., the same noise conditions) into all three algorithms. In addition, since our work does not require Assumption 3 of \cite{Yang2025Differentially} to achieve DP, for fairness, we replace the constant stepsize in Diff-TV-AB with a diminishing stepsize. The results are shown in Figs. \ref{fig_residual_ppdoagt} and \ref{fig_residual_difftvab}. Under higher privacy levels, with privacy budgets computed via the theoretical formulas of PP-DOAGT itself being 18446.44, 3652.67, and 365.27 (where $\delta = 1$), and those computed via the theoretical formulas of Diff-TV-AB itself being 102.40, 20.34, and 7.12 (where $\vartheta = 1$), our algorithm achieves optimization accuracy comparable to that of the PP-DOAGT and Diff-TV-AB methods. This is because our algorithm exhibits lower sensitivity than PP-DOAGT and Diff-TV-AB, as shown in Table \ref{tab: sensitivity}.

\begin{figure}[!t]
  \centering
  \includegraphics[width = 0.8\linewidth]{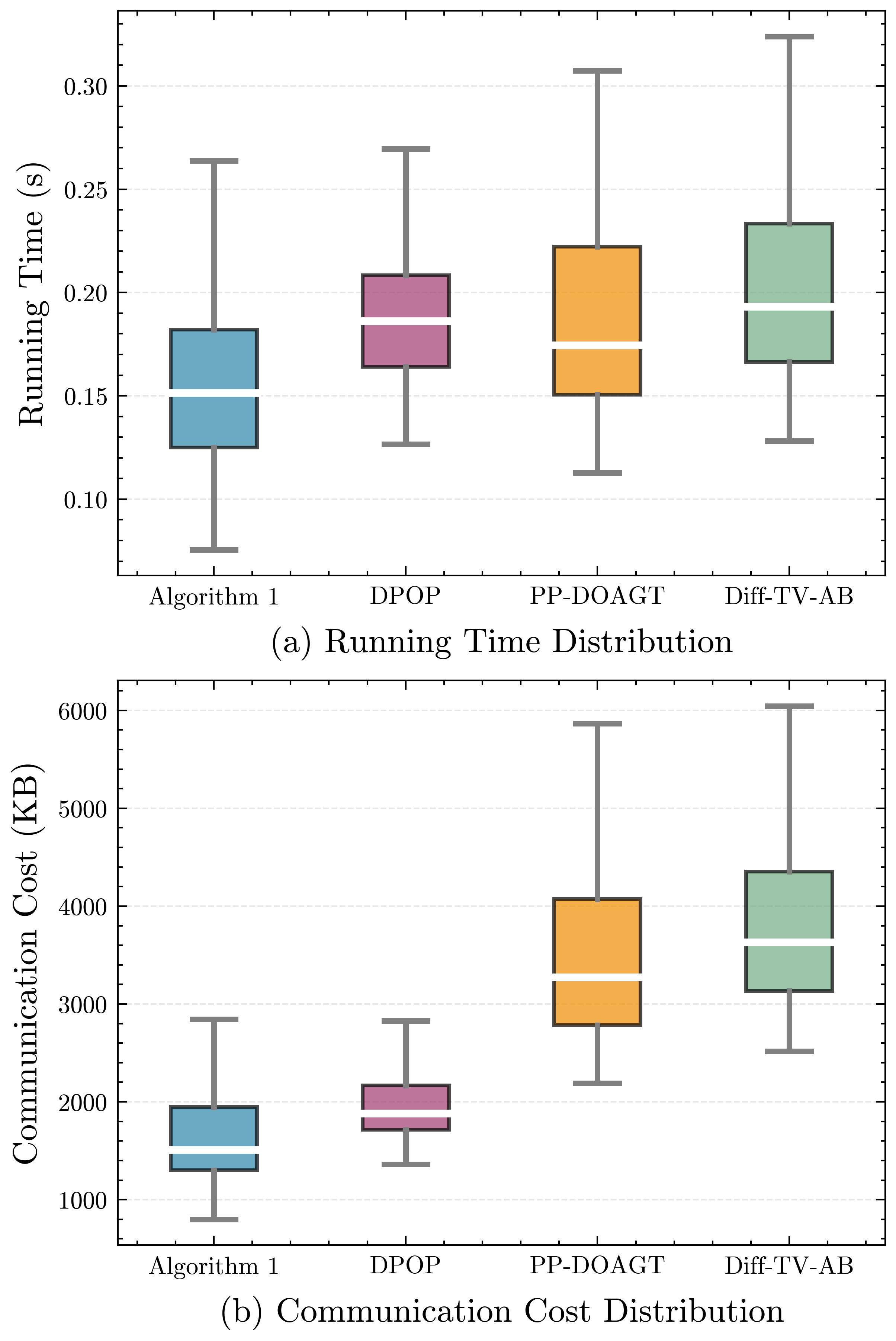}
  \caption{Running time and communication cost comparison of different algorithms.}
  \label{fig_algorithm_performance_comparison}
\end{figure}
To further validate the practical efficiency of Algorithm \ref{alg:alg_private}, we conduct a comparative experiment among four methods: Algorithm \ref{alg:alg_private}, DPOP, PP-DOAGT, and Diff-TV-AB under identical problem configurations. We adopt the same stopping criterion, that is, the iteration is terminated when the optimization residual falls below $1\times 10^{-2}$. To ensure statistical reliability, we perform 100 valid independent runs for each method and record the running time and total communication cost for every trial. The results are presented in Fig. \ref{fig_algorithm_performance_comparison}. Algorithm \ref{alg:alg_private} achieves the lowest median running time and the smallest median communication cost, demonstrating significantly better performance than PP-DOAGT and Diff-TV-AB. This advantage can be attributed to two main factors. On one hand, compared with DPOP, Algorithm \ref{alg:alg_private} exhibits better convergence behavior, reaching the prescribed accuracy in fewer iterations. On the other hand, relative to PP-DOAGT and Diff-TV-AB, Algorithm \ref{alg:alg_private}  significantly reduces the per-iteration computational and communication overhead (see Table \ref{table_algorithm_communication}), which yields higher efficiency in each iteration step.

%At higher privacy levels (the privacy budgets computed for PP-DOAGT via its own theoretical formula are 18446.44, 3652.67, and 365.27 with $\delta = 1$, respectively), our algorithm achieves comparable optimization accuracy to the PP-DOAGT method. This is due to the fact that our algorithm has lower sensitivity than the PP-DOAGT method (as shown in Table \ref{tab: sensitivity}).
%Due to the lower sensitivity of our algorithm (as shown in Table \ref{tab: sensitivity}), it demonstrates stronger privacy under the same noise.
%Note that, unlike our algorithm and DPOP, which provide an explicit closed-form expression from the privacy budget to the required noise parameters, PP-DOAGT provides a implicit, complex relation in which the privacy budget is expressed as a function of the noise and other internal parameters (see Eq. (17) or Eq. (22) in \cite{huang2024differential}). This relation cannot be directly used to obtain a deployable noise sequence for a given privacy budget. Thus, to allow a meaningful comparison, we inject the same sequence of noise variances (i.e., the same noise conditions) into both algorithms.} The results are shown in Fig. \ref{fig_residual_ppdoagt}. 

\subsection{Comparison of Differentially Private Methods Based on Distributed Learning on MNIST}
In this subsection, we further validate the effectiveness of Algorithm \ref{alg:alg_private} through distributed learning tasks on the MNIST dataset, while comprehensively evaluating its performance via parameter sensitivity analysis and privacy-accuracy trade-off studies. The MNIST dataset, a classic benchmark comprising handwritten digit images, is widely utilized for evaluating algorithm performance in image classification tasks. In this experiment, we consider a distributed learning scenario with 100 agents collaboratively training a convolutional neural network (CNN), whose communication topology is illustrated in Fig. \ref{fig_graph}. The CNN architecture adopts a two-layer convolutional structure: each layer consists of a $5\times 5$ convolution, BatchNorm normalization, and a $2\times 2$ max-pooling operation, followed by a fully-connected layer that outputs predictions over 10 classes. Each agent maintains a local copy of this CNN. We uniformly distribute the MNIST training images across all agents and train the model using a batch size of 16 with a cross-entropy loss function.

\begin{figure}[!t]
  \centering
  \includegraphics[width = 0.8\linewidth]{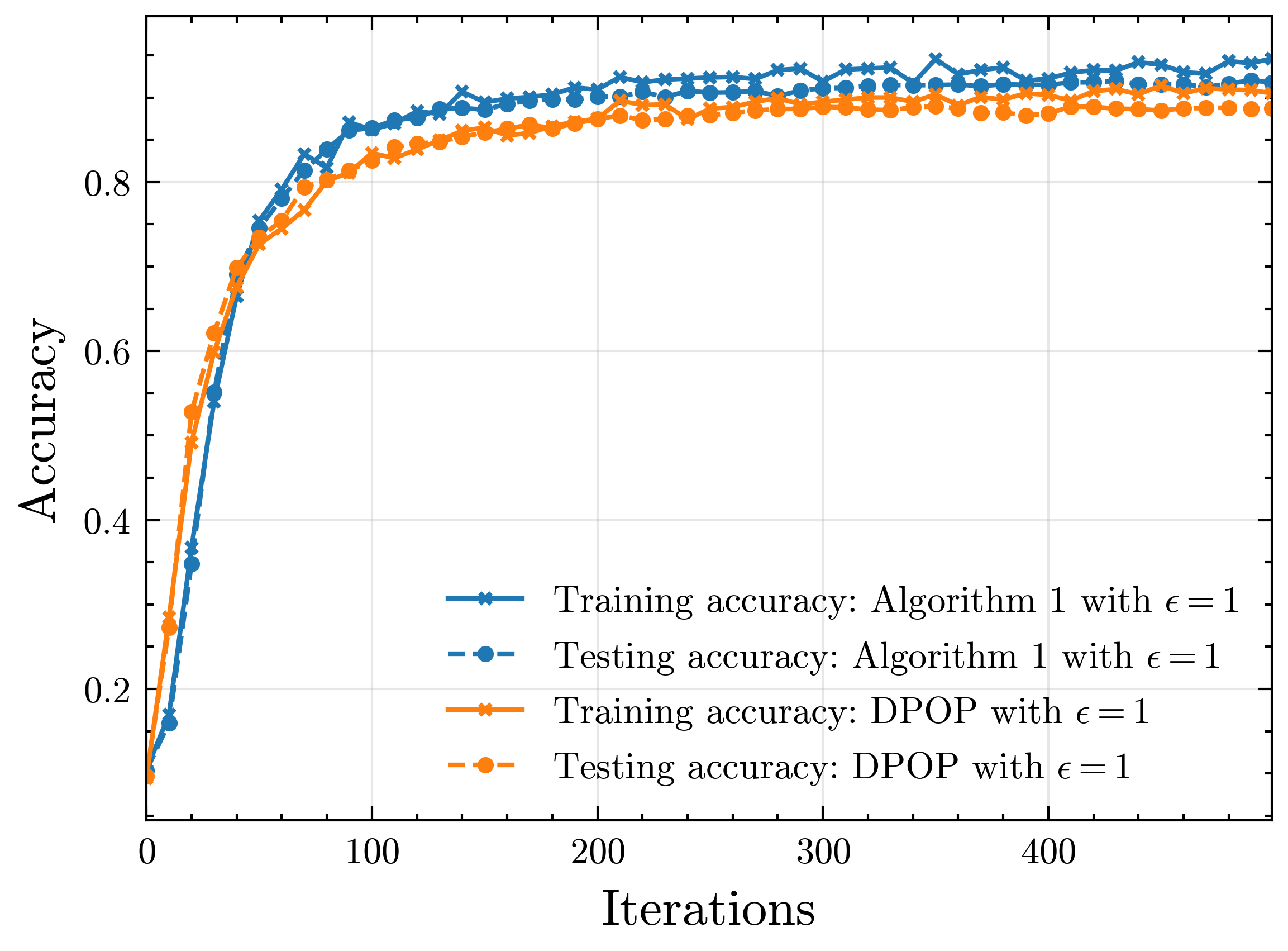}
  \caption{Accuracy comparison between Algorithm \ref{alg:alg_private} and DPOP \cite{huang2015differentially} under the same privacy budget on MNIST image classification.}
  \label{fig_accuracy_comparison_alg1_vs_dpop}
\end{figure}

\begin{figure}[!t]
  \centering
  \includegraphics[width = 0.8\linewidth]{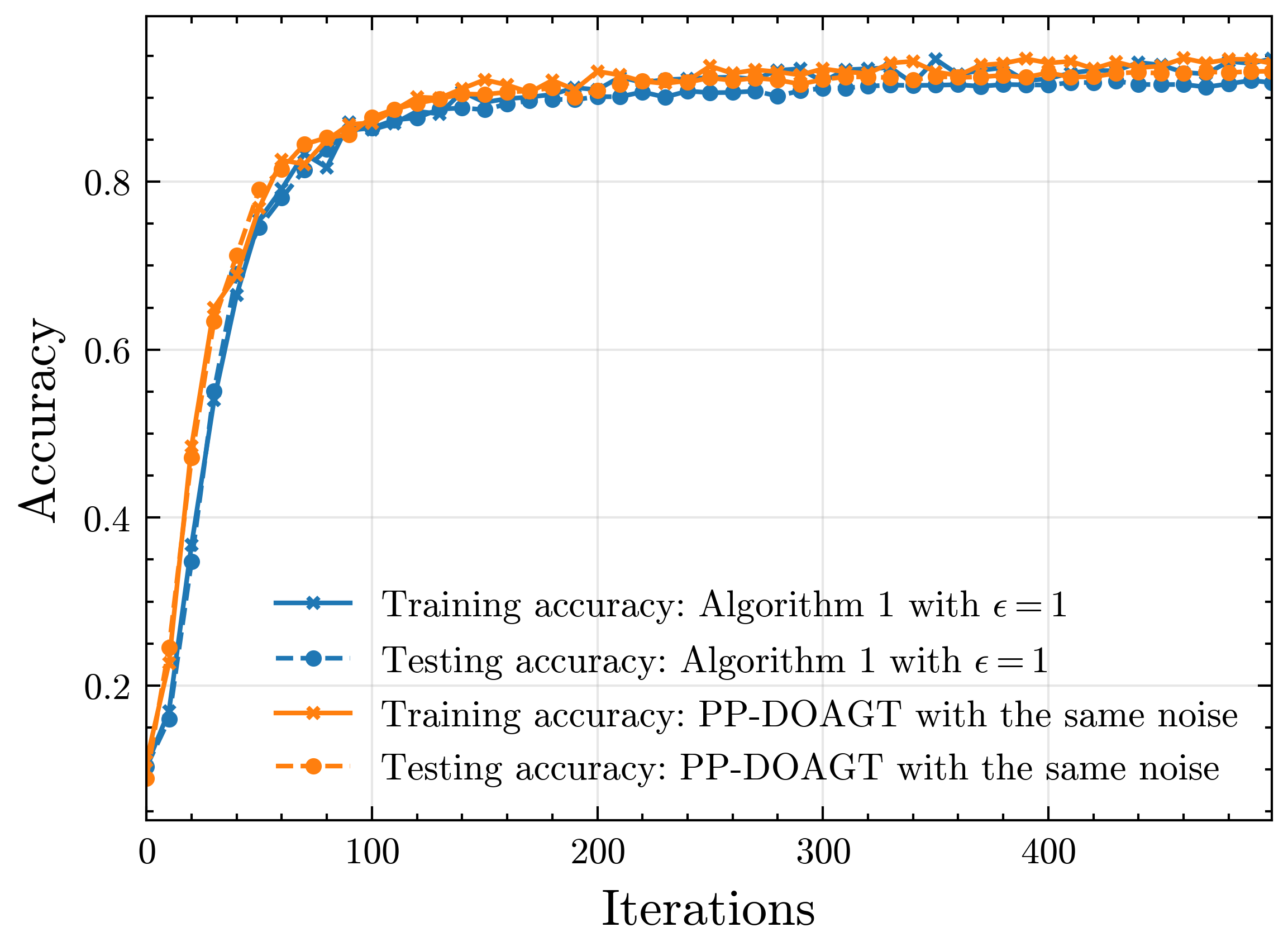}
  \caption{Accuracy comparison between Algorithm \ref{alg:alg_private} and PP-DOAGT \cite{huang2024differential} under the same noise on MNIST image classification.}
  \label{fig_accuracy_comparison_alg1_vs_ppdoagt}
\end{figure}

\begin{figure}[!t]
  \centering
  \includegraphics[width = 0.8\linewidth]{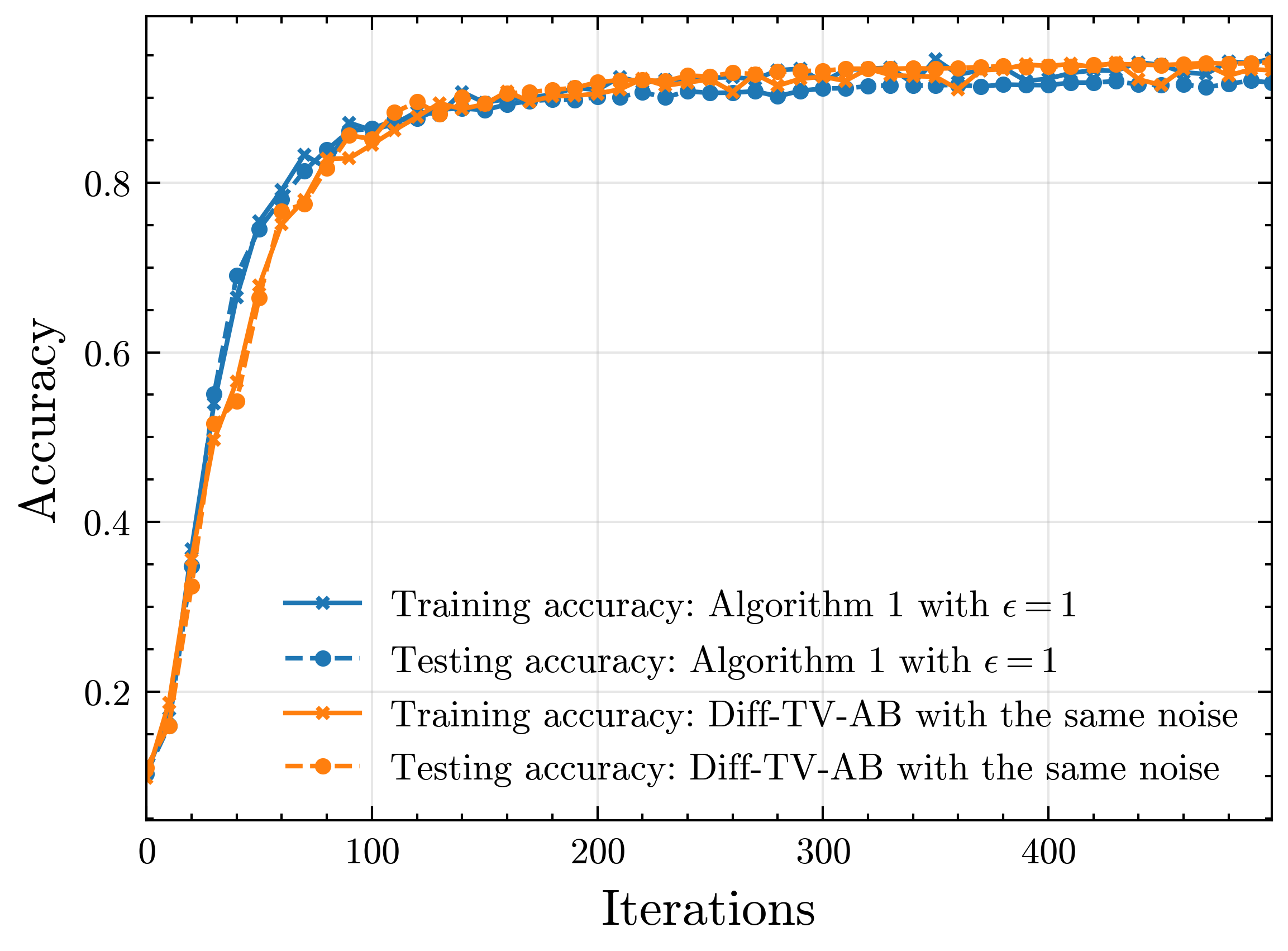}
  \caption{Accuracy comparison between Algorithm \ref{alg:alg_private} and Diff-TV-AB \cite{Yang2025Differentially} under the same noise on MNIST image classification.}
  \label{fig_accuracy_comparison_alg1_vs_difftvab}
\end{figure}

For Algorithm \ref{alg:alg_private}, the relevant parameters are set as $\gamma = 0.001$, $\beta = 100$, $q_2 = 0.999$, and $q_1 = 0.998$. We record the evolution of training and testing accuracy of Algorithm \ref{alg:alg_private} under a privacy budget of $\epsilon = 1$, as shown by the blue lines in Fig. \ref{fig_accuracy_comparison_alg1_vs_dpop}. The results demonstrate that the final training and testing accuracy both exceed 90\%, indicating that our algorithm can maintain high model performance while ensuring DP. For comparison, we implement the DPOP method from \cite{huang2015differentially}, with its performance depicted by the orange lines in Fig. \ref{fig_accuracy_comparison_alg1_vs_dpop}. Under the same privacy budget, our algorithm demonstrates superior performance in both training and testing accuracy compared to the DPOP method. Furthermore, we reproduce two recently proposed methods, PP-DOAGT \cite{huang2024differential} and Diff-TV-AB \cite{Yang2025Differentially}, with experimental results shown by the orange lines in Figs. \ref{fig_accuracy_comparison_alg1_vs_ppdoagt} and \ref{fig_accuracy_comparison_alg1_vs_difftvab}. The results indicate that under the same noise conditions, our algorithm achieves optimization accuracy comparable to both PP-DOAGT and Diff-TV-AB. Benefiting from its lower sensitivity design (as shown in Table \ref{tab: sensitivity}), our algorithm exhibits stronger privacy protection performance relative to PP-DOAGT and Diff-TV-AB.

\begin{figure}[!t]
  \centering
  \includegraphics[width = 0.8\linewidth]{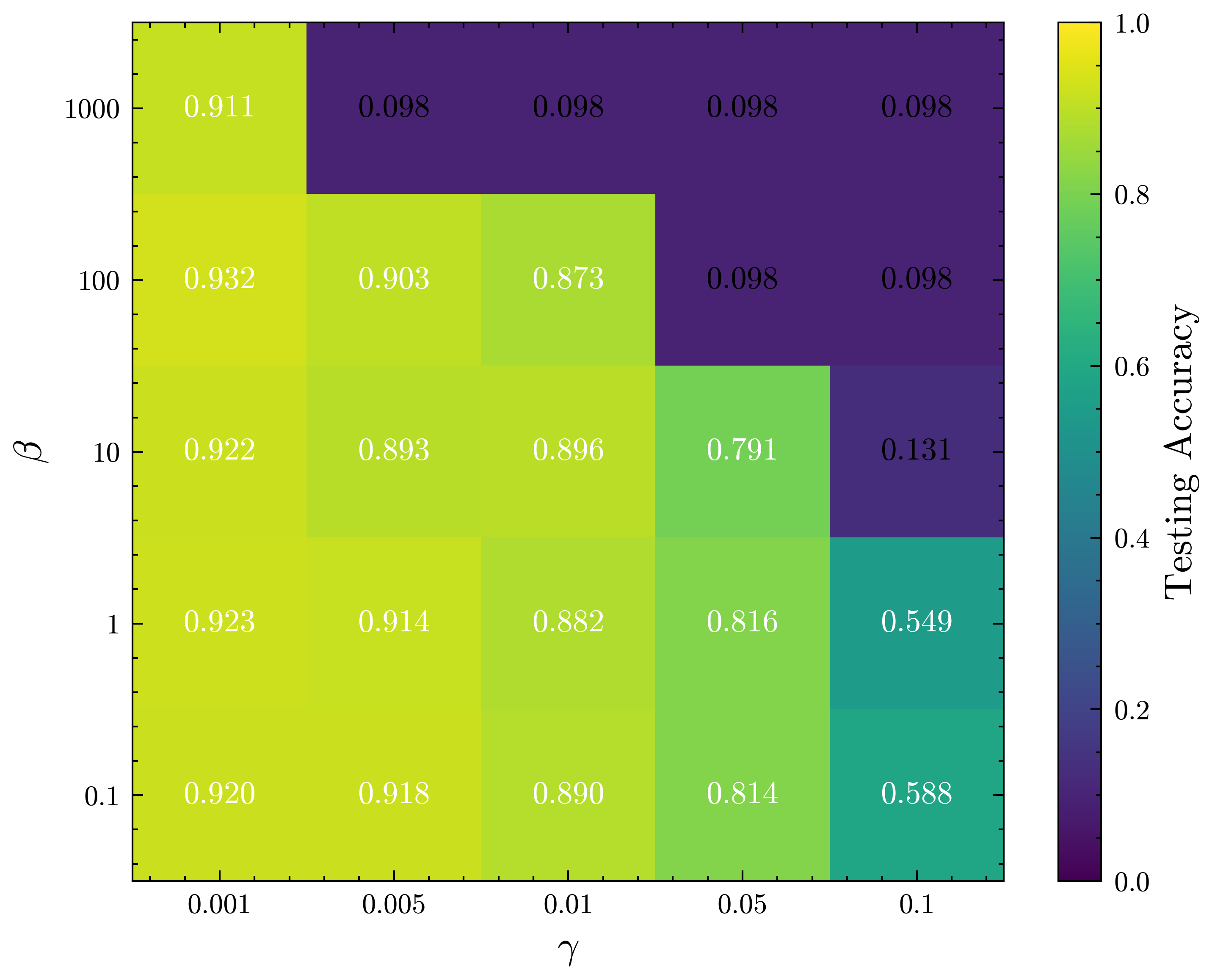}
  \caption{Testing accuracy of Algorithm \ref{alg:alg_private} under different $\gamma$ and $\beta$ parameter configurations.}
  \label{fig_sensitivity_analysis}
\end{figure}

To evaluate the robustness and stability of Algorithm \ref{alg:alg_private}, we conduct a parameter sensitivity analysis, focusing on the effects of initial stepsize parameter $\gamma$ and parameter $\beta$ on model performance. The parameter $\gamma$ is tested over the range $\{0.001, 0.005, 0.01, 0.05, 0.1\}$, while $\beta$ is evaluated across $\{0.1, 1, 10, 100, 1000\}$. Experimental results using testing accuracy as the evaluation metric are presented in Fig. \ref{fig_sensitivity_analysis}. Analysis reveals that when $\gamma$ takes relatively small values (specifically $\gamma \le 0.01$), and $\beta$ satisfies $\gamma\beta \le 1$, the algorithm achieves stable testing accuracy exceeding 85\%, demonstrating desirable performance levels. This finding experimentally validates the necessity of the condition $\gamma\beta\le 1$ for convergence. Furthermore, we observe that within the range of $\gamma \le 0.01$, even with substantial variations in $\beta$ values, testing accuracy remains consistently high as long as $\gamma\beta\le 1$ is maintained. This indicates Algorithm \ref{alg:alg_private}'s insensitivity to parameter variations and confirms its strong robustness.

\begin{figure}[!t]
  \centering
  \includegraphics[width = 0.8\linewidth]{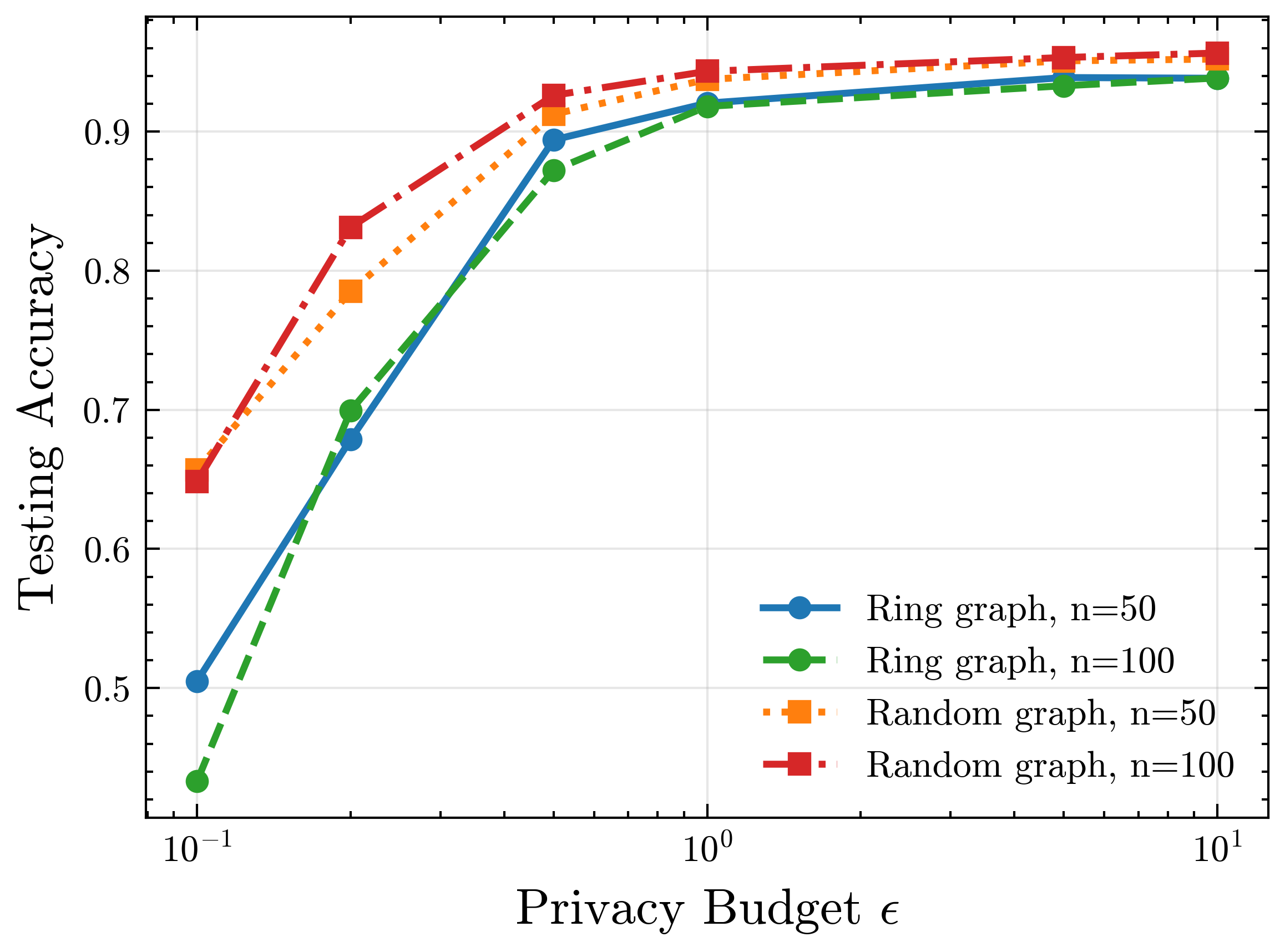}
  \caption{Privacy-accuracy trade-off of Algorithm \ref{alg:alg_private} under different network topologies and scales.}
  \label{fig_privacy_accuracy_tradeoff}
\end{figure}

To further evaluate the performance of Algorithm \ref{alg:alg_private}, we systematically investigate its privacy-accuracy trade-off characteristics under different network topologies (ring graph and ER random graph) and network scales ($n=50$ and $n=100$). The experimental results are presented in Fig. \ref{fig_privacy_accuracy_tradeoff}. Analysis of Fig. \ref{fig_privacy_accuracy_tradeoff} reveals that across all experimental configurations, testing accuracy consistently improves with increasing the privacy budget $\epsilon$, clearly demonstrating the inherent privacy-accuracy trade-off in differentially private algorithms, which aligns with the theoretical result in Theorem \ref{theorem:convergence}. Furthermore, under the same privacy budgets, the ER random graph topology achieves higher testing accuracy compared to the ring graph topology. This phenomenon can be attributed to the higher connectivity of random graphs, which facilitates information propagation among agents, thereby accelerating convergence and enhancing final performance. Notably, across all combinations of network topologies and scales, Algorithm \ref{alg:alg_private} maintains testing accuracy above 85\% when the privacy budget $\epsilon \ge 0.5$. These results demonstrate both the scalability of Algorithm \ref{alg:alg_private} and its capability to maintain high model accuracy while providing moderate privacy protection.

% The training images of the MNIST dataset are uniformly distributed across all agents, and the model is trained with a batch size of 16 and a cross-entropy loss function.

\begin{figure}[!t]
  \centering
  \includegraphics[width = 0.3\linewidth]{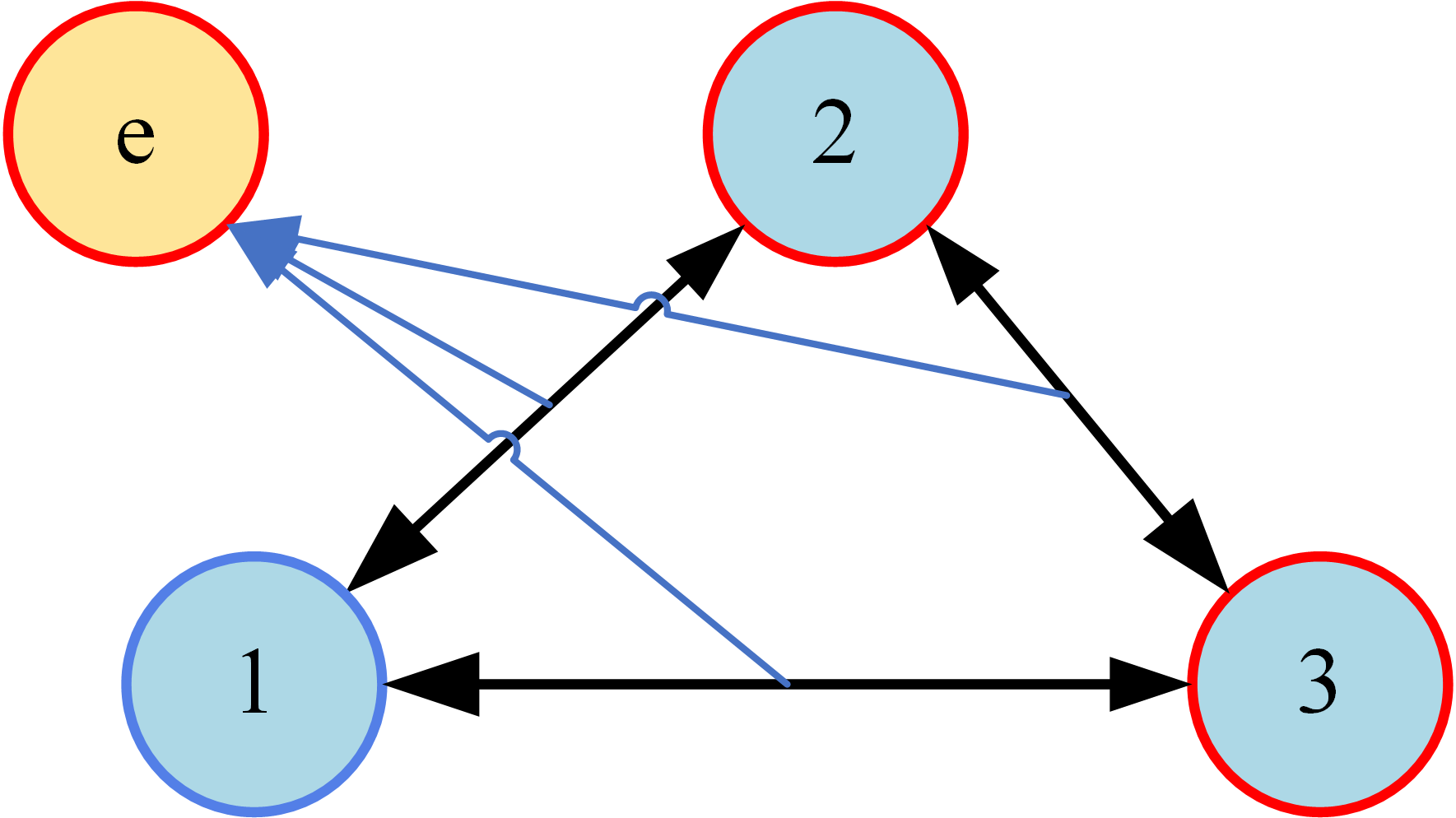}
  \caption{Schematic illustration of the honest-but-curious adversary and the external eavesdropper.}
  \label{fig_graph_attack}
\end{figure}

\begin{table*}[!t]
    \centering
    \caption{Comparison of the Privacy, Accuracy, Applicability, and Efficiency of the Three Different Types of Algorithms.}
    \begin{tabular}{cccccc}
        \toprule
        \multirow{2}{*}{Methods} & M-NMI under  & M-NMI under & \multirow{2}{*}{Accuracy} & Topological & \multirow{2}{*}{Efficiency} \\
        & external eavesdropper & honest-but-curious adversary & & assumption & \\
        \midrule
        Algorithm \ref{alg:alg_private} with $\epsilon = 10$ & 0.52  & 0.52 & $1.9\times10^{-4}$ & \ding{55} & High \\
        Algorithm \ref{alg:alg_private} with $\epsilon = 1$ & 0.24  & 0.24 & $2.0\times10^{-3}$ & \ding{55} & High \\
        Algorithm \ref{alg:alg_private} with $\epsilon = 0.1$ & 0.047  & 0.047 & $3.0\times10^{-2}$ & \ding{55} & High \\
        Correlated randomness-based method \cite{huan2023dynamics} & 1 & 0.47 & $6.6\times10^{-9}$ & \Checkmark & High \\
        Encryption-based method \cite{zhang2019enabling} & 0 & 0.30 & $3.2\times10^{-6}$ & \ding{55} & Low\\
%        Our algorithm  & \multirow{2}{*}{0.00}  & \multirow{2}{*}{0.1704}  \\
%        (Random weights) & & \\
%        Our algorithm  & \multirow{2}{*}{0.00} & \multirow{2}{*}{0.5449}  \\
%        (Fixed weights) & & \\
%        \cite{gao2023dynamics}  & 1.00   & 0.3315   \\
%        \cite{cheng2024privacy}  & 1.00  &  0.3853 \\
        \bottomrule
    \end{tabular}
    \label{table_NMI}  
\end{table*}

\subsection{Comprehensive Performance Comparison Across Different Types of Privacy-Preserving Methods}
To objectively evaluate the privacy levels of different types of privacy-preserving algorithms, we employ the Maximum Normalized Mutual Information (M-NMI), a modified mutual information-based metric from \cite{li2021privacypreserving}, as the privacy measure:
	\begin{equation}\label{eq:M-NMI}
	\begin{aligned}
	\text{M-NMI} \triangleq \max_{1\le k\le K}\left\{ \frac{I(\mathcal{V}_{i}(k), \mathcal{I}_{j}(k))}{I(\mathcal{V}_{i}(k), \mathcal{V}_{i}(k))} \right\},
	\end{aligned}
	\end{equation}
where $\mathcal{V}_{i}(k)$ is the private variable of agent $i$ at iteration $k$, $\mathcal{I}_{j}(k)$ represent the information attacker $j$ can obtain about $\mathcal{V}_{i}(k)$, and $K$ is the total number of iterations. $I(\mathcal{V}_{i}(k), \mathcal{I}_{j}(k))$ is the mutual information between $\mathcal{V}_{i}(k)$ and $\mathcal{I}_{j}(k)$, which quantifies the amount of agent $i$'s private information contained in the information acquired by attacker $j$. $I(\mathcal{V}_{i}(k), \mathcal{V}_{i}(k))$ represents the entropy of agent $i$'s private information, which is used for normalization. M-NMI quantifies the worst-case privacy leakage across all iterations. A smaller M-NMI value indicates stronger privacy preservation of the algorithm. For $\mathcal{V}_{i}(k)$, we consider the agent's gradients as its private information \cite{zhu2019deep}. For $\mathcal{I}_j(k)$, it depends on the attacker's model and the algorithm's update rule. We consider two common attacker models: \textit{the honest-but-curious adversary} and \textit{the external eavesdropper} \cite{goldreich2009foundations}. We adopt the system as shown in Fig. \ref{fig_graph_attack}, in which agent 1 is an honest agent, agents 2 and 3 are honest-but-curious adversaries, and agent $e$ is an external eavesdropper. Agents 1, 2, and 3 collaboratively solve Problem (\ref{problem_sensor}) with dimensions $m = 1$ and $p = 1$. Agents 2, 3, and $e$ all attempt to infer the private information of agent 1. For each method, we perform 5,000 Monte Carlo trials, each comprising 1,000 iterations. We then compute its M-NMI values under the honest-but-curious adversary and the external eavesdropper using the NPEET toolbox \cite{NPEET}.

Under Algorithm \ref{alg:alg_private}, the private information of agent 1 is $\nabla f_{1}\left(\bm{z}_{1}(k)\right), 1\le k\le K$, i.e., $\mathcal{V}_{1}(k) = \left\{ \nabla f_{1}\left(\bm{z}_{1}(k)\right) \right\}$. Honest-but-curious adversaries 2 and 3 can collude to infer the private information of agent 1. They receive $\bm{z}_{1}(k)$ from agent 1, infer $\bm{y}_{1}(k)$ using update (\ref{algorithm_2}) and public information $W$ and $\beta$, and approximate $\nabla f_{1}\left(\bm{z}_{1}(k)\right)$ using update (\ref{algorithm_3}) and public information $\alpha_{k}$, that is,
	\begin{equation}\label{eq:information_agent2}
	\begin{aligned}
	\mathcal{I}_{2}(k) = \left\{ \bm{z}_{1}(k), \bm{y}_{1}(k), \frac{1}{\alpha_{k}}\left( \bar{\bm{z}}_{1}(k) - \bm{z}_{1}(k) \right) - \bm{y}_{1}(k) \right\}.
	\end{aligned}
	\end{equation}
The external eavesdropper $e$ acquires the same information about agent 1 as the honest-but-curious adversaries, leading to $\mathcal{I}{e}(k) = \mathcal{I}_{2}(k)$. The corresponding simulation results under different privacy budgets are shown in the first three rows of Table \ref{table_NMI}. It can be seen that our algorithm can simultaneously defend against the honest-but-curious adversary and the external eavesdropper. Furthermore, as the privacy budget decreases, both the M-NMI value and optimization accuracy decline, which aligns with the intuition of a differentially private algorithm.
%\textcolor{red}{We use the privacy measure M-NMI, proposed in Section \ref{section_measure},} to compare Algorithm \ref{alg:alg_private} with two other types of privacy-preserving methods: the correlated randomness-based method from \cite{huan2023dynamics} and the encryption-based method from \cite{zhang2019enabling}. For clarity, we adopt the system illustrated in Fig. \ref{fig_graph_attack} and set the dimensions in Problem (\ref{problem_sensor}) to $m = 1$ and $p = 1$. For each method, we perform 5,000 Monte Carlo trials, each comprising 1,000 iterations. We then compute its M-NMI values under the honest-but-curious adversary and the external eavesdropper using the NPEET toolbox \cite{NPEET}. For Algorithm \ref{alg:alg_private}, the information accessible to attackers is given by (\ref{eq:information_agent2}). The corresponding simulation results under different privacy budgets are shown in the first three rows of Table \ref{table_NMI}. It can be seen that our algorithm can simultaneously defend against the honest-but-curious adversary and the external eavesdropper. Furthermore, as the privacy budget decreases, both the M-NMI value and optimization accuracy decline, which aligns with the intuition of a differentially private algorithm. 

For comparison, we implement the correlated randomness-based privacy-preserving method from \cite{huan2023dynamics}. The correlated randomness-based method relies on an additional topological assumption: each honest agent has at least one neighboring agent that is also honest. Thus, for this method, we assume that agent 3 in Fig. \ref{fig_graph_attack} is also an honest agent, and we compute its M-NMI under the honest-but-curious adversary using the following information:
	\begin{equation*}
	\begin{aligned}
	\mathcal{I}_{2}(k) = \left\{ \bm{x}_{1}^{k}, \Lambda_{1}^{k}\bm{y}_{1}^{k}, C_{21}^{k}\bm{y}_{1}^{k} + B_{21}^{k}(\nabla f_{1}(\bm{x}_{1}^{k+1}) - \nabla f_{1}(\bm{x}_{1}^{k})) \right\},
	\end{aligned}
	\end{equation*}
which consists of information directly received by agent 2 from agent 1. The relevant results are shown in the fourth row of Table \ref{table_NMI}. It can be observed that this method achieves high accuracy, as the correlated randomness-based method enables exact convergence. Under the external eavesdropper, the M-NMI value of this method is 1, indicating that it cannot defend against the external eavesdropper. This is because the privacy of the correlated randomness-based method relies on the additional topological assumption, which does not work against the external eavesdropper. We also implement the encryption-based privacy-preserving method from \cite{zhang2019enabling}. Its M-NMI under the honest-but-curious adversary is computed using the following information:
	\begin{equation*}
	\begin{aligned}
	\mathcal{I}_{2}(k) = \Big\{ \frac{1}{\alpha_{k}}\Big(\sum_{j\in\{2,3\}} b_{1\to j}b_{j\to1}(\bm{x}_{j}(k) - \bm{x}_{1}(k)\Big) \Big\},
	\end{aligned}
	\end{equation*}
which is an estimate of agent 1's gradient according to its update rule. The corresponding results are shown in the last row of Table \ref{table_NMI}. It can be seen that this method achieves an M-NMI value of 0 under the external eavesdropper, meaning the encryption-based method is fully resistant to the external eavesdropper. This is because the external eavesdropper can only access encrypted data. Moreover, this method exhibits high accuracy as the encryption-based method can also achieve exact convergence. However, the encryption-based method requires numerous encryption and decryption operations, resulting in low computational and communication efficiency. In addition, although all three types of algorithms can defend against the honest-but-curious adversary, Table \ref{table_NMI} demonstrates that our algorithm allows for enhanced privacy performance through adjustable privacy budgets, whereas the privacy performance of the other two methods remains fixed.

Based on the comparative results presented in Table \ref{table_NMI}, the proposed algorithm demonstrates distinct advantages in the following practical deployment scenarios: 
\begin{enumerate}
\item Scenarios Requiring Defense Against Both Internal and External Threats: Correlated randomness-based methods fail to protect against external eavesdroppers (M‑NMI = 1). The proposed algorithm achieves low M‑NMI values under the external eavesdropper and honest-but-curious adversary, confirming its capability to resist both types of adversaries simultaneously.
\item Scenarios Demanding Tunable and Quantifiable Privacy Guarantees: Both correlated randomness-based and encryption-based methods offer a fixed level of privacy protection. Our algorithm allows the privacy budget $\epsilon$ to be adjusted according to specific application requirements, providing configurable and provable privacy.
\item Resource-Constrained Large-Scale Distributed Systems: Encryption-based methods incur high computational and communication overhead due to frequent encryption and decryption operations. The proposed algorithm only requires lightweight noise injection and, compared to existing DP gradient-tracking methods, reduces the number of communication variables and noise injections by half, significantly lowering the resource burden.
\item Networks with Dynamic or Unknown Topology: Correlated randomness-based methods rely on the assumption that each honest agent has at least one honest neighbor, which is often impractical in self-organizing or time-varying networks. Our algorithm only requires network connectivity, making it directly applicable to mobile or peer-to-peer networks without imposing additional topological constraints.
\end{enumerate}
In summary, the proposed differentially private algorithm provides a versatile and practical solution for privacy-preserving distributed optimization. It effectively balances provable privacy, a configurable privacy-accuracy trade-off, lower resource overhead, and broader topological applicability under various real-world constraints and threat models.

\section{Conclusion}\label{section_conclusion}
In this paper, we propose a differentially private distributed optimization algorithm for undirected graphs, and derive an explicit closed-form expression for the noise parameter as a function of the privacy budget. Compared to existing state-of-the-art algorithms, the proposed algorithm exhibits lower sensitivity, thereby achieving higher optimization accuracy under the same privacy budget. Furthermore, the proposed algorithm reduces both the number of communication variables and the amount of injected noise by half, significantly enhancing communication and computational efficiency. Moreover, we rigorously analyze the algorithm's differential privacy, establish its convergence, and derive an upper bound on its optimization accuracy. The theoretical results reveal the inherent trade-off between privacy and accuracy and provide guidance for parameter tuning. Future work will proceed along three promising directions:
(i) Extending the sensitivity‑reduction framework to directed graph topologies.
(ii) Enhancing the algorithm's convergence rate while preserving its privacy and accuracy guarantees, where potential pathways include designing adaptive stepsize schedules (e.g., two-phase or noise-aware decay) and incorporating variance‑reduction techniques.
(iii) Generalizing the algorithm to broader problem classes, including convex (but not strongly convex) objectives and structured nonconvex problems frequently encountered in modern machine learning.

%Future work will focus on developing differentially private distributed algorithms with lower sensitivity for directed graphs.
%In addition, we introduce a general privacy measure to quantify the privacy levels of various types of privacy-preserving algorithms.

\appendices
\section{Proof of Proposition \ref{prop:sensitivity}}\label{proof_proposition}
%\subsection{Proof of Proposition \ref{prop:sensitivity}}
\begin{proof}
We prove the proposition by directly comparing the algorithm's sensitivity under different designs. We consider the classical differentially private distributed gradient descent (DP-DGD) as the baseline algorithm:
	\begin{equation}
	\mathbf{x}(k) = W\mathbf{z}(k) - \alpha_{k} \nabla F(\mathbf{z}(k)),
	\end{equation}
where $\mathbf{z}(k) \triangleq \mathbf{x}(k - 1) + \bm{\xi}(k)$ is the noisy state and $\alpha_{k}$ is the stepsize. Note that in the DP-DGD algorithm, agents use their own noisy states for consensus and compute gradients based on noisy states. According to Definition \ref{def:sensitivity}, and since
	\begin{equation}\label{eq:sensitivity_dpdgd}
	\begin{aligned}
\| \mathbf{x}(k) - \mathbf{x}'(k) \|_{1} &= \| \bm{x}_{i_{0}}(k) - \bm{x}_{i_{0}}'(k) \|_{1} \\
&= \alpha_{k} \| \nabla f_{i_{0}}(\bm{z}_{i_{0}}(k)) - \nabla f_{i_{0}}'(\bm{z}_{i_{0}}(k)) \|_{1} \\
&\le \delta \alpha_{k},
	\end{aligned}
	\end{equation}	
the sensitivity of the DP-DGD algorithm, denoted as $\Delta_{1}(k)$, is bounded by $\delta \alpha_{k}$, i.e., $\Delta_{1}(k) \le \delta \alpha_{k}$. The proof details for (\ref{eq:sensitivity_dpdgd}) can be found in Theorem \ref{lemma_sensitivity} or in Lemma 4 of \cite{huang2015differentially}.

If agents use their true states for consensus, the DP-DGD algorithm can be reformulated as
	\begin{equation}
	\mathbf{x}(k) = W_{0}\mathbf{x}(k-1) + (W-W_{0})\mathbf{z}(k) - \alpha_{k} \nabla F(\mathbf{z}(k)),
	\end{equation}
where $W_{0} \triangleq \text{diag}\{W\}$. Let $\Delta_{2}(k)$ denote the sensitivity of the above algorithm at iteration $k$. Since $\| \bm{x}_{i_{0}}(k) - \bm{x}_{i_{0}}'(k) \|_{1} = \|W_{i_{0}i_{0}}(\bm{x}_{i_{0}}(k-1) - \bm{x}_{i_{0}}'(k-1)) - \alpha_{k}( \nabla f_{i_{0}}(\bm{z}_{i_{0}}(k)) - \nabla f_{i_{0}}'(\bm{z}_{i_{0}}(k))) \|_{1} \le W_{i_{0}i_{0}}\Delta_{2}(k-1) + \delta \alpha_{k}$, we have $\Delta_{2}(k) \le W_{i_{0}i_{0}}\Delta_{2}(k-1) + \delta \alpha_{k}$. Due to $\Delta_{2}(k-1) \ge 0$ for all $k\ge1$, we obtain the first argument.

If agents compute gradients based on their true states, the DP-DGD algorithm is reformulated as
	\begin{equation}
	\mathbf{x}(k) = W\mathbf{z}(k) - \alpha_{k} \nabla F(\mathbf{x}(k-1)).
	\end{equation}
Let $\Delta_{3}(k)$ denote the sensitivity of the above algorithm at iteration $k$. Since $\| \bm{x}_{i_{0}}(k) - \bm{x}_{i_{0}}'(k) \|_{1} = \alpha_{k} \| \nabla f_{i_{0}}(\bm{x}_{i_{0}}(k-1)) - \nabla f_{i_{0}}'(\bm{x}_{i_{0}}'(k-1)) \|_{1} = \alpha_{k} \| \nabla f_{i_{0}}(\bm{x}_{i_{0}}(k-1)) - \nabla f_{i_{0}}(\bm{x}_{i_{0}}'(k-1)) + \nabla f_{i_{0}}(\bm{x}_{i_{0}}'(k-1)) - \nabla f_{i_{0}}'(\bm{x}_{i_{0}}'(k-1)) \|_{1} \le \sqrt{n}L\alpha_{k}\Delta_{3}(k-1) + \delta \alpha_{k}$, we obtain $\Delta_{3}(k) \le \sqrt{n}L\alpha_{k}\Delta_{3}(k-1) + \delta \alpha_{k}$. Because of $\Delta_{3}(k-1) \ge 0$ for all $k\ge1$, we have the second argument.
%	\begin{equation}
%	\begin{aligned}
%\Delta_{3}(k) &= \| \bm{x}_{i_{0}}(k) - \bm{x}_{i_{0}}'(k) \|_{1} \\
%&= \alpha_{k} \| \nabla f_{i_{0}}(\bm{x}_{i_{0}}(k-1)) - \nabla f_{i_{0}}'(\bm{x}_{i_{0}}'(k-1)) \|_{1} \\
%&= \alpha_{k} \| \nabla f_{i_{0}}(\bm{x}_{i_{0}}(k-1)) - \nabla f_{i_{0}}(\bm{x}_{i_{0}}'(k-1)) + \nabla f_{i_{0}}(\bm{x}_{i_{0}}'(k-1)) - \nabla f_{i_{0}}'(\bm{x}_{i_{0}}'(k-1)) \|_{1} \\
%&\le \sqrt{n}L\alpha_{k}\Delta_{3}(k-1) + \delta \alpha_{k}.
%	\end{aligned}
%	\end{equation}

For the third argument, when neighboring agents share only variable $\bm{x}$, the algorithm's sensitivity depends on the bound of $\| \mathbf{x}(k) - \mathbf{x}'(k) \|_{1}$. If agents are also required to exchange additional variables (e.g., an auxiliary variable $\bm{y}$), the algorithm's sensitivity increases to $\| \mathbf{x}(k) - \mathbf{x}'(k) \|_{1} + \| \mathbf{y}(k) - \mathbf{y}'(k) \|_{1}$. Hence, the third argument holds.
\end{proof}

\section{Proof of Basic Lemmas}\label{proof_lemmas}
\subsection{Proof of Lemma \ref{lemma:xi}}\label{appe:proof_lemma_xi}
\begin{proof}
It follows from Assumption \ref{assumption_noise} that $\mathbb{E}\left[\bm{\xi}(k)|\mathcal{M}_{k}\right] = 0$, $\mathbb{E}\left[\|\bm{\xi}(k)\||\mathcal{M}_{k}\right] \le np\nu_{k}$, and $\mathbb{E}\left[\|\bm{\xi}(k)\|^{2}|\mathcal{M}_{k}\right] = 2np\nu_{k}^{2}$. For part \textit{(a)}, it follows from the definition of $\mathbf{z}(k)$ that
\begin{equation}
\begin{small}
\begin{aligned}
\mathbb{E}\left[\langle\mathbf{z}(k),\bm{\xi}(k)\rangle|\mathcal{M}_{k}\right] &= \mathbb{E}\left[\langle \bm{\xi}(k),\bm{\xi}(k)\rangle|\mathcal{M}_{k}\right] \\
&\le \mathbb{E}\left[\|\bm{\xi}(k)\|^{2}|\mathcal{M}_{k}\right] \\
&= 2np\nu_{k}^{2}.
\end{aligned}
\end{small}
\end{equation}

For part \textit{(b)}, in light of Assumption \ref{assump:smooth_convex}, we have
\begin{equation}
\begin{small}
\begin{aligned}
&\quad \mathbb{E}\left[\left\langle\nabla F(\mathbf{z}(k)),\bm{\xi}(k)\right\rangle|\mathcal{M}_{k}\right] \\
&= \mathbb{E}\left[\left\langle\nabla F(\mathbf{z}(k))-\nabla F(\mathbf{x}(k-1)),\bm{\xi}(k)\right\rangle|\mathcal{M}_{k}\right] \\
&\le L\mathbb{E}\left[\|\bm{\xi}(k)\|^{2}|\mathcal{M}_{k}\right] \\
&= 2npL\nu_{k}^{2}.
\end{aligned}
\end{small}
\end{equation}

For part \textit{(c)}, following a similar derivation as in Lemma 8 of \cite{pu2021distributed}, we obtain
\begin{equation}
\begin{small}
\begin{aligned}
&\quad \mathbb{E}\left[\left\langle\nabla F(\mathbf{x}(k)),\bm{\xi}(k)\right\rangle|\mathcal{M}_{k}\right] \\
&\le \sum_{i=1}^{n}\left( L|W_{ii}-\alpha_{k}\beta+\alpha_{k}\beta W_{ii}|+ L^{2}\alpha_{k} \right) \mathbb{E}\left[\|\bm{\xi}_{i}(k)\|^{2}|\mathcal{M}_{k}\right] \\
&\le 2np(L+L^{2}\alpha_{k})\nu_{k}^{2},
\end{aligned}
\end{small}
\end{equation}
where the second inequality follows from the fact that $W_{ii}-\alpha_{k}\beta+\alpha_{k}\beta W_{ii}\in(-1,1)$. Then, we have
\begin{equation}
\begin{small}
\begin{aligned}
&\quad \mathbb{E}\left[\left\langle\nabla F(\mathbf{z}(k+1)),\bm{\xi}(k)\right\rangle|\mathcal{M}_{k}\right] \\
&= \mathbb{E}\left[\left\langle\nabla F(\mathbf{z}(k+1))-\nabla F(\mathbf{x}(k)),\bm{\xi}(k)\right\rangle|\mathcal{M}_{k}\right] \\
&\quad + \mathbb{E}\left[\left\langle\nabla F(\mathbf{x}(k)),\bm{\xi}(k)\right\rangle|\mathcal{M}_{k}\right] \\
&\le L\mathbb{E}\left[\|\bm{\xi}(k+1)\|\|\bm{\xi}(k)\||\mathcal{M}_{k}\right] + 2np(L+L^{2}\alpha_{k})\nu_{k}^{2} \\
&\le n^{2}p^{2}L\nu_{k+1}\nu_{k} + 2np(L+L^{2}\alpha_{k})\nu_{k}^{2} \\
&\le npL(2+np+2L\alpha_{k})\nu_{k}^{2}.
\end{aligned}
\end{small}
\end{equation}
%where the last inequality is due to the fact that $\nu_{k+1}\le \nu_{k}$.

For part \textit{(d)}, by substituting (\ref{algorithm_compact_y}) into (\ref{algorithm_compact_x}), we get
	\begin{equation}
	\begin{small}
	\begin{aligned}
	\mathbf{x}(k) &= ((1+\alpha_{k}\beta)W-\alpha_{k}\beta I)\mathbf{z}(k) - \alpha_{k}\mathbf{y}(k-1) \\
	&\quad - \alpha_{k}\nabla F(\mathbf{z}(k)).
	\end{aligned}
	\end{small}
	\end{equation}
%因为a是由b唯一确定的，并基于a和b部分，我们有
Since $\mathbf{y}(k-1)$ is uniquely determined by $\mathcal{M}_{k}$, and based on parts \textit{(a)} and \textit{(b)}, we have
\begin{equation}
\begin{small}
\begin{aligned}
&\quad \mathbb{E}\left[\langle\mathbf{x}(k),\bm{\xi}(k)\rangle|\mathcal{M}_{k}\right] \\
&= \mathbb{E}\left[\left\langle ((1+\alpha_{k}\beta)W-\alpha_{k}\beta I)\mathbf{z}(k) - \alpha_{k}\nabla F(\mathbf{z}(k)),\bm{\xi}(k)\right\rangle|\mathcal{M}_{k}\right]\\
&\le 2np\left(\|(1+\alpha_{k}\beta)W-\alpha_{k}\beta I\|+L\alpha_{k}\right)\nu_{k}^{2} \\
&\le 2np\left(3+L\alpha_{k}\right)\nu_{k}^{2}.
\end{aligned}
\end{small}
\end{equation}

%对于e部分，由1和部分a和b，我们有
For part \textit{(e)}, it follows from (\ref{equation_bar_x}) and parts \textit{(a)} and \textit{(b)} that
\begin{equation}
\begin{small}
\begin{aligned}
&\quad \mathbb{E}\left[\langle\mathbf{1}\bar{\bm{x}}(k),\bm{\xi}(k)\rangle|\mathcal{M}_{k}\right] \\
&= \mathbb{E}\left[\left\langle\frac{1}{n}\mathbf{1}\mathbf{1}^{T}\mathbf{z}(k)-\frac{\alpha_{k}}{n}\mathbf{1}\mathbf{1}^{T}\nabla F(\mathbf{z}(k)),\bm{\xi}(k)\right\rangle|\mathcal{M}_{k}\right] \\
&\le 2np(1+L\alpha_{k})\nu_{k}^{2}.
\end{aligned}
\end{small}
\end{equation}

%我们可以由a的定义得到f部分：
By the definition of $\mathbf{z}(k)$, part \textit{(f)} can be obtained as follows:
%We can obtain part (f) from the definition of $\mathbf{z}(k)$:
\begin{equation}\label{equation_z_barx}
\begin{small}
\begin{aligned}
&\quad \mathbb{E}\left[\|\mathbf{z}(k+1) - \mathbf{1}\bar{\bm{x}}(k)\|^{2}|\mathcal{M}_{k}\right] \\
&= \mathbb{E}\left[\|\mathbf{x}(k) - \mathbf{1}\bar{\bm{x}}(k) + \bm{\xi}(k+1)\|^{2}|\mathcal{M}_{k}\right] \\
&= \mathbb{E}\left[\|\mathbf{x}(k)-\mathbf{1}\bar{\bm{x}}(k)\|^{2}|\mathcal{M}_{k}\right] + \mathbb{E}\left[\|\bm{\xi}(k+1)\|^{2}|\mathcal{M}_{k}\right] \\
&= \mathbb{E}\left[ \|\mathbf{x}(k) - \mathbf{1}\bar{\bm{x}}(k) \|^{2} | \mathcal{M}_{k} \right] + 2np\nu_{k+1}^{2}.
\end{aligned}
\end{small}
\end{equation}

Similarly, we have
\begin{equation}\label{equation_z_k}
\begin{small}
\begin{aligned}
&\quad \mathbb{E}\left[\|\mathbf{z}(k+1)-\mathbf{z}(k)\|^{2}|\mathcal{M}_{k}\right] \\
&= \mathbb{E}\left[ \|\mathbf{x}(k)-\mathbf{x}(k-1) + \bm{\xi}(k+1)-\bm{\xi}(k)\|^{2}|\mathcal{M}_{k}\right] \\
&= \mathbb{E}\left[\|\mathbf{x}(k)-\mathbf{x}(k-1)\|^{2}|\mathcal{M}_{k}\right] + \mathbb{E}\left[ \| \bm{\xi}(k+1) - \bm{\xi}(k) \|^{2}|\mathcal{M}_{k}\right] \\
&\quad + 2\mathbb{E}\left[\langle\mathbf{x}(k),-\bm{\xi}(k)\rangle|\mathcal{M}_{k}\right] \\
&\le \mathbb{E}\left[\|\mathbf{x}(k)-\mathbf{x}(k-1)\|^{2}|\mathcal{M}_{k}\right] + 2np\nu_{k+1}^{2} \\
&\quad + 2np(7+2L\alpha_{k})\nu_{k}^{2},
\end{aligned}
\end{small}
\end{equation}
where in the first inequality we use part \textit{(d)}.
\end{proof}

\subsection{Proof of Lemma \ref{lemma:first_inequality}}
\begin{proof}
Let $\bm{a}_{1} \triangleq \bar{\bm{x}}(k) - \frac{\alpha_{k+1}}{n}\mathbf{1}^{T}\nabla F(\mathbf{1}\bar{\bm{x}}(k)) - \bm{x}^{*}$, $\bm{b}_{1} \triangleq \frac{\alpha_{k+1}}{n}\mathbf{1}^{T}\left( \nabla F(\mathbf{z}(k+1)) - \nabla F(\mathbf{1}\bar{\bm{x}}(k)) \right)$, and $\bm{c}_{1} \triangleq \frac{1}{n}\mathbf{1}^{T}\bm{\xi}(k+1)$. In light of (\ref{eq:x_x_star}) and Lemma \ref{lemma_smooth_convex}, if $\alpha_{k+1}\le\frac{2}{\mu+L}$ for any $k\ge0$, we have
\begin{equation}\label{eq:first_1}
\begin{small}
\begin{aligned}
&\quad \mathbb{E}\left[\| \bar{\bm{x}}(k+1) - \bm{x}^{*} \|^{2}|\mathcal{M}_{k}\right] \\
&= \mathbb{E}\left[\|\bm{a}_{1}\|^{2}|\mathcal{M}_{k}\right] + \mathbb{E}\left[\|\bm{b}_{1}\|^{2}|\mathcal{M}_{k}\right] + \mathbb{E}\left[\|\bm{c}_{1}\|^{2}|\mathcal{M}_{k}\right] \\
&\quad + \mathbb{E}\left[2\langle \bm{a}_{1},-\bm{b}_{1}\rangle|\mathcal{M}_{k}\right] + \mathbb{E}\left[2\langle -\bm{b}_{1},\bm{c}_{1}\rangle|\mathcal{M}_{k}\right] \\
&\le (1-\alpha_{k+1}\mu)^{2} \mathbb{E}\left[ \|\bar{\bm{x}}(k) - \bm{x}^{*}\|^{2}|\mathcal{M}_{k}\right] \\
&\quad + \frac{L^{2}\alpha_{k+1}^{2}}{n} \mathbb{E}\left[ \|\mathbf{z}(k+1) - \mathbf{1}\bar{\bm{x}}(k)\|^{2} | \mathcal{M}_{k} \right] \\
&\quad + 2p\nu_{k+1}^{2} + \mathbb{E}\left[2\langle \bm{a}_{1},-\bm{b}_{1}\rangle|\mathcal{M}_{k}\right] + \mathbb{E}\left[2\langle -\bm{b}_{1},\bm{c}_{1}\rangle|\mathcal{M}_{k}\right].
\end{aligned}
\end{small}
\end{equation}

We now need to bound the last two terms in the above inequality. By Lemma \ref{lemma_smooth_convex}, we obtain
\begin{equation}
\begin{small}
\begin{aligned}
&\quad 2 \langle \bm{a}_{1}, -\bm{b}_{1} \rangle \\
&\le 2 \Big\| \bar{\bm{x}}(k) - \frac{\alpha_{k+1}}{n}\mathbf{1}^{T}\nabla F(\mathbf{1}\bar{\bm{x}}(k)) - \bm{x}^{*} \Big\| \| \bm{b}_{1} \| \\
&\le 2(1 - \alpha_{k+1}\mu)\| \bar{\bm{x}}(k) - \bm{x}^{*} \|\| \bm{b}_{1} \| \\
&\le (1 - \alpha_{k+1}\mu) \Big( \alpha_{k+1}\mu \| \bar{\bm{x}}(k) - \bm{x}^{*} \|^{2} + \frac{1}{\alpha_{k+1}\mu}\| \bm{b}_{1} \|^{2} \Big) \\
&\le \alpha_{k+1}\mu(1 - \alpha_{k+1}\mu) \| \bar{\bm{x}}(k) - \bm{x}^{*} \|^{2} \\
&\quad +  \frac{L^{2}\alpha_{k+1}(1- \alpha_{k+1}\mu)}{n\mu}\| \mathbf{z}(k+1) - \mathbf{1}\bar{\bm{x}}(k) \|^{2},
\end{aligned}
\end{small}
\end{equation}
where in the third inequality we use Young's inequality. Then, taking the conditional expectation, we get
\begin{equation}\label{eq:first_2}
\begin{small}
\begin{aligned}
&\quad \mathbb{E}\left[ 2 \langle \bm{a}_{1}, -\bm{b}_{1} \rangle | \mathcal{M}_{k} \right] \\
&\le \alpha_{k+1}\mu(1 - \alpha_{k+1}\mu) \mathbb{E}\left[ \| \bar{\bm{x}}(k) - \bm{x}^{*} \|^{2} | \mathcal{M}_{k} \right] \\
&\quad + \frac{L^{2}\alpha_{k+1}(1 - \alpha_{k+1}\mu)}{n\mu} \mathbb{E}\left[ \| \mathbf{z}(k+1) - \mathbf{1}\bar{\bm{x}}(k) \|^{2} | \mathcal{M}_{k} \right].
\end{aligned}
\end{small}
\end{equation}

For the last term in (\ref{eq:first_1}), since $\mathbf{x}(k)$ is independent of $\bm{\xi}(k+1)$, we have
\begin{equation}\label{eq:first_3}
\begin{small}
\begin{aligned}
&\quad \mathbb{E}\left[ 2 \langle -\bm{b}_{1}, \bm{c}_{1} \rangle | \mathcal{M}_{k} \right] \\
&= \frac{2\alpha_{k+1}}{n^{2}} \mathbb{E}\left[ \left\langle -\mathbf{1}^{T}\nabla F(\mathbf{z}(k+1)), \mathbf{1}^{T}\bm{\xi}(k+1) \right\rangle | \mathcal{M}_{k} \right] \\
&\le 4pL\alpha_{k+1}\nu_{k+1}^{2},
\end{aligned}
\end{small}
\end{equation}
where the first inequality follows from the part \textit{(b)} of Lemma \ref{lemma:xi}. Finally, combining (\ref{eq:first_1}), (\ref{eq:first_2}), and (\ref{eq:first_3}) gives
\begin{equation}
\begin{small}
\begin{aligned}
&\quad \mathbb{E}\left[ \|\bar{\bm{x}}(k+1) - \bm{x}^{*}\|^{2} | \mathcal{M}_{k} \right] \\
&\le (1-\alpha_{k+1}\mu) \mathbb{E}\left[ \|\bar{\bm{x}}(k) - \bm{x}^{*}\|^{2} | \mathcal{M}_{k}  \right] \\
&\quad + \frac{L^{2}\alpha_{k+1}}{n\mu} \mathbb{E}\left[  \|\mathbf{z}(k+1) - \mathbf{1}\bar{\bm{x}}(k)\|^{2} | \mathcal{M}_{k} \right] \\
&\quad + (2p+4pL\alpha_{k+1})\nu_{k+1}^{2} \\
&\le (1-\alpha_{k+1}\mu) \mathbb{E}\left[ \|\bar{\bm{x}}(k) - \bm{x}^{*}\|^{2} | \mathcal{M}_{k}  \right] \\
&\quad + \frac{L^{2}\alpha_{k+1}}{n\mu} \mathbb{E}\left[  \|\mathbf{x}(k) - \mathbf{1}\bar{\bm{x}}(k)\|^{2} | \mathcal{M}_{k} \right] \\
&\quad + 2p(1+2L\alpha_{k+1}+\frac{L^{2}\alpha_{k+1}}{\mu})\nu_{k+1}^{2},
\end{aligned}
\end{small}
\end{equation}
where the second inequality follows from the part \textit{(f)} of Lemma \ref{lemma:xi}. Taking full expectation on both sides of the above inequality completes the proof.
\end{proof}

\subsection{Proof of Lemma \ref{lemma:second_inequality}}
\begin{proof}
Let $\bm{a}_{2} \triangleq \tilde{W}(k)\mathbf{x}(k) - \mathbf{1}\bar{\bm{x}}(k)$,
$\bm{b}_{2} \triangleq \frac{\alpha_{k+1}}{\alpha_{k}}W(\mathbf{x}(k) - \mathbf{x}(k-1))$, $\bm{c}_{2} \triangleq \alpha_{k+1}(\nabla F(\mathbf{z}(k+1)) - \nabla F(\mathbf{z}(k))) - \frac{\alpha_{k+1}}{n}\mathbf{1}\mathbf{1}^{T}\nabla F(\mathbf{z}(k+1))$,
and $\bm{d}_{2} \triangleq \left( (1+\alpha_{k+1}\beta)W - \alpha_{k+1}\beta I - \frac{1}{n}\mathbf{1}\mathbf{1}^{T} \right)\bm{\xi}(k+1) - \frac{\alpha_{k+1}}{\alpha_{k}}W\bm{\xi}(k)$. In light of (\ref{eq:x_x_bar}), we have
\begin{equation}\label{eq:x_x_bar_initial}
\begin{small}
\begin{aligned}
&\quad \mathbb{E}\left[\|\mathbf{x}(k+1) - \mathbf{1}\bar{\bm{x}}(k+1)\|^{2}|\mathcal{M}_{k}\right] \\
&= \mathbb{E}\left[\|\bm{a}_{2}\|^{2}|\mathcal{M}_{k}\right] + \mathbb{E}\left[\|\bm{b}_{2}\|^{2}|\mathcal{M}_{k}\right] + \mathbb{E}\left[\|\bm{c}_{2}\|^{2}|\mathcal{M}_{k}\right] \\
&\quad + \mathbb{E}\left[\|\bm{d}_{2}\|^{2}|\mathcal{M}_{k}\right] + \mathbb{E}\left[2\langle \bm{a}_{2},\bm{b}_{2}\rangle|\mathcal{M}_{k}\right] + \mathbb{E}\left[2\langle \bm{a}_{2},-\bm{c}_{2}\rangle|\mathcal{M}_{k}\right] \\
&\quad + \mathbb{E}\left[2\langle \bm{b}_{2},-\bm{c}_{2}\rangle|\mathcal{M}_{k}\right] + \mathbb{E}\left[2\langle \bm{a}_{2},\bm{d}_{2}\rangle|\mathcal{M}_{k}\right] \\
&\quad + \mathbb{E}\left[2\langle \bm{b}_{2},\bm{d}_{2}\rangle|\mathcal{M}_{k}\right] + \mathbb{E}\left[2\langle -\bm{c}_{2},\bm{d}_{2}\rangle|\mathcal{M}_{k}\right].
\end{aligned}
\end{small}
\end{equation}

Next, we will bound each term on the right-hand side of the above equality, respectively. From Remark \ref{remark:W_tilde}, it follows that $\tilde{W}(k)$ is doubly-stochastic and $\sigma_{k} = \| \tilde{W}(k) - \frac{1}{n}\mathbf{1}\mathbf{1}^{T} \| < 1$ for all $k\ge1$. Thus, we have
\begin{equation}\label{eq:sigma}
\begin{small}
\begin{aligned}
%\mathbb{E}\left[ \|\bm{a}_{2}\|^{2} | \mathcal{M}_{k} \right] &= \mathbb{E}\left[ \left\| \tilde{W}(k)\mathbf{x}(k) - \mathbf{1}\bar{\bm{x}}(k) \right\|^{2} | \mathcal{M}_{k} \right] \\
\mathbb{E}\left[ \|\bm{a}_{2}\|^{2} | \mathcal{M}_{k} \right] &= \mathbb{E}\Big[ \Big\| \Big(\tilde{W}(k) - \frac{1}{n}\mathbf{1}\mathbf{1}^{T}\Big)(\mathbf{x}(k) - \mathbf{1}\bar{\bm{x}}(k)) \Big\|^{2} | \mathcal{M}_{k} \Big] \\
&\le \sigma^{2} \mathbb{E}\left[ \|\mathbf{x}(k) - \mathbf{1}\bar{\bm{x}}(k)\|^{2} | \mathcal{M}_{k} \right],
\end{aligned}
\end{small}
\end{equation}
where $\sigma \triangleq \max_{k\ge1}\{\sigma_{k}\} < 1$. For the second term in (\ref{eq:x_x_bar_initial}), we have
%Since $\alpha_{k} = \gamma q_{1}^{k-1}$ and $\|W\|=1$, we have
\begin{equation}\label{eq:second_b}
\begin{small}
\begin{aligned}
%\mathbb{E}\left[ \|\bm{b}_{2}\|^{2} | \mathcal{M}_{k} \right] &= \frac{\alpha_{k+1}^{2}}{\alpha_{k}^{2}} \mathbb{E}\left[ \left\| W(\mathbf{x}(k) - \mathbf{x}(k-1)) \right\|^{2} | \mathcal{M}_{k} \right] \\
\mathbb{E}\left[ \|\bm{b}_{2}\|^{2} | \mathcal{M}_{k} \right] &\le q_{1}^{2}\mathbb{E}\left[ \|\mathbf{x}(k) - \mathbf{x}(k-1) \|^{2} | \mathcal{M}_{k} \right].
\end{aligned}
\end{small}
\end{equation}

It follows from Lemma \ref{lemma_smooth_convex} that
%Based on Assumption \ref{assump:smooth_convex}, it follows that
\begin{equation}\label{equation_nabla_F}
\begin{footnotesize}
\begin{aligned}
&\quad \Big\|\frac{1}{n}\bm{1}^{T}\nabla F(\mathbf{z}(k+1))\Big\|^{2} \\
&= \Big\|\frac{1}{n}\bm{1}^{T}\nabla F(\mathbf{z}(k+1))-\frac{1}{n}\bm{1}^{T}\nabla F(\mathbf{1}\bar{\bm{x}}(k))+\frac{1}{n}\bm{1}^{T}\nabla F(\mathbf{1}\bar{\bm{x}}(k)) \Big\|^{2} \\
&\le 2\Big\|\frac{1}{n}\bm{1}^{T}\nabla F(\mathbf{z}(k+1)) - \frac{1}{n}\bm{1}^{T}\nabla F(\mathbf{1}\bar{\bm{x}}(k)) \Big\|^{2} \\
&\quad + 2\Big\|\frac{1}{n}\bm{1}^{T}\nabla F(\mathbf{1}\bar{\bm{x}}(k)) \Big\|^{2} \\
&\le \frac{2L^{2}}{n}\| \mathbf{z}(k+1) - \mathbf{1}\bar{\bm{x}}(k) \|^{2} + 2L^{2}\| \bar{\bm{x}}(k) - \bm{x}^{*} \|^{2}.
\end{aligned}
\end{footnotesize}
\end{equation}
Then, we have
\begin{equation}\label{eq:second_c}
\begin{small}
\begin{aligned}
&\quad \mathbb{E}\left[ \|\bm{c}_{2}\|^{2} | \mathcal{M}_{k} \right] \\
&\le 2\alpha_{k+1}^{2} \mathbb{E}\left[ \left\| \nabla F(\mathbf{z}(k+1)) - \nabla F(\mathbf{z}(k)) \right\|^{2} | \mathcal{M}_{k} \right] \\
&\quad + 2n\alpha_{k+1}^{2} \mathbb{E}\Big[ \Big\|\frac{1}{n}\bm{1}^{T}\nabla F(\mathbf{z}(k+1))\Big\|^{2} | \mathcal{M}_{k} \Big] \\
&\le 2L^{2}\alpha_{k+1}^{2} \mathbb{E}\left[ \left\|\mathbf{z}(k+1) - \mathbf{z}(k)\right\|^{2} | \mathcal{M}_{k}\right] \\
&\quad + 4L^{2}\alpha_{k+1}^{2} \mathbb{E}\left[\left\|\mathbf{z}(k+1) - \mathbf{1}\bar{\bm{x}}(k)\right\|^{2}| \mathcal{M}_{k}\right] \\
&\quad + 4nL^{2}\alpha_{k+1}^{2}\mathbb{E}\left[ \left\|\bar{\bm{x}}(k) - \bm{x}^{*}\right\|^{2}| \mathcal{M}_{k}\right] \\
&\le 2L^{2}\alpha_{k+1}^{2}\mathbb{E}\left[\|\mathbf{x}(k)-\mathbf{x}(k-1)\|^{2}|\mathcal{M}_{k}\right] \\
&\quad + 4L^{2}\alpha_{k+1}^{2}\mathbb{E}\left[\left\|\mathbf{x}(k)-\mathbf{1}\bar{\bm{x}}(k)\right\|^{2}|\mathcal{M}_{k}\right] \\
&\quad  + 4nL^{2}\alpha_{k+1}^{2}\mathbb{E}\left[\left\|\bar{\bm{x}}(k)-\bm{x}^{*}\right\|^{2}|\mathcal{M}_{k}\right] \\
&\quad + 12npL^{2}\alpha_{k+1}^{2}\nu_{k+1}^{2} + 4npL^{2}(7+2L\alpha_{k})\alpha_{k+1}^{2}\nu_{k}^{2},
%8npL^{2}(5+L\alpha_{k})\alpha_{k+1}^{2}\nu_{k}^{2}
\end{aligned}
\end{small}
\end{equation}
where the last inequality follows from the parts \textit{(f)} and \textit{(g)} of Lemma \ref{lemma:xi}.

For the fourth term of (\ref{eq:x_x_bar_initial}), since $\bm{\xi}(k+1)$ and $\bm{\xi}(k)$ are independent of each other, we have
\begin{equation}
\begin{small}
\begin{aligned}
\mathbb{E}\left[ \|\bm{d}_{2}\|^{2}|\mathcal{M}_{k}\right]
%&= \mathbb{E}\left[ \left\| \left( (1+\alpha_{k+1}\beta)W - \alpha_{k+1}\beta I - \frac{1}{n}\mathbf{1}\mathbf{1}^{T} \right)\bm{\xi}(k+1) - \frac{\alpha_{k+1}}{\alpha_{k}}W\bm{\xi}(k) \right\|^{2}|\mathcal{M}_{k}\right] \\
&\le 9\mathbb{E}\left[\left\|\bm{\xi}(k+1)\right\|^{2}|\mathcal{M}_{k}\right] + \mathbb{E}\left[ \left\|\bm{\xi}(k)\right\|^{2}|\mathcal{M}_{k}\right] \\
&= 18np\nu_{k+1}^{2} + 2np\nu_{k}^{2}.
\end{aligned}
\end{small}
\end{equation}

%与1的证明类似，再基于a，我们有
Similar to the proof in (\ref{eq:sigma}), it follows that
\begin{equation}
\begin{footnotesize}
\begin{aligned}
%2\langle \bm{a}_{2},\bm{b}_{2}\rangle &\le 2\|\tilde{W}(k)\mathbf{x}(k) - \mathbf{1}\bar{\bm{x}}(k)\|\|\bm{b}_{2}\| \\
&\quad \mathbb{E}\left[ 2\langle \bm{a}_{2},\bm{b}_{2}\rangle |\mathcal{M}_{k}\right] \\
&\le 2\sigma \mathbb{E}\left[ \|\mathbf{x}(k) - \mathbf{1}\bar{\bm{x}}(k)\|\| \bm{b}_{2} \| |\mathcal{M}_{k}\right] \\
&\le \sigma \Big( \frac{1-\sigma^{2}}{3\sigma} \mathbb{E}\left[ \| \mathbf{x}(k) - \mathbf{1}\bar{\bm{x}}(k) \|^{2} |\mathcal{M}_{k}\right] + \frac{3\sigma}{1-\sigma^{2}} \mathbb{E}\left[ \|\bm{b}_{2}\|^{2} |\mathcal{M}_{k}\right] \Big) \\
&= \frac{1-\sigma^{2}}{3} \mathbb{E}\left[ \|\mathbf{x}(k) - \mathbf{1}\bar{\bm{x}}(k)\|^{2} |\mathcal{M}_{k}\right] + \frac{3\sigma^{2}}{1-\sigma^{2}} \mathbb{E}\left[ \|\bm{b}_{2}\|^{2} |\mathcal{M}_{k}\right],
\end{aligned}
\end{footnotesize}
\end{equation}
where in the second inequality we use Young's Inequality.

Similarly, we have
\begin{equation}
\begin{small}
\begin{aligned}
&\quad\mathbb{E}\left[2\langle \bm{a}_{2},-\bm{c}_{2}\rangle|\mathcal{M}_{k}\right] \\
&\le \frac{1 - \sigma^{2}}{3}\mathbb{E}\left[\|\mathbf{x}(k) - \mathbf{1}\bar{\bm{x}}(k)\|^{2}|\mathcal{M}_{k}\right] + \frac{3\sigma^{2}}{1 - \sigma^{2}}\mathbb{E}\left[\|\bm{c}_{2}\|^{2}|\mathcal{M}_{k}\right],
\end{aligned}
\end{small}
\end{equation}
and
\begin{equation}
\begin{small}
\begin{aligned}
\mathbb{E}\left[ 2\langle \bm{b}_{2},-\bm{c}_{2}\rangle|\mathcal{M}_{k} \right] &\le \mathbb{E}\left[  \|\bm{b}_{2}\|^{2} | \mathcal{M}_{k} \right] + \mathbb{E}\left[ \|\bm{c}_{2}\|^{2} | \mathcal{M}_{k} \right].
\end{aligned}
\end{small}
\end{equation}

For the last three terms of (\ref{eq:x_x_bar_initial}), it follows from Lemma \ref{lemma:xi} that
\begin{equation}
\begin{small}
\begin{aligned}
\mathbb{E}\left[ 2\langle \bm{a}_{2},\bm{d}_{2}\rangle|\mathcal{M}_{k}\right] &= \frac{2\alpha_{k+1}}{\alpha_{k}}\mathbb{E}\left[ \Big\langle \tilde{W}(k)\mathbf{x}(k),-W\bm{\xi}(k) \Big\rangle|\mathcal{M}_{k}\right] \\
&\quad + \frac{2\alpha_{k+1}}{\alpha_{k}}\mathbb{E}\left[ \left\langle\mathbf{1}\bar{\bm{x}}(k),W\bm{\xi}(k)\right\rangle|\mathcal{M}_{k}\right]\\
&\le 8np(2+L\alpha_{k})\nu_{k}^{2},
\end{aligned}
\end{small}
\end{equation}
\begin{equation}
\begin{small}
\begin{aligned}
\mathbb{E}\left[ 2\langle \bm{b}_{2},\bm{d}_{2}\rangle|\mathcal{M}_{k}\right] &= \frac{2\alpha_{k+1}^{2}}{\alpha_{k}^{2}} \mathbb{E}\left[\left\langle W\mathbf{x}(k), -W\bm{\xi}(k)\right\rangle|\mathcal{M}_{k}\right] \\
%&\le \frac{4np\|W\|^{2}(3+L\alpha_{k})\alpha_{k+1}^{2}\nu_{k}^{2}}{\alpha_{k}^{2}}\\
&\le 4np(3+L\alpha_{k})\nu_{k}^{2},
\end{aligned}
\end{small}
\end{equation}
\begin{equation}\label{eq:second_cd}
\begin{small}
\begin{aligned}
&\quad \mathbb{E}\left[ 2\langle -\bm{c}_{2},\bm{d}_{2}\rangle|\mathcal{M}_{k}\right] \\
&= 2\alpha_{k+1}\mathbb{E}\Big[ \Big\langle (\frac{1}{n}\mathbf{1}\mathbf{1}^{T}-I)\nabla F(\mathbf{z}(k+1)),\\
&\qquad \Big( (1+\alpha_{k+1}\beta)W - \alpha_{k+1}\beta I - \frac{1}{n}\mathbf{1}\mathbf{1}^{T}\Big)\bm{\xi}(k+1)\Big\rangle|\mathcal{M}_{k}\Big] \\
&\quad + 2\frac{\alpha_{k+1}^{2}}{\alpha_{k}}\mathbb{E}\Big[ \Big\langle(\frac{1}{n}\mathbf{1}\mathbf{1}^{T}-I)\nabla F(\mathbf{z}(k+1)),-W\bm{\xi}(k) \Big\rangle|\mathcal{M}_{k}\Big]\\
&\quad + 2\frac{\alpha_{k+1}^{2}}{\alpha_{k}}\mathbb{E}\left[ \left\langle \nabla F(\mathbf{z}(k)),-W\bm{\xi}(k)\right\rangle|\mathcal{M}_{k}\right] \\
&\le 12npL\alpha_{k+1}\nu_{k+1}^{2} + 2npL(4+np+2L\alpha_{k})\alpha_{k+1}\nu_{k}^{2}.
\end{aligned}
\end{small}
\end{equation}

Finally, combining (\ref{eq:x_x_bar_initial})-(\ref{eq:second_b}) and (\ref{eq:second_c})-(\ref{eq:second_cd}), we have
\begin{equation}
\begin{small}
\begin{aligned}
&\quad\mathbb{E}\left[\|\mathbf{x}(k+1) - \mathbf{1}\bar{\bm{x}}(k+1)\|^{2}|\mathcal{M}_{k}\right] \\
&\le \frac{2+\sigma^{2}}{3}\mathbb{E}\left[\|\mathbf{x}(k) - \mathbf{1}\bar{\bm{x}}(k)\|^{2}|\mathcal{M}_{k}\right] + \frac{2+\sigma^{2}}{1-\sigma^{2}}\mathbb{E}\left[ \|\bm{b}_{2}\|^{2}|\mathcal{M}_{k}\right] \\
&\quad + \frac{2+\sigma^{2}}{1-\sigma^{2}}\mathbb{E}\left[\|\bm{c}_{2}\|^{2}|\mathcal{M}_{k}\right] + 2np(9+6L\alpha_{k+1})\nu_{k+1}^{2} \\
&\quad + 2np(15+6L\alpha_{k}+4L\alpha_{k+1}+npL\alpha_{k+1}+2L^{2}\alpha_{k}\alpha_{k+1})\nu_{k}^{2} \\
&\le \Big(\frac{2+\sigma^{2}}{3}+\frac{4(2+\sigma^{2})L^{2}\alpha_{k+1}^{2}}{1-\sigma^{2}}\Big)\mathbb{E}\left[\|\mathbf{x}(k) - \mathbf{1}\bar{\bm{x}}(k)\|^{2}|\mathcal{M}_{k}\right] \\
&\quad + \frac{4n(2+\sigma^{2})L^{2}\alpha_{k+1}^{2}}{1-\sigma^{2}}\mathbb{E}\left[\left\|\bar{\bm{x}}(k)-\bm{x}^{*}\right\|^{2}|\mathcal{M}_{k}\right] \\
&\quad + \frac{(2+\sigma^{2})(q_{1}^{2}+2L^{2}\alpha_{k+1}^{2})}{1-\sigma^{2}}\mathbb{E}\left[\|\mathbf{x}(k)-\mathbf{x}(k-1)\|^{2}|\mathcal{M}_{k} \right] \\
&\quad + 2np\Big(9+6L\alpha_{k+1}+\frac{18L^{2}\alpha_{k+1}^{2}}{1-\sigma^{2}}\Big)\nu_{k+1}^{2} \\
&\quad + 2np\Big(15+6L\alpha_{k}+4L\alpha_{k+1}+npL\alpha_{k+1}+2L^{2}\alpha_{k}\alpha_{k+1}\\
&\qquad +\frac{6L^{2}(7+2L\alpha_{k})\alpha_{k+1}^{2}}{1-\sigma^{2}}\Big)\nu_{k}^{2}.
%12np\left(4+3L\alpha_{k}+\frac{2(5+L\alpha_{k})L^{2}\alpha_{k+1}^{2}}{1-\sigma^{2}}\right)\nu_{k}^{2}
\end{aligned}
\end{small}
\end{equation}
Taking the full expectation for the above inequality gives the desired result.
\end{proof}

\subsection{Proof of Lemma \ref{lemma:third_inequality}}
\begin{proof}
In light of (\ref{eq:x_x_k}), and let
$\bm{a}_{3} \triangleq \frac{\alpha_{k+1}}{\alpha_{k}}W(\mathbf{x}(k) - \mathbf{x}(k-1))$,
$\bm{b}_{3} \triangleq (1-\frac{\alpha_{k+1}}{\alpha_{k}}+\alpha_{k+1}\beta)\left(W-I\right)\mathbf{x}(k)$,
$\bm{c}_{3} \triangleq \alpha_{k+1}(\nabla F(\mathbf{z}(k+1)) - \nabla F(\mathbf{z}(k)))$,
$\bm{d}_{3} \triangleq \left((1+\alpha_{k+1}\beta)W - \alpha_{k+1}\beta I\right)\bm{\xi}(k+1) - \frac{\alpha_{k+1}}{\alpha_{k}}W\bm{\xi}(k)$, we have

\begin{equation}\label{eq:x_x_k_initial}
\begin{small}
\begin{aligned}
&\quad \mathbb{E}\left[\|\mathbf{x}(k+1) - \mathbf{x}(k)\|^{2}|\mathcal{M}_{k}\right] \\
&= \mathbb{E}\left[\|\bm{a}_{3}\|^{2}|\mathcal{M}_{k}\right] + \mathbb{E}\left[\|\bm{b}_{3}\|^{2}|\mathcal{M}_{k}\right] + \mathbb{E}\left[\|\bm{c}_{3}\|^{2}|\mathcal{M}_{k}\right] \\
&\quad + \mathbb{E}\left[\|\bm{d}_{3}\|^{2}|\mathcal{M}_{k}\right] + \mathbb{E}\left[2\langle \bm{a}_{3},\bm{b}_{3}\rangle|\mathcal{M}_{k}\right] + \mathbb{E}\left[2\langle \bm{a}_{3},-\bm{c}_{3}\rangle|\mathcal{M}_{k}\right] \\
&\quad + \mathbb{E}\left[2\langle \bm{b}_{3},-\bm{c}_{3}\rangle|\mathcal{M}_{k}\right] + \mathbb{E}\left[2\langle \bm{a}_{3},\bm{d}_{3}\rangle|\mathcal{M}_{k}\right] \\
&\quad + \mathbb{E}\left[2\langle \bm{b}_{3},\bm{d}_{3}\rangle|\mathcal{M}_{k}\right] + \mathbb{E}\left[2\langle -\bm{c}_{3},\bm{d}_{3}\rangle|\mathcal{M}_{k}\right].
\end{aligned}
\end{small}
\end{equation}

%\begin{equation}
%\begin{aligned}
%\| \mathbf{1}\bar{\bm{x}}(k-1) - \mathbf{1}\bar{\bm{x}}(k-2) \|^{2} &= \left\| \frac{\alpha_{k-1}}{n}\mathbf{1}\mathbf{1}^{T} \nabla F(\mathbf{z}(k-1)) - \frac{1}{n}\mathbf{1}\mathbf{1}^{T}\bm{\xi}(k-1) \right\|^{2} \\
%&\le \frac{\alpha_{k-1}^{2}}{n}\left\| \mathbf{1}^{T} \nabla F(\mathbf{z}(k-1)) \right\|^{2} + \left\| \frac{1}{n}\mathbf{1}\mathbf{1}^{T}\bm{\xi}(k-1) \right\|^{2} \\
%&\le  2\alpha_{k-1}^{2}L^{2}\| \mathbf{x}(k-1) - \mathbf{x}(k-2) \|^{2} + 4\alpha_{k-1}^{2}L^{2}\|\mathbf{x}(k-1) - \mathbf{1}\bar{\bm{x}}(k-1) \|^{2} \\
%&\quad + 4n\alpha_{k-1}^{2}L^{2}\| \bar{\bm{x}}(k-1) - \bm{x}^{*} \|^{2} + 2np(1 + 8nL^{2})\nu_{k-1}^{2}.
%\end{aligned}
%\end{equation}

%与引理6的证明类似，我们将首先bound上式右边的每一项。
Similar to the proof of Lemma \ref{lemma:second_inequality}, we will bound each term on the right-hand side of the above equality. For the first term, it follows from (\ref{equation_bar_x}) that
\begin{equation}\label{equation_W_x}
\begin{small}
\begin{aligned}
&\quad W(\mathbf{x}(k) - \mathbf{x}(k-1)) \\
&= \Big(W - \frac{1}{n}\mathbf{1}\mathbf{1}^{T}\Big)(\mathbf{x}(k) - \mathbf{x}(k-1)) + \mathbf{1}\bar{\bm{x}}(k) - \mathbf{1}\bar{\bm{x}}(k-1) \\
&= \Big(W - \frac{1}{n}\mathbf{1}\mathbf{1}^{T}\Big)(\mathbf{x}(k) - \mathbf{x}(k-1)) - \frac{\alpha_{k}}{n}\mathbf{1}\mathbf{1}^{T} \nabla F(\mathbf{z}(k)) \\
&\quad + \frac{1}{n}\mathbf{1}\mathbf{1}^{T}\bm{\xi}(k).
\end{aligned}
\end{small}
\end{equation}

Then, we obtain
\begin{equation}\label{ieq:third_a}
\begin{footnotesize}
\begin{aligned}
&\quad \mathbb{E}\left[\|\bm{a}_{3} \|^{2}|\mathcal{M}_{k}\right] \\
&\le \sigma^{2}\mathbb{E}\left[\|\mathbf{x}(k)-\mathbf{x}(k-1)\|^{2}|\mathcal{M}_{k}\right] + \mathbb{E}\Big[\Big\| \frac{1}{n}\mathbf{1}\mathbf{1}^{T}\bm{\xi}(k)\Big\|^{2}|\mathcal{M}_{k}\Big] \\
&\quad + \alpha_{k+1}^{2}\mathbb{E}\Big[\Big\|\frac{1}{n}\mathbf{1}\mathbf{1}^{T} \nabla F(\mathbf{z}(k))\Big\|^{2}|\mathcal{M}_{k}\Big] \\
&\quad + 2q_{1}\mathbb{E}\Big[\Big\langle(W-\frac{1}{n}\mathbf{1}\mathbf{1}^{T})(\mathbf{x}(k)-\mathbf{x}(k-1)), \\
&\qquad -\frac{\alpha_{k+1}}{n}\mathbf{1}\mathbf{1}^{T}\nabla F(\mathbf{z}(k))\Big\rangle|\mathcal{M}_{k}\Big] \\
&\quad + 2q_{1}^{2}\mathbb{E}\Big[\Big\langle(W-\frac{1}{n}\mathbf{1}\mathbf{1}^{T})\mathbf{x}(k),\frac{1}{n}\mathbf{1}\mathbf{1}^{T}\bm{\xi}(k)\Big\rangle|\mathcal{M}_{k}\Big] \\
&\quad + 2q_{1}\mathbb{E}\Big[\Big\langle-\frac{\alpha_{k+1}}{n}\mathbf{1}\mathbf{1}^{T}\nabla F(\mathbf{z}(k)),\frac{1}{n}\mathbf{1}\mathbf{1}^{T}\bm{\xi}(k)\Big\rangle|\mathcal{M}_{k}\Big] \\
&\le \sigma^{2}\mathbb{E}\left[\|\mathbf{x}(k)-\mathbf{x}(k-1)\|^{2}|\mathcal{M}_{k}\right] + \mathbb{E}\Big[\Big\|\frac{\alpha_{k+1}}{n}\mathbf{1}\mathbf{1}^{T} \nabla F(\mathbf{z}(k))\Big\|^{2}|\mathcal{M}_{k}\Big] \\
&\quad + \sigma\Big( \frac{1 - \sigma^{2}}{4\sigma}\mathbb{E}\left[\|\mathbf{x}(k) -\mathbf{x}(k-1)\|^{2}|\mathcal{M}_{k}\right] \\
&\qquad + \frac{4\sigma}{1 - \sigma^{2}} \mathbb{E}\Big[ \left\| \frac{\alpha_{k+1}}{n}\mathbf{1}\mathbf{1}^{T} \nabla F(\mathbf{z}(k))\right\|^{2}|\mathcal{M}_{k}\Big] \Big)\\
&\quad + 2np(7+2L\alpha_{k}+2L\alpha_{k+1})\nu_{k}^{2} \\
&= \frac{1+3\sigma^{2}}{4}\mathbb{E}\left[\|\mathbf{x}(k) -\mathbf{x}(k-1)\|^{2}|\mathcal{M}_{k}\right] \\
&\quad + \frac{1+3\sigma^{2}}{1-\sigma^{2}}\mathbb{E}\Big[ \Big\| \frac{\alpha_{k+1}}{n}\mathbf{1}\mathbf{1}^{T} \nabla F(\mathbf{z}(k))\Big\|^{2}|\mathcal{M}_{k}\Big] \\
&\quad + 2np(7+2L\alpha_{k}+2L\alpha_{k+1})\nu_{k}^{2},
\end{aligned}
\end{footnotesize}
\end{equation}
where the first inequality is due to the fact that $\|W - \frac{1}{n}\mathbf{1}\mathbf{1}^{T}\|\le\sigma<1$, and the second inequality follows from Young's Inequality as well as parts \textit{(b)} and \textit{(d)} of Lemma \ref{lemma:xi}.

For the second term in (\ref{eq:x_x_k_initial}), we have
%Since $\alpha_{k+1}\beta=\gamma\beta q_{1}^{k}\le q_{1}$, we have
\begin{equation}
\begin{small}
\begin{aligned}
&\quad \mathbb{E}\left[\|\bm{b}_{3}\|^{2}|\mathcal{M}_{k}\right] \\
& = \left(1-q_{1}+\alpha_{k+1}\beta\right)^{2}\mathbb{E}\left[\left\|(W-I)(\mathbf{x}(k)-\mathbf{1}\bar{\bm{x}}(k)) \right\|^{2}|\mathcal{M}_{k}\right] \\
& \le \|W-I\|^{2}\mathbb{E}\left[\|\mathbf{x}(k)-\mathbf{1}\bar{\bm{x}}(k)\|^{2}|\mathcal{M}_{k}\right].
\end{aligned}
\end{small}
\end{equation}

%对于第三项，我们有
By part \textit{(g)} of Lemma \ref{lemma:xi}, we obtain
\begin{equation}
\begin{footnotesize}
\begin{aligned}
\mathbb{E}\left[\|\bm{c}_{3}\|^{2}|\mathcal{M}_{k}\right]
&\le L^{2}\alpha_{k+1}^{2}\mathbb{E}\left[\|\mathbf{z}(k+1)-\mathbf{z}(k)\|^{2}|\mathcal{M}_{k}\right] \\
&\le L^{2}\alpha_{k+1}^{2}\mathbb{E}\left[\|\mathbf{x}(k) - \mathbf{x}(k-1)\|^{2}|\mathcal{M}_{k}\right] \\
&\quad + 2npL^{2}\alpha_{k+1}^{2}\nu_{k+1}^{2} + 2npL^{2}(7+2L\alpha_{k})\alpha_{k+1}^{2}\nu_{k}^{2}.
%4np(4+L\alpha_{k})L^{2}\alpha_{k+1}^{2}\nu_{k}^{2}.
\end{aligned}
\end{footnotesize}
\end{equation}

For the fourth term in (\ref{eq:x_x_k_initial}), since $\bm{\xi}(k+1)$ and $\bm{\xi}(k)$ are independent of each other, we have
\begin{equation}
\begin{footnotesize}
\begin{aligned}
\mathbb{E}\left[\|\bm{d}_{3}\|^{2}|\mathcal{M}_{k}\right]
&= \mathbb{E}\left[\left\|\left((1+\alpha_{k+1}\beta)W - \alpha_{k+1}\beta I\right)\bm{\xi}(k+1)\right\|^{2}|\mathcal{M}_{k}\right] \\
&\quad + \frac{\alpha_{k+1}^{2}}{\alpha_{k}^{2}}\mathbb{E}\left[\left\|W\bm{\xi}(k)\right\|^{2}|\mathcal{M}_{k}\right]\\
&\le 18np\nu_{k+1}^{2} + 2np\nu_{k}^{2}.
\end{aligned}
\end{footnotesize}
\end{equation}

%对于第五项，由1、引理1和a，我们有
For the fifth term in (\ref{eq:x_x_k_initial}), it follows from (\ref{equation_W_x}) that
\begin{equation}\label{equation_a3_b3}
\begin{footnotesize}
\begin{aligned}
&\quad \mathbb{E}\left[2\langle \bm{a}_{3},\bm{b}_{3}\rangle|\mathcal{M}_{k}\right] \\
&= 2q_{1}\mathbb{E}\Big[\Big\langle(W - \frac{1}{n}\mathbf{1}\mathbf{1}^{T})(\mathbf{x}(k) - \mathbf{x}(k-1)), \bm{b}_{3}\Big\rangle|\mathcal{M}_{k}\Big] \\
&\quad + 2\mathbb{E}\left[\left\langle-\frac{\alpha_{k+1}}{n}\mathbf{1}\mathbf{1}^{T} \nabla F(\mathbf{z}(k)), \bm{b}_{3}\right\rangle|\mathcal{M}_{k}\right] \\
&\quad + 2q_{1}\mathbb{E}\Big[\Big\langle\frac{1}{n}\mathbf{1}\mathbf{1}^{T}\bm{\xi}(k), (\tilde{W}(k) - I)\mathbf{x}(k)\Big\rangle|\mathcal{M}_{k}\Big] \\
&\le \sigma\Big(\frac{1 - \sigma^{2}}{4\sigma}\mathbb{E}\left[\|\mathbf{x}(k)-\mathbf{x}(k-1) \|^{2}|\mathcal{M}_{k}\right] + \frac{4\sigma}{1 - \sigma^{2}}\mathbb{E}\left[\| \bm{b}_{3}\|^{2}|\mathcal{M}_{k}\right]\Big) \\
&\quad + \mathbb{E}\left[\| \bm{b}_{3}\|^{2}|\mathcal{M}_{k}\right] + \mathbb{E}\Big[ \Big\| \frac{\alpha_{k+1}}{n}\mathbf{1}\mathbf{1}^{T} \nabla F(\mathbf{z}(k))\Big\|^{2}|\mathcal{M}_{k}\Big] \\
&\quad + 8np(3+L\alpha_{k})\nu_{k}^{2}\\
&= \frac{1 - \sigma^{2}}{4}\mathbb{E}\left[\|\mathbf{x}(k)-\mathbf{x}(k-1) \|^{2}|\mathcal{M}_{k}\right] + \frac{1+3\sigma^{2}}{1-\sigma^{2}}\mathbb{E}\left[\| \bm{b}_{3}\|^{2}|\mathcal{M}_{k}\right] \\
&\quad + \mathbb{E}\Big[ \Big\| \frac{\alpha_{k+1}}{n}\mathbf{1}\mathbf{1}^{T} \nabla F(\mathbf{z}(k))\Big\|^{2}|\mathcal{M}_{k}\Big] + 8np(3+L\alpha_{k})\nu_{k}^{2},
\end{aligned}
\end{footnotesize}
\end{equation}
where the first inequality follows from Young's Inequality and part \textit{(d)} of Lemma \ref{lemma:xi}.

Similar to the proof in (\ref{equation_a3_b3}), we have
\begin{equation}
\begin{footnotesize}
\begin{aligned}
&\quad \mathbb{E}\left[2\langle \bm{a}_{3},-\bm{c}_{3}\rangle|\mathcal{M}_{k}\right] \\
&= 2q_{1}\mathbb{E}\Big[\Big\langle(W - \frac{1}{n}\mathbf{1}\mathbf{1}^{T})(\mathbf{x}(k) - \mathbf{x}(k-1)), -\bm{c}_{3}\Big\rangle|\mathcal{M}_{k}\Big] \\
&\quad + 2\mathbb{E}\left[\left\langle\frac{\alpha_{k+1}}{n}\mathbf{1}\mathbf{1}^{T} \nabla F(\mathbf{z}(k)), \bm{c}_{3}\right\rangle|\mathcal{M}_{k}\right] \\
&\quad + 2q_{1}\alpha_{k+1}\mathbb{E}\Big[\Big\langle\frac{1}{n}\mathbf{1}\mathbf{1}^{T}\bm{\xi}(k), \nabla F(\mathbf{z}(k))-\nabla F(\mathbf{z}(k+1))\Big\rangle|\mathcal{M}_{k}\Big] \\
&\le \sigma\Big(\frac{1 - \sigma^{2}}{4\sigma}\mathbb{E}\left[\|\mathbf{x}(k)-\mathbf{x}(k-1) \|^{2}|\mathcal{M}_{k}\right] + \frac{4\sigma}{1 - \sigma^{2}}\mathbb{E}\left[\|\bm{c}_{3}\|^{2}|\mathcal{M}_{k}\right]\Big) \\
&\quad + \mathbb{E}\left[ \left\| \frac{\alpha_{k+1}}{n}\mathbf{1}\mathbf{1}^{T} \nabla F(\mathbf{z}(k))\right\|^{2}|\mathcal{M}_{k}\right] + \mathbb{E}\left[\|\bm{c}_{3}\|^{2}|\mathcal{M}_{k}\right] \\
&\quad + 2npL(4+np+2L\alpha_{k})\alpha_{k+1}\nu_{k}^{2} \\
&= \frac{1 - \sigma^{2}}{4}\mathbb{E}\left[\|\mathbf{x}(k)-\mathbf{x}(k-1) \|^{2}|\mathcal{M}_{k}\right] + \frac{1+3\sigma^{2}}{1-\sigma^{2}}\mathbb{E}\left[\|\bm{c}_{3}\|^{2}|\mathcal{M}_{k}\right] \\
&\quad + \mathbb{E}\left[ \left\| \frac{\alpha_{k+1}}{n}\mathbf{1}\mathbf{1}^{T} \nabla F(\mathbf{z}(k))\right\|^{2}|\mathcal{M}_{k}\right] \\
&\quad + 2npL(4+np+2L\alpha_{k})\alpha_{k+1}\nu_{k}^{2},
\end{aligned}
\end{footnotesize}
\end{equation}
where the first inequality follows from the parts \textit{(b)} and \textit{(c)} of Lemma \ref{lemma:xi}. In addition,
\begin{equation}
\begin{small}
\begin{aligned}
\mathbb{E}\left[2\langle \bm{b}_{3},-\bm{c}_{3}\rangle|\mathcal{M}_{k}\right] \le \mathbb{E}\left[\|\bm{b}_{3}\|^{2}|\mathcal{M}_{k}\right] + \mathbb{E}\left[\|\bm{c}_{3}\|^{2}|\mathcal{M}_{k}\right].
\end{aligned}
\end{small}
\end{equation}

%对于1中最后三项，可由引理3得到
For the last three terms in (\ref{eq:x_x_k_initial}), we use Lemma \ref{lemma:xi} to obtain
\begin{equation}
\begin{small}
\begin{aligned}
\mathbb{E}\left[2\langle \bm{a}_{3},\bm{d}_{3}\rangle|\mathcal{M}_{k}\right]
&= 2\mathbb{E}\left[\left\langle \frac{\alpha_{k+1}}{\alpha_{k}}W\mathbf{x}(k),-\frac{\alpha_{k+1}}{\alpha_{k}}W\bm{\xi}(k)\right\rangle|\mathcal{M}_{k}\right] \\
&\le 4np(3+L\alpha_{k})\nu_{k}^{2},
\end{aligned}
\end{small}
\end{equation}

\begin{equation}
\begin{small}
\begin{aligned}
&\quad \mathbb{E}\left[2\langle \bm{b}_{3},\bm{d}_{3}\rangle|\mathcal{M}_{k}\right] \\
&= 2\mathbb{E}\left[\left\langle (\tilde{W}(k)-I)\mathbf{x}(k),-\frac{\alpha_{k+1}}{\alpha_{k}}W\bm{\xi}(k)\right\rangle|\mathcal{M}_{k}\right] \\
&\le 8np(3+L\alpha_{k})\nu_{k}^{2},
\end{aligned}
\end{small}
\end{equation}

\begin{equation}\label{ieq:third_cd}
\begin{small}
\begin{aligned}
&\quad \mathbb{E}\left[2\langle -\bm{c}_{3},\bm{d}_{3}\rangle|\mathcal{M}_{k}\right] \\
&= 2\mathbb{E}[\langle-\alpha_{k+1}\nabla F(\mathbf{z}(k+1)), \\
&\qquad \left((1+\alpha_{k+1}\beta)W - \alpha_{k+1}\beta I\right)\bm{\xi}(k+1)\rangle|\mathcal{M}_{k}] \\
&\quad + 2\mathbb{E}\left[\left\langle\alpha_{k+1}\nabla F(\mathbf{z}(k+1)), \frac{\alpha_{k+1}}{\alpha_{k}}W\bm{\xi}(k)\right\rangle|\mathcal{M}_{k}\right] \\
&\quad + 2\mathbb{E}\left[\left\langle\alpha_{k+1}\nabla F(\mathbf{z}(k)), -\frac{\alpha_{k+1}}{\alpha_{k}}W\bm{\xi}(k)\right\rangle|\mathcal{M}_{k}\right] \\
&\le 12npL\alpha_{k+1}\nu_{k+1}^{2} + 2npL(4+np+2L\alpha_{k})\alpha_{k+1}\nu_{k}^{2}.
\end{aligned}
\end{small}
\end{equation}

To derive the final result, we also need to bound $\mathbb{E}\left[ \left\|\frac{1}{n}\mathbf{1}^{T} \nabla F(\mathbf{z}(k)) \right\|^{2} |\mathcal{M}_{k}\right]$. From (\ref{equation_nabla_F}), we have
%现在只需要bound住a即可。由1以及引理1，有
\begin{equation}\label{ieq:nabla_F_k}
\begin{footnotesize}
\begin{aligned}
&\quad \mathbb{E}\Big[\Big\|\frac{1}{n}\mathbf{1}^{T}\nabla F(\mathbf{z}(k))\Big\|^{2}|\mathcal{M}_{k}\Big] \\
%&= \mathbb{E}\Big[\Big\|\frac{1}{n}\mathbf{1}^{T}\nabla F(\mathbf{z}(k)) - \frac{1}{n}\mathbf{1}^{T}\nabla F(\mathbf{z}(k+1)) + \frac{1}{n}\mathbf{1}^{T}\nabla F(\mathbf{z}(k+1))\Big\|^{2}|\mathcal{M}_{k}\Big] \\
&\le 2\mathbb{E}\Big[\Big\|\frac{1}{n}\mathbf{1}^{T}\nabla F(\mathbf{z}(k))-\frac{1}{n}\mathbf{1}^{T}\nabla F(\mathbf{z}(k+1))\Big\|^{2}|\mathcal{M}_{k}\Big] \\
&\quad + 2\mathbb{E}\Big[ \Big\|\frac{1}{n}\mathbf{1}^{T}\nabla F(\mathbf{z}(k+1))\Big\|^{2}|\mathcal{M}_{k}\Big] \\
&\le \frac{2L^{2}}{n}\mathbb{E}\left[\|\mathbf{z}(k+1) - \mathbf{z}(k)\|^{2}|\mathcal{M}_{k}\right] + 4L^{2}\mathbb{E}\left[\| \bar{\bm{x}}(k) - \bm{x}^{*} \|^{2}|\mathcal{M}_{k}\right]\\
&\quad + \frac{4L^{2}}{n} \mathbb{E}\left[ \| \mathbf{z}(k+1) - \mathbf{1}\bar{\bm{x}}(k) \|^{2}|\mathcal{M}_{k}\right] \\
&\le \frac{2L^{2}}{n}\mathbb{E}\left[\|\mathbf{x}(k)-\mathbf{x}(k-1)\|^{2}|\mathcal{M}_{k}\right] + 4L^{2}\mathbb{E}\left[\|\bar{\bm{x}}(k) - \bm{x}^{*}\|^{2}|\mathcal{M}_{k}\right] \\
&\quad + \frac{4L^{2}}{n}\mathbb{E}\left[\|\mathbf{x}(k) - \mathbf{1}\bar{\bm{x}}(k)\|^{2}|\mathcal{M}_{k}\right] + 12pL^{2}\nu_{k+1}^{2} \\
&\quad + 4pL^{2}(7+2L\alpha_{k})\nu_{k}^{2}, 
%8p(5+L\alpha_{k})n^{2}L^{2}\nu_{k}^{2}
\end{aligned}
\end{footnotesize}
\end{equation}
where the third inequality follows from parts \textit{(f)} and \textit{(g)} of Lemma \ref{lemma:xi}.

Finally, by combining (\ref{eq:x_x_k_initial}) and (\ref{ieq:third_a})-(\ref{ieq:nabla_F_k}), we obtain
\begin{equation}\label{ieq:x_x_k_condition_e}
\begin{footnotesize}
\begin{aligned}
&\quad \mathbb{E}\left[\|\mathbf{x}(k+1) - \mathbf{x}(k)\|^{2}|\mathcal{M}_{k}\right] \\
&\le \frac{3+\sigma^{2}}{4} \mathbb{E}\left[\|\mathbf{x}(k) - \mathbf{x}(k-1)\|^{2}|\mathcal{M}_{k}\right] + 6np(3+2L\alpha_{k+1})\nu_{k+1}^{2} \\
&\quad + 4np(19+6L\alpha_{k}+5L\alpha_{k+1}+npL\alpha_{k+1}+2L^{2}\alpha_{k}\alpha_{k+1})\nu_{k}^{2} \\
%+ 2np(47 + 32L\alpha_{k})\nu_{k}^{2}
&\quad + \frac{3+\sigma^{2}}{1-\sigma^{2}}\Big( \mathbb{E}\left[\| \bm{b}_{3} \|^{2}|\mathcal{M}_{k}\right] + \mathbb{E}\left[ \| \bm{c}_{3} \|^{2} |\mathcal{M}_{k}\right] \\
&\qquad + n\alpha_{k+1}^{2}\mathbb{E}\Big[ \Big\|\frac{1}{n}\mathbf{1}^{T} \nabla F(\mathbf{z}(k)) \Big\|^{2} |\mathcal{M}_{k}\Big] \Big) \\
&\le \Big( \frac{3+\sigma^{2}}{4} + \frac{3(3+\sigma^{2})L^{2}\alpha_{k+1}^{2}}{1-\sigma^{2}} \Big) \mathbb{E}\left[\|\mathbf{x}(k) - \mathbf{x}(k-1)\|^{2}|\mathcal{M}_{k}\right] \\
&\quad + \frac{4n(3+\sigma^{2})L^{2}\alpha_{k+1}^{2}}{1-\sigma^{2}} \mathbb{E}\left[\|\bar{\bm{x}}(k) - \bm{x}^{*}\|^{2}|\mathcal{M}_{k}\right] \\
&\quad + \frac{(3+\sigma^{2})(\|W-I\|^{2} + 4L^{2}\alpha_{k+1}^{2})}{1-\sigma^{2}} \mathbb{E}\left[\|\mathbf{x}(k) - \mathbf{1}\bar{\bm{x}}(k)\|^{2}|\mathcal{M}_{k}\right] \\
&\quad + 2np\Big(9+6L\alpha_{k+1}+\frac{28L^{2}\alpha_{k+1}^{2}}{1-\sigma^{2}}\Big)\nu_{k+1}^{2} \\
&\quad + 4np\Big(19+6L\alpha_{k}+5L\alpha_{k+1}+npL\alpha_{k+1}+2L^{2}\alpha_{k}\alpha_{k+1}\\
&\qquad+\frac{6L^{2}(7+2L\alpha_{k})\alpha_{k+1}^{2}}{1-\sigma^{2}}\Big)\nu_{k}^{2}.
%&\quad + 2np\left(47 + 32L\alpha_{k} + \frac{8(14+3L\alpha_{k})L^{2}\alpha_{k}^{2}}{1-\sigma^{2}}\right)\nu_{k}^{2}.
\end{aligned}
\end{footnotesize}
\end{equation}
Taking the full expectation on both sides of the above inequality completes the proof.
\end{proof}

\section*{References}
\bibliographystyle{IEEEtran}%参考文献
\normalem
\bibliography{IEEEabrv,2401_bibfile}

% Generated by IEEEtran.bst, version: 1.14 (2015/08/26)
\begin{thebibliography}{10}
\providecommand{\url}[1]{#1}
\csname url@samestyle\endcsname
\providecommand{\newblock}{\relax}
\providecommand{\bibinfo}[2]{#2}
\providecommand{\BIBentrySTDinterwordspacing}{\spaceskip=0pt\relax}
\providecommand{\BIBentryALTinterwordstretchfactor}{4}
\providecommand{\BIBentryALTinterwordspacing}{\spaceskip=\fontdimen2\font plus
\BIBentryALTinterwordstretchfactor\fontdimen3\font minus
  \fontdimen4\font\relax}
\providecommand{\BIBforeignlanguage}[2]{{%
\expandafter\ifx\csname l@#1\endcsname\relax
\typeout{** WARNING: IEEEtran.bst: No hyphenation pattern has been}%
\typeout{** loaded for the language `#1'. Using the pattern for}%
\typeout{** the default language instead.}%
\else
\language=\csname l@#1\endcsname
\fi
#2}}
\providecommand{\BIBdecl}{\relax}
\BIBdecl

\bibitem{wang2024privacy}
Y.~Wang, ``Privacy in multi-agent systems,'' \emph{arXiv preprint arXiv:
  2403.02631}, 2024.

\bibitem{zhu2019deep}
L.~Zhu, Z.~Liu, and S.~Han, ``Deep leakage from gradients,'' in
  \emph{Proceedings of the 33rd International Conference on Neural Information
  Processing Systems}, 2019, {A}rt. no. 1323.

\bibitem{mo2016privacy}
Y.~Mo and R.~M. Murray, ``Privacy preserving average consensus,'' \emph{IEEE
  Transactions on Automatic Control}, vol.~62, no.~2, pp. 753--765, 2016.

\bibitem{gade2018private}
S.~Gade and N.~H. Vaidya, ``Private optimization on networks,'' in
  \emph{Proceedings of the American Control Conference}, 2018, pp. 1402--1409.

\bibitem{Han2022PrivacyPreserving}
D.~Han, K.~Liu, H.~Sandberg, S.~Chai, and Y.~Xia, ``Privacy-preserving dual
  averaging with arbitrary initial conditions for distributed optimization,''
  \emph{IEEE Transactions on Automatic Control}, vol.~67, no.~6, pp.
  3172--3179, 2022.

\bibitem{Zhang2025Privacy}
K.~Zhang, X.~Liao, and H.~Li, ``Privacy protection of dual averaging push for
  decentralized optimization via zero-sum structured perturbations,''
  \emph{IEEE Transactions on Dependable and Secure Computing, Early Access},
  2025, doi: 10.1109/TDSC.2025.3620745.

\bibitem{li2020privacypreserving}
Q.~Li, R.~Heusdens, and M.~G. Christensen, ``Privacy-preserving distributed
  optimization via subspace perturbation: A general framework,'' \emph{IEEE
  Transactions on Signal Processing}, vol.~68, pp. 5983--5996, 2020.

\bibitem{huan2023dynamics}
H.~Gao, Y.~Wang, and A.~Nedić, ``Dynamics based privacy preservation in
  decentralized optimization,'' \emph{Automatica}, vol. 151, {A}rt. no. 110878,
  2023.

\bibitem{cheng2024privacy}
H.~Cheng, X.~Liao, H.~Li, and Y.~Zhao, ``Privacy-preserving push-pull method
  for decentralized optimization via state decomposition,'' \emph{IEEE
  Transactions on Signal and Information Processing over Networks}, vol.~10,
  pp. 513--526, 2024.

\bibitem{Lu2025Sdppaad}
Q.~Lü, C.~He, K.~Zhang, H.~Li, and T.~Huang, ``{SD-PPDDA}: A privacy efficient
  decentralized dual averaging algorithm over networks,'' \emph{IEEE
  Transactions on Machine Learning in Communications and Networking}, vol.~3,
  pp. 1197--1209, 2025.

\bibitem{zhang2019admm}
C.~Zhang, M.~Ahmad, and Y.~Wang, ``{ADMM} based privacy-preserving
  decentralized optimization,'' \emph{IEEE Transactions on Information
  Forensics and Security}, vol.~14, no.~3, pp. 565--580, 2019.

\bibitem{zhang2019enabling}
C.~Zhang and Y.~Wang, ``Enabling privacy-preservation in decentralized
  optimization,'' \emph{IEEE Transactions on Control of Network Systems},
  vol.~6, no.~2, pp. 679--689, 2019.

\bibitem{Mia2025QuanCrypt}
M.~J. Mia and M.~H. Amini, ``{QuanCrypt-FL}: Quantized homomorphic encryption
  with pruning for secure federated learning,'' \emph{IEEE Transactions on
  Artificial Intelligence, Early Access}, 2025, doi: 10.1109/TAI.2024.3474567.

\bibitem{han2024Adaptive}
J.~Han and L.~Yan, ``Adaptive batch homomorphic encryption for joint federated
  learning in cross-device scenarios,'' \emph{IEEE Internet of Things Journal},
  vol.~11, no.~6, pp. 9338--9354, 2024.

\bibitem{Liu2024CryptographyBasedPM}
B.~Liu, F.~Xie, and L.~Chai, ``Cryptography-based privacy-preserving method for
  distributed optimization over time-varying directed graphs with enhanced
  efficiency,'' \emph{IEEE Transactions on Automatic Control}, vol.~70, no.~9,
  pp. 5808--5822, 2025.

\bibitem{Yu2025Lightweight}
B.~Yu, J.~Zhao, K.~Zhang, J.~Gong, and H.~Qian, ``Lightweight and dynamic
  privacy-preserving federated learning via functional encryption,'' \emph{IEEE
  Transactions on Information Forensics and Security}, vol.~20, pp. 2496--2508,
  2025.

\bibitem{han2017Differentially}
S.~Han, U.~Topcu, and G.~J. Pappas, ``Differentially private distributed
  constrained optimization,'' \emph{IEEE Transactions on Automatic Control},
  vol.~62, no.~1, pp. 50--64, 2017.

\bibitem{huang2015differentially}
Z.~Huang, S.~Mitra, and N.~Vaidya, ``Differentially private distributed
  optimization,'' in \emph{Proceedings of the 16th International Conference on
  Distributed Computing and Networking}, 2015, {A}rt. no. 4.

\bibitem{Qu2018Harnessing}
G.~Qu and N.~Li, ``Harnessing smoothness to accelerate distributed
  optimization,'' \emph{IEEE Transactions on Control of Network Systems},
  vol.~5, no.~3, pp. 1245--1260, 2018.

\bibitem{ding2022differentially}
T.~Ding, S.~Zhu, J.~He, C.~Chen, and X.~Guan, ``Differentially private
  distributed optimization via state and direction perturbation in multiagent
  systems,'' \emph{IEEE Transactions on Automatic Control}, vol.~67, no.~2, pp.
  722--737, 2022.

\bibitem{Yang2025Differentially}
Z.~Yang, W.~He, and S.~Yang, ``Differentially private distributed optimization
  over time-varying unbalanced networks with linear convergence rates,''
  \emph{IEEE Transactions on Signal Processing}, vol.~73, pp. 1138--1152, 2025.

\bibitem{wang2024tailoring}
Y.~Wang and A.~Nedić, ``Tailoring gradient methods for differentially private
  distributed optimization,'' \emph{IEEE Transactions on Automatic Control},
  vol.~69, no.~2, pp. 872--887, 2024.

\bibitem{yu2023gradienttracking}
Y.~Xuan and Y.~Wang, ``Gradient-tracking based differentially private
  distributed optimization with enhanced optimization accuracy,''
  \emph{Automatica}, vol. 155, {A}rt. no. 111150, 2023.

\bibitem{huang2024differential}
L.~Huang, J.~Wu, D.~Shi, S.~Dey, and L.~Shi, ``Differential privacy in
  distributed optimization with gradient tracking,'' \emph{IEEE Transactions on
  Automatic Control}, vol.~69, no.~9, pp. 5727--5742, 2024.

\bibitem{Wang2023GradientTrackingBased}
Y.~Wang and T.~Başar, ``Gradient-tracking-based distributed optimization with
  guaranteed optimality under noisy information sharing,'' \emph{IEEE
  Transactions on Automatic Control}, vol.~68, no.~8, pp. 4796--4811, 2023.

\bibitem{Zhao2025VarianceReduced}
S.~Zhao, S.~Song, and Y.~Liu, ``A variance-reduced aggregation based gradient
  tracking method for distributed optimization over directed networks,''
  \emph{IEEE Transactions on Automatic Control}, vol.~70, no.~6, pp.
  4109--4115, 2025.

\bibitem{Fu2022AdapDPFL}
J.~Fu, Z.~Chen, and X.~Han, ``{Adap DP-FL}: Differentially private federated
  learning with adaptive noise,'' in \emph{Proceedings of the 2022 IEEE
  International Conference on Trust, Security and Privacy in Computing and
  Communications}, 2022, pp. 656--663.

\bibitem{Xue2024DifferentiallyPrivate}
R.~Xue, K.~Xue, B.~Zhu, X.~Luo, T.~Zhang, Q.~Sun, and J.~Lu, ``Differentially
  private federated learning with an adaptive noise mechanism,'' \emph{IEEE
  Transactions on Information Forensics and Security}, vol.~19, pp. 74--87,
  2024.

\bibitem{girgis2021shuffled}
A.~Girgis, D.~Data, S.~Diggavi, P.~Kairouz, and A.~T. Suresh, ``Shuffled model
  of differential privacy in federated learning,'' in \emph{Proceedings of the
  24th International Conference on Artificial Intelligence and Statistics},
  2021, pp. 2521--2529.

\bibitem{liu2021flame}
R.~Liu, Y.~Cao, H.~Chen, R.~Guo, and M.~Yoshikawa, ``Flame: Differentially
  private federated learning in the shuffle model,'' in \emph{Proceedings of
  the AAAI Conference on Artificial Intelligence}, 2021, pp. 8688--8696.

\bibitem{Liu2023Echo}
Y.~Liu, S.~Zhao, L.~Xiong, Y.~Liu, and H.~Chen, ``Echo of neighbors: privacy
  amplification for personalized private federated learning with shuffle
  model,'' in \emph{Proceedings of the AAAI Conference on Artificial
  Intelligence}, 2023, pp. 11\,865--11\,872.

\bibitem{Jakovetic2014fast}
D.~Jakovetić, J.~Xavier, and J.~M.~F. Moura, ``Fast distributed gradient
  methods,'' \emph{IEEE Transactions on Automatic Control}, vol.~59, no.~5, pp.
  1131--1146, 2014.

\bibitem{pu2020push}
S.~Pu, W.~Shi, J.~Xu, and A.~Nedi{\'c}, ``{Push-Pull} gradient methods for
  distributed optimization in networks,'' \emph{IEEE Transactions on Automatic
  Control}, vol.~66, no.~1, pp. 1--16, 2020.

\bibitem{pu2021distributed}
S.~Pu and A.~Nedi{\'c}, ``Distributed stochastic gradient tracking methods,''
  \emph{Mathematical Programming}, vol. 187, pp. 409--457, 2021.

\bibitem{horn2012matrix}
R.~A. Horn and C.~R. Johnson, \emph{Matrix Analysis}.\hskip 1em plus 0.5em
  minus 0.4em\relax Cambridge University Press, 2012.

\bibitem{li2021privacypreserving}
Q.~Li, J.~S. Gundersen, R.~Heusdens, and M.~G. Christensen,
  ``Privacy-preserving distributed processing: Metrics, bounds and
  algorithms,'' \emph{IEEE Transactions on Information Forensics and Security},
  vol.~16, pp. 2090--2103, 2021.

\bibitem{goldreich2009foundations}
O.~Goldreich, \emph{Foundations of Cryptography: Volume 2, Basic
  Applications}.\hskip 1em plus 0.5em minus 0.4em\relax Cambridge University
  Press, 2009.

\bibitem{NPEET}
G.~V. Steeg, ``Non-parametric entropy estimation toolbox ({NPEET}),''
  \url{https://github.com/gregversteeg/NPEET}, 2022.

\end{thebibliography}
\begin{IEEEbiography}[{\includegraphics[width=1in,height=1.25in,clip,keepaspectratio]{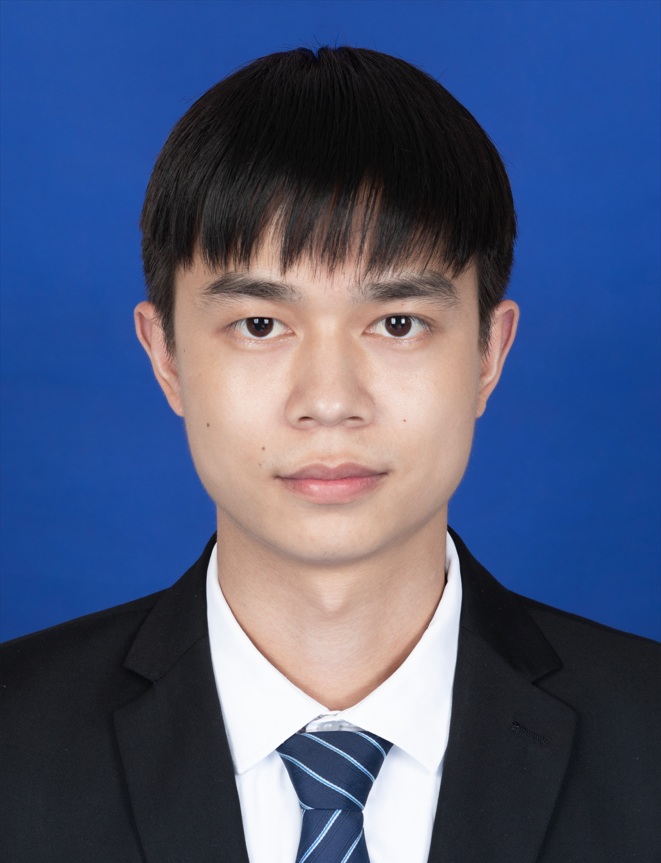}}]{Furan Xie} received the B.Eng. degree in automation from Wuhan University of Science and Technology, Wuhan, China, in 2021, where he is currently working toward the Ph.D. degree in control science and engineering with the School of Artificial Intelligence and Automation. His research interests include distributed optimization, privacy preservation, and power systems.
\end{IEEEbiography}
\begin{IEEEbiography}[{\includegraphics[width=1in,height=1.25in,clip,keepaspectratio]{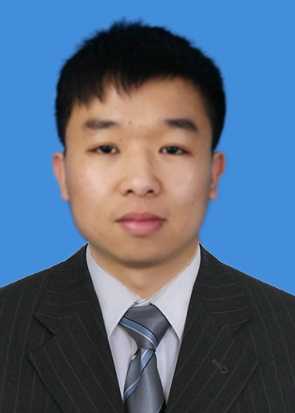}}]{Bing Liu} received the Ph.D degree in Control Science and Engineering from Huazhong University of Science and Technology, Wuhan, China, in 2014. From September 2014 to December 2016, he was a Postdoctoral Fellow in the Department of Electrical and Computer Engineering, University of Windsor, Windsor, ON, Canada. 

He is currently an Associate Professor with the School of Artificial Intelligence and Automation, Wuhan University of Science and Technology, Wuhan. His research interests include distributed optimization and power systems.
\end{IEEEbiography}
\begin{IEEEbiography}[{\includegraphics[width=1in,height=1.25in,clip,keepaspectratio]{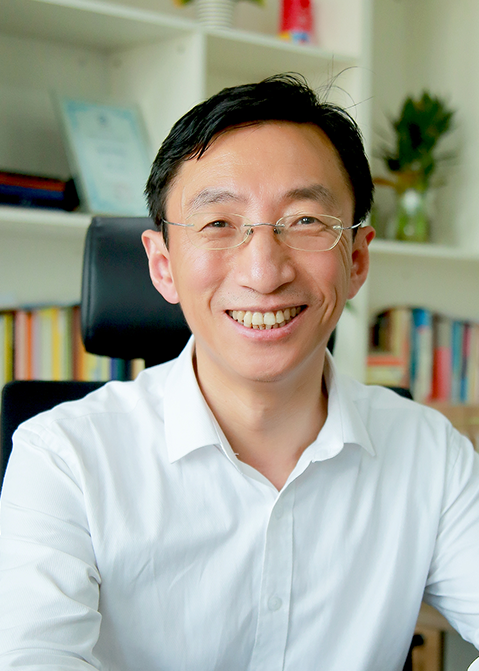}}]{Li Chai}(S’00-M’03) received the B.S. degree in applied mathematics and the M.S. degree in control science and engineering from Zhejiang University, China, in 1994 and 1997, respectively, and the Ph.D degree in electrical engineering from the Hong Kong University of Science and Technology, Hong Kong, in 2002. From August 2002 to December 2007, he was at Hangzhou Dianzi University, China. He worked as a professor at Wuhan University of Science and Technology, China, from 2008 to 2022. In August 2022, he joined Zhejiang University, China, where he is currently a professor at the College of Control Science and Engineering. He has been a postdoctoral researcher or visiting scholar at Monash University, Newcastle University, Australia and Harvard University, USA. 

His research interests include distributed optimization, filter banks, graph signal processing, and networked control systems. Professor Chai is the recipient of the Distinguished Young Scholar of the National Science Foundation of China. He has published over 100 fully refereed papers in prestigious journals and leading conferences. He serves as the Associate Editor of IEEE Transactions on Circuit and Systems II: Express Briefs, Control and Decision and Journal of Image and Graphs.
\end{IEEEbiography}
\end{document}